\pdfoutput=1
\documentclass[a4paper]{amsart}
\usepackage{knottedStyle,knottedMacros}

\title{Knotted families from graspers}

\author{Danica Kosanovi\'c}
\address{ETH Z\"urich, Department of Mathematics, R\"amistrasse 101, 8092 Z\"urich, Switzerland}
\email{danica.kosanovic@math.ethz.ch}

\date{March 14, 2024}

\begin{document}

\begin{abstract}
    For any smooth manifold $M$ of dimension $d\geq4$ we construct explicit classes in homotopy groups of spaces of embeddings of either an arc or a circle into $M$, in every degree that is a multiple of $d-3$, and show that they are detected in the Taylor tower of Goodwillie and Weiss. The classes are obtained from families of string links constructed in the $d$-ball.
\end{abstract}

\maketitle

% primary: 57K45, 57R40
% secondary: 55Q15, 57K12

% \begin{spacing}{0.5}
%     \tableofcontents
% \end{spacing}

%~~~~~~~~~~~~~~~~~~~~~~~~~~~~~~~~~~~~~~
\section{Introduction}\label{sec:intro}

Given smooth manifolds $C$ and $M$ let $K\colon C\hra M$ be a neat smooth embedding. This means that $K$ is a smooth injective map with injective derivative at each point, $K(C)$ intersects $\partial M$ transversely and precisely in $K(\partial C)$, and we assume that the boundary condition $K|_{\partial C}$ is fixed. 

Let $\D^k$ denote the unit ball in $\R^k$. The trefoil knot $T\colon\D^1\hra\D^3$ can be obtained from the standard embedding $\u\colon\D^1\hra\D^3$, called the unknot, by the following operation. Consider the meridians $m_1,m_2\colon\S^1\hra\D^3\sm\u(\D^1)$ to $\u$ at two different points, then replace a subarc of $\u$ by the arc describing the commutator $[m_1,m_2]\coloneqq m_1m_2m_1^{-1}m_2^{-1}$, as on the left of Figure~\ref{fig:Borr}. Equivalently, pick three subarcs of $\u$ and connect sum them into the Borromean rings, see the right hand side of Figure~\ref{fig:Borr}. This operation is called the Gusarov--Habiro \emph{clasper surgery} on $\u$ of degree $n=2$ \cite{Gusarov-main,Gusarov,Habiro}. In general, a surgery of degree $n\geq1$ uses $n-1$ copies of the Borromean rings and $n+1$ subarcs of $\u$, and follows a shape of a planar rooted binary tree with $n$ leaves. For $n=1$ clasper surgery is simply a crossing change: ambiently connect sum a subarc of $\u$ into a meridian of $\u$.

\begin{figure}[!htbp]
    \centering
    \begin{tikzpicture}[scale=0.5,even odd rule,line width=0.2mm]
    \clip (-1.2,-1.2) rectangle (14.5,6.4);
    \draw[thick]
    %unutarnja druga
                        (4.6,0)
                        to[out=90,in=100, distance=3.3cm] (8-.15,-0.2)
                        to[out=-80,in=-80, distance=1.1cm] (8+.2,0.4)
                        to[out=100,in=90, distance=5.1cm]
                        (2.6,0)
    %spoljna druga
                        (2.2,0)
                        to[out=90,in=100, distance=5.5cm] (8+.45,0.4)
                        to[out=-80,in=-80, distance=1.4cm] (8-.4,-0.2)
                        to[out=100,in=90, distance=2.7cm] (5,0.4) ;
    \fill[white,draw=black,thick]
    %unutarnja prva
                        (3,0)
                        to[out=90,in=100, distance=5.6cm] (12-.15,-0.2)
                        to[out=-80,in=-80, distance=1.1cm] (12+.2,0.4)
                        to[out=100,in=90, distance=7.6cm]
                        (1.4,0.2)
    %spoljna prva
                    -- (1,0.3) to[out=90,in=100, distance=8cm] (12+.45,0.4)
                        to[out=-80,in=-80, distance=1.4cm] (12-.4,-0.2)
                        to[out=100,in=90, distance=5.2cm]
                        (3.4,0) -- cycle ;
    % for crossings
    \fill[white]
            (1,-0.5) rectangle (5,0.5)
            (8.2,-0.2) rectangle (8.6,0.2)
            (12.2,-0.2) rectangle (12.6,0.2) ;
    % missing parts
    \draw[thick]
            (1,0) -- (1,0.5) (5,0)-- (5,0.5)
            (1.4,0.5) -- (2.2,0.5) (2.6,0.5) -- (3,0.5) (3.4,0.5)--(4.6,0.5) ;
    \draw[thick]
            (-1,0) -- (1,0) (5,0) -- (7.3,0) (8,0)  -- (11.3,0)   (12,0) -- (14,0) ;

    \draw[red,thick] 
        (8.1,0.1) to[out=100,in=100, distance=0.5cm] (7.9,0) to[out=-80,in=-85, distance=0.5cm] (8.1,-0.1)
        (8.9,-0.5) node{\tiny$m_1$} ;
    \draw[blue,thick] 
        (12.1,0.1) to[out=100,in=100, distance=0.5cm] (11.9,0) to[out=-80,in=-85, distance=0.5cm] (12.1,-0.1) 
        (12.9,-0.5) node{\tiny$m_2$} ;

\end{tikzpicture}
\qquad\qquad
\begin{tikzpicture}[scale=0.6]
\clip (-6,-5) rectangle (4.2,4);

\draw[thick] (-5.5,-4) -- (4,-4);

\begin{knot}[
%  draft mode=crossings,
  flip crossing/.list={3,4},
  clip width=7
]
    \strand[black,thick] (1,0) circle[radius=2cm];
    \strand[black,thick] (-1,0) circle[radius=2cm];
    \strand[,thick] (0,{sqrt(3)}) circle[radius=2cm];
\end{knot}

\fill[white]    (-1.28,-1.9) rectangle (-1.12,-2.1) 
                (-1.28,-4.1) rectangle (-1.12,-3.9);
\fill[white]    (1.28,-1.9) rectangle (1.12,-2.1)
                (1.28,-4.1) rectangle (1.12,-3.9);
\fill[white]    (-1.25,3.5) rectangle (-0.95,3.3)
                (-4.18,-4.1) rectangle (-4.02,-3.9);

\draw[thick,black] (-1.3,-4) -- (-1.3,-2) (-1.1,-4) -- (-1.1,-2);
\draw[thick,black] (1.1,-4) -- (1.1,-2) (1.3,-4) -- (1.3,-2);
\draw[thick,]   (-4.2,-4) -- (-4.2,-2) 
                to[out=90,in=180,distance=3cm] (-0.95,3.5)
                (-4,-4) -- (-4,-2) to[out=90,in=180,distance=2.6cm] (-1.22,3.3);

\end{tikzpicture}
    \caption{Two isotopic knots, obtained by degree $2$ clasper surgery on the unknot $\u$.}
    \label{fig:Borr}
\end{figure}
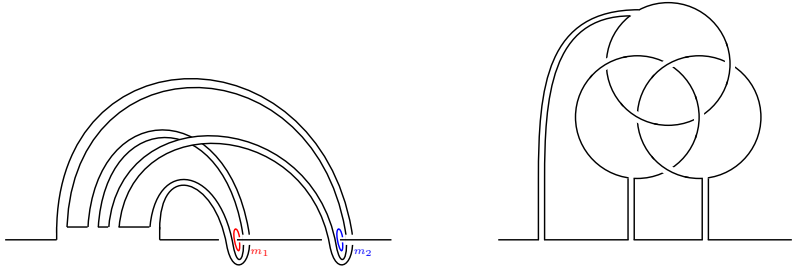
In other dimensions certain analogues of clasper surgery have been considered. For example, for any $k\geq2$ one can start with the standard embedding $\u\colon\D^{4k-1}\hra\D^{6k}$, pick two meridians $m_1,m_2\colon\S^{2k}\hra\D^{6k}\sm\u(\D^{4k-1})$, defined as the boundaries of the normal disks at two points of $\u(\D^{4k-1})$, and consider their Whitehead product $[m_1,m_2]\colon\S^{4k-1}\to\S^{2k}\vee\S^{2k}\to\D^{6k}\sm\u(\D^{4k-1})$. This is defined by postcomposing with $m_1\vee m_2$ the universal Whitehead product $\S^{4k-1}\to\S^{2k}\vee\S^{2k}$, which is the attaching map of the top $4k$-cell in $\S^{2k}\times\S^{2k}$. It turns out that $[m_1,m_2]$ is homotopic to an embedding (see~\eqref{eq:Haefliger-foliate} and Section~\ref{subsec:Haefliger}):
\begin{equation}\label{eq:Haefliger}
    [m_1,m_2]_{emb}\colon\S^{4k-1}\hra\D^{6k}\sm\u(\D^{4k-1}).
\end{equation}
Thus, we can tube $\u$ into the sphere $[m_1,m_2]_{emb}$ (i.e.\ ambiently connect sum them), to obtain a new embedding $T\colon\D^{4k-1}\hra\D^{6k}$. This is precisely the trefoil of Haefliger~\cite{Haefliger}, which generates the group of isotopy classes $\pi_0\Embp(\D^{4k-1},\D^{6k})\cong\Z$, as he showed in \cite{Haefliger2} (note that since codimension is $2k+1>2$ by Zeeman~\cite{Zeeman} this group is trivial in the $\mathsf{PL}$ and $\mathsf{Top}$ categories!). Other analogues of clasper surgery appear in Koschorke's work in the context of link maps, see for example~\cite{Koschorke-arbitrary}, and more recently, in the construction of Watanabe~\cite{Watanabe} of nontrivial classes in homotopy groups of the space of diffeomorphisms of a disk of any dimension. See Section~\ref{subsec:related-work}.

%~~~~~~~~~~
\subsection{The main result}
In this paper we generalise clasper surgery in another way: for embeddings of a $1$-dimensional source $C$ into a smooth manifold $M$ of dimension $d\geq4$, with a fixed boundary condition. However, instead of studying as above the set of isotopy classes of embeddings, which is the set $\pi_0\Embp(C,M)$ of path components, we study \emph{higher homotopy groups} $\pi_k\Embp(C,M)$, $k\geq1$, where $\u\colon C\hra M$ is an arbitrary basepoint omitted from the notation. Thus, a nontrivial ``knot'' is now a nonzero homotopy class, represented by a $k$-parameter family of embeddings of $C$ into $M$ which cannot be trivialised through such families. In other words, all embeddings are isotopic to $\u$, but the family cannot be homotoped to the family $\const_\u$ constantly equal to $\u$.

In this setting a meridian of $\u$ (the boundary of a normal disk) is a sphere of dimension $d-2$, so the Whitehead product of $n$ meridians is a sphere of dimension $n(d-3)+1$ (we will actually use the equivalent Samelson product, see Section~\ref{subsec:Samelson}). We can foliate this sphere by embedded circles, and then ambiently connect sum a fixed subarc of $\u$ into each of these circles, to obtain an $n(d-3)$-parameter family of embeddings of $C$, so an element of $\pi_{n(d-3)}\Embp(C,M)$. We make this idea precise in the form of \emph{grasper surgery}, which generalises clasper surgery to all embedding spaces with $1$-dimensional source.

In particular, for grasper surgery of degree $n=1$ we connect sum $\u$ with the circles foliating its meridian. Note that we have to pick an arc guiding a subarc of $\u$ into the vicinity of the meridian, and this choice is up to isotopy determined by a fundamental group element $g\in\pi_1M$, since $d\geq4$. There is a crossing change homotopy from the resulting family back to $\const_\u$ through immersions, by foliating the meridian disk. This defines a realisation map into the relative homotopy group:
\begin{equation}\label{eq:realmap1}
    \realmap_1\colon \Z[\pi_1M]\ra
    \pi_{d-2}\big(\Imm_\partial(C,M),\Embp(C,M)\big).
\end{equation}
This map is an isomorphism for $C$ connected (that is, $C=\D^1$ or $C=\S^1$), as we elaborate in the next paragraph. By general position such relative homotopy groups vanish in degrees $k<d-2$, implying that $\pi_k\Embp(C,M)\cong \pi_k\Imm(C,M)$ for $k<d-3$. Thus, $\pi_{d-3}$ is the lowest homotopy group that can distinguish embeddings from immersions, and the difference between them (i.e.\ the kernel of the surjection $\pi_{d-3}\Embp(C,M)\sra \pi_{d-3}\Imm(C,M))$ is a quotient of $\Z[\pi_1M]$.

Indeed, the inverse of $\realmap_1$ is an explicit map $\Dax$ coming from the work of Dax \cite{Dax}, as shown in \cite{KT-highd}, and in \cite{Gabai-disks} for $d=4$, and in \cite{K-Dax} for $C=\S^1$. Equivalently -- and relevantly for this paper, $\Dax=\evrel_2$ is the map to the second layer of the Taylor tower constructed by Goodwillie and Weiss~\cite{Weiss}. This is a tower of spaces $\cdots\to T_{n+1}(C,M)\to T_n(C,M)\to\cdots$ with maps
\[
    \ev_n\colon \Embp(C,M)\ra T_n(C,M),
    \quad n\geq1,
\]
so that on relative homotopy groups one has the induced map
\[
    \evrel_{n+1}\colon\pi_k\big(T_n(C,M),\Embp(C,M)\big)\ra \pi_k\big(T_n(C,M),T_{n+1}(C,M)\big).
\]
The latter group is trivial for $1\leq k\leq n(d-3)$ by \cite{Weiss,GW}, and in \cite{K-thesis-paper} we defined 
\begin{equation}\label{eq:Lie}
    \pi_{n(d-3)+1}\big(T_n(C,M),T_{n+1}(C,M)\big)\overset{\cong}{\ra}\Lie_{\pi_1M}(n),
\end{equation}
an explicit isomorphism with the group $\Lie_{\pi_1M}(n)$, the quotient by the antisymmetry $(\AS)$ and Jacobi relations $(\IHX)$ of the free abelian group $\Z[\Tree_{\pi_1M}(n)]$ generated by the set $\Tree_{\pi_1M}(n)$ of planar rooted trivalent trees with $n$ leaves decorated by elements of $\pi_1M$ (see Section~\ref{subsec:trees}).

\begin{thm}\label{thm:main}
    Let $C$ be equal to $\D^1$ or $\S^1$ and $M$ be any compact smooth manifold with boundary, $\dim M=d\geq4$. For any $n\geq1$ there is an explicit homomorphism
    \[
        \realmap_n\colon\Lie_{\pi_1M}(n)\ra\pi_{n(d-3)+1}\big(T_n(C,M),\Embp(C,M)\big)
    \]
    of abelian groups, given by ``grasper surgery of degree $n$'', such that 
    \[
    \evrel_{n+1}\circ\realmap_n=\Id_{\Lie_{\pi_1M}(n)}.
    \]
\end{thm}
A \emph{grasper} of degree $n$ is simply an embedding $\TG_n\colon\ball^d\hra M$, which is disjoint from $\partial M$, and which intersects the basepoint knot $\u$ only in fixed intervals
\begin{equation}\label{eq:G-condition}
    \TG_n(\ball^d)\cap\u(C)=\u(J_i)=\TG_n(a_i)
\end{equation}
for fixed arcs $a_i\colon\D^1\hra\ball^d$ with $0\leq i\leq n$, see Figure~\ref{fig:grasper}. Grasper surgery is the family of embeddings obtained by replacing the subarc $\u(J_0)=\TG(a_0)$ with $\TG_n(\mu_{\Gamma,d}(\vec{\tau}))$, for $\vec{\tau}\in\S^{n(d-3)}$ and a fixed family $\mu_{\Gamma,d}\colon\S^{n(d-3)}\to\Embp(\D^1,\ball^d\sm\bigsqcup_{i=1}^na_i)$, see \eqref{eq-intro:mu} and Corollary~\ref{cor:string-links}. In introducing these notions we are inspired by \emph{gropes} of Conant and Teichner~\cite{CT1} and \emph{claspers} of Gusarov and Habiro \cite{Gusarov,Habiro}. Both are closely related to iterated commutators in the fundamental group of the complement of a knot (its lower central series). A grope contains a clasper as a subset, but also carries additional data. For graspers this data is $\mu^\star_{\Gamma,d}$ from~\eqref{eq-intro:mu-star}, and is used for the trivialisation in $\pi_{n(d-3)}T_n(C,M)$.

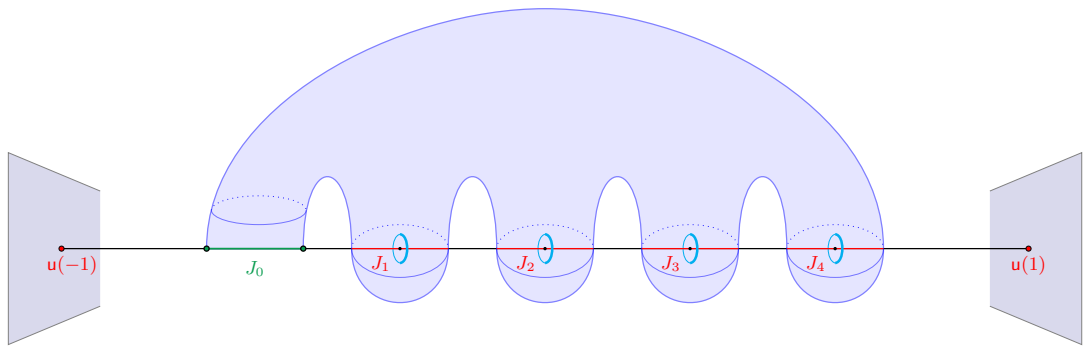
\begin{figure}[!htbp]
    \centering
    \begin{tikzpicture}[scale=0.8]
    \clip (-2.2,-2.1) rectangle (20.2,5.1);
    % boundary of M
    \fill[blue!50!black,opacity=0.15,draw=black,draw opacity=0.5]
        (18.2,-1.2) -- (20.1,-2) -- (20.1,2) --  (18.2,1.2) ;
    \fill[blue!50!black,opacity=0.15,draw=black,draw opacity=0.5]
        (-.2,1.2) -- (-2.1,2) -- (-2.1,-2) -- (-.2,-1.2) ;
    \draw[semithick]
        (-1,0) -- (19,0) ;

    \fill[color=blue!20!,opacity=0.5,draw=blue,semithick]
        (2,0) to[out=90,in=180,distance=3cm] 
        (9,5) to[out=0,in=90,distance=3cm] 
        (16,0)  to[out=-90,in=-90,distance=1.5cm]
        (14,0)  to[out=90,in=90,distance=2cm]
        (13,0) to[out=-90,in=-90,distance=1.5cm]
        (11,0)  to[out=90,in=90,distance=2cm]
        (10,0) to[out=-90,in=-90,distance=1.5cm]
        (8,0)  to[out=90,in=90,distance=2cm]
        (7,0) to[out=-90,in=-90,distance=1.5cm]
        (5,0)  to[out=90,in=90,distance=2cm]
        (4,0) --
        (2,0);
    % spheres
    \foreach \y/\i/\j in {6/0,9/1,12/2,15/3}{
        \draw[blue,opacity=0.5] (\y-1,0) arc (180:360:1cm and 0.6cm);
        \draw[blue,dotted] (\y-1,0) arc (180:0:1cm and 0.5cm);
        % \shade[ball color=blue!20!,opacity=0.3] (\y,0) circle (1cm);
        % \draw[blue,semithick] (\y,0) circle (1cm);
        % \draw (\y,-1) node[below=2pt]{$\mathbb{S}_{\i}$};
        }
     % vertical arcs
    \foreach \x in {6,9,12,15}{
        \draw[cyan,thin]
            (\x,0.3) to[out=-165,in=165, distance=0.2cm] (\x,-0.3);
    }
        
    \draw[blue,opacity=0.5] (2.1,0.8) arc (180:360:0.98cm and 0.3cm);
    \draw[blue,dotted] (2.1,0.8) arc (180:0:0.98cm and 0.3cm);
    
    \draw[mygreen, thick]
            (2,0) -- (4,0) %node[fill=mygreen!30,draw, pos=0.5]{}
            (3, 0) node[below=2pt]{\small$J_0$};
    \fill[mygreen,draw=black,semithick] 
        (2,0) circle (1.5pt) (4,0) circle (1.5pt);

    \fill[red,draw=black]
        (-1,0) circle (1.5pt) (-0.78,0) node[below]{\small$\mathsf{u}({-}1)$}
        (19,0) circle (1.5pt) node[below]{\small$\mathsf{u}(1)$};
    \draw[semithick,red]
        (5,0)--(7,0) node[pos=0.3,below]{\small$J_1$}
        (8,0)--(10,0) node[pos=0.3,below]{\small$J_2$}
        (11,0)--(13,0) node[pos=0.3,below]{\small$J_3$}
        (14,0)--(16,0) node[pos=0.3,below]{\small$J_4$};

             % vertical arcs
    \foreach \x in {6,9,12,15}{
        \draw[cyan,very thick]
            (\x,0.3) to[out=-5,in=5, distance=0.2cm] (\x,-0.3);
        \fill (\x,0) circle (1pt) %node[below]{\small$u_{\i}$}
        ;
    }
\end{tikzpicture}
    \caption{A grasper in a $d$-manifold $M$ (here $d=3$) relative to $\u$ (the horizontal line) is a $d$-ball that ``grasps'' $\u$ in fixed subintervals $J_i$. The family $\mu_{\Gamma,d}$ has a particular linking pattern with the arcs $a_i$, that under $\TG_n$ get identified with $\u(J_i)$.}
    \label{fig:grasper}
\end{figure}
Let us give a few comments, before turning to applications and related results.
\begin{rem}
    The essence of the theorem is the compatibility of a priori very different $n$-fold brackets: on one hand, the generators of \eqref{eq:Lie} come from $n$-fold Whitehead brackets of $(d-1)$-spheres (of configurations in $\Conf(M)$ in which one point goes around another one locally), whereas degree $n$ graspers involve certain ``embedded commutators'' of $(d-2)$-spheres (which are meridians of $\u$, see Figure~\ref{fig:grasper} for $d=3$ and Figure~\ref{fig:realmap-2} for $d=4$). Thus, the proof calls for a thorough geometric understanding of the evaluation map $\ev_{n+1}$, which we achieve using the reduced punctured knots model $\rpT_n(M)$ of $T_n(\D^1,M)$ developed in \cite{KT-rpn}. The two viewpoints fit together via a sort of suspension operation (see Proposition~\ref{prop:suspensions} and  Theorem~\ref{thm:analogues}).

   The results of this paper hold for $d=3$ with a slight modification, but that case was already studied in \cite{K-thesis-paper} (where the proof for $d\geq4$ was hinted in Remark~1.10). The proof here appears easier for several reasons. Firstly, we rely on the key Theorem~\ref{thm:thesis} from \cite{K-thesis-paper}, describing the homotopy type of the layers. Secondly, the reduced punctured knots model from \cite{KT-rpn} makes formulae shorter. Thirdly, Theorem~\ref{thm:analogues} is a new result that makes the main proof in Section~\ref{subsec:main-proof-arc} neat and concise.
\end{rem}

\begin{rem}\label{rem:generalise-all-degrees}
    The reader might wonder why we only have degrees that are multiples of $d-3$ and if it is possible to define classes in homotopy groups of all degrees. This stems from the mentioned connectivity of the fibre of $T_{n+1}(C,M)\to T_n(C,M)$, and the easy description of its lowest nonvanishing homotopy group as the group of Lie trees. Nevertheless, one can also express its higher homotopy groups combinatorially in terms of $M$ (cf.\ Theorem~\ref{thm:thesis}). Therefore, we strongly believe that \emph{all classes} in homotopy groups of $\Embp(C,M)$ can be obtained by analogues of grasper surgery. We plan to pursue this in future work, see Question~\ref{quest:real-map-all-deg}.
\end{rem}

\begin{rem}\label{rem:generalise-nonconnected}
    The assumption in Theorem~\ref{thm:main} that $C$ is connected is inessential, but convenient (the Taylor tower has been less studied for nonconnected cases). The group $\ker(\p_{n+1})$ should have analogous description, where $\pi_1M$-decorated trees additionally have leaves partitioned into $c=|\pi_0C|$ subsets, and this has a clear analogue for graspers relative to $\u\colon C\hra M$, see Remark~\ref{rem:nonconn-graspers}.
\end{rem}

\subsection{Consequences}

    By the Goodwillie--Klein--Weiss Convergence Theorem~\cite{GKW,GKmultiple}, the map $\evrel_{n+1}$ is an isomorphism for $1\leq k\leq (n+1)(d-3)$, see Theorem~\ref{thm:rpn} below.
    However, in the proof of Theorem~\ref{thm:main} we do not use that result until the very end, so it independently shows the surjectivity of $\evrel_{n+1}$, and thus of $\pi_{n(d-3)}\ev_{n+1}$ by induction (see also Questions~\ref{quest:real-map-all-deg} and \ref{quest:reprove-GKW}). We do use that result to conclude that $\realmap_n$ vanishes on the relations $(\AS)$ and $(\IHX)$, but one could give a geometric argument instead, see Remark~\ref{rem:graded-signs}.

    On the other hand, the Convergence Theorem implies the following.
    % See also Section~\ref{subsec:related-work} for further references.

\begin{cor}\label{cor:main}
Let $C$ be equal to $\D^1$ or $\S^1$ and $M$ be any compact smooth manifold with boundary, $\dim M=d\geq4$.
    For every linear combination $F$ of trees from the set $\Tree_{\pi_1M}(n)$ there is a family $\realmap_n(F)\colon\S^{n(d-3)}\to\Embp(C,M)$, which is homotopic to $\const_\u$ if and only if $F$ belongs to the linear combination of $(\AS)$ and $(\IHX)$ relations and the image of $\delta_n\colon\pi_{n(d-3)+1}T_n(C,M)\to\Lie_{\pi_1M}(n)$.
    Moreover, the nontriviality of $\realmap_n(F)$ is first detected in $\pi_{n(d-3)}T_k(C,M)$ for $k=n+1$. 
\end{cor}
The homomorphism $\delta_n$ comes from the natural map from the absolute to the relative homotopy group, and its image can be identified with the union of images of the differentials in the spectral sequence computing the homotopy groups of the tower of fibrations given by the canonical maps $T_{k+1}(C,M)\to T_k(C,M)$, $k\geq1$. This spectral sequence has been studied extensively in the cases $C=\D^1$, $M=\D^d$ for $d\geq4$; for example, the image of $\delta_n$ has been computed rationally. In Section~\ref{subsec:related-work} below we review those results. Combining them with Theorem~\ref{thm:main} gives the following.

\begin{cor}\label{cor:Dd}
    For any $d\geq4$ and $n\geq1$ the grasper classes generate a subgroup
    \[\begin{tikzcd}
        {\mathcal{A}^{\mathsf{T,odd/even}}_n\otimes\Q}\rar[hook ]{\realmap_n} & \pi_{n(d-3)}\Embp(\D^1,\D^d)\otimes\Q,
    \end{tikzcd}
    \]
    that realises all classes in the lower vanishing line of the Goodwillie--Weiss--Sinha spectral sequence.
    For $n\leq 6$ these inclusions are true without tensoring with $\Q$, and the abelian groups $\mathcal{A}^{\mathsf{T,odd/even}}_n$ are free of the following ranks:
\begin{center}\renewcommand{\arraystretch}{1.1}
\begin{tabular}{c|cccccccccccc}
    $n$& 1 &2 &3 &4 &5 &6 \\\hline
    $\mathrm{rank}(\ATodd_n)$ &$0$ &$1$ &$1$ &$2$ &$3$ &$5$\\
    $\mathrm{rank}(\ATeven_n)$ &$0$ &$1$ &$1$ &$0$ &$2$ &$1$ 
\end{tabular}
\end{center}
    For $n=2$ this is actually an isomorphism, realising the lowest nonvanishing homotopy group $\pi_{2(d-3)}\Embp(\D^1,\D^d)$ as $ \Z\;\cong\faktor{\Z[\Tree(2)]}{\AS}\cong\Lie(2)\cong\mathcal{A}^{\mathsf{T,odd/even}}_2$.
\end{cor}
Note that the inclusion $\realmap_n$ will be mostly strict in other degrees~\cite{Turchin-Hodge}; for example, although $\ATeven_4=0$ we have $\pi_4\Embp(\D^1,\D^4)\otimes\Q\cong\Q$ coming from $\pi_4T_3(\D^1,\D^4)$ by~\cite[Table 2]{Scannell-Sinha}.

Furthermore, even though our main Theorem~\ref{thm:main} is a statement about embeddings of connected 1-manifolds $C=\D^1$ or $C=\S^1$, an intermediate result, Theorem~\ref{thm:multi-family}, constructs families $\mu_{\Gamma_i,d}$ of string links $\Embp(\bigsqcup_{i=0}^n\D^1,\D^d)$, and has the following consequence.

\begin{cor}\label{cor:string-links}
    For any $F=\sum_i\e_i\Gamma_i\in\Z[\Tree(n)]$ there is a class
    \[
        \sum_i\e_i[\mu_{\Gamma_i,d}]\in\pi_{n(d-3)}
        \Big(\Embp(\bigsqcup_{i=0}^n\D^1,\D^d);\bigsqcup_{i=0}^n a_i\Big),
    \]
    that is nontrivial as soon as $\realmap_n(F)\in\pi_{n(d-3)}\Embp(\D^1,M)$ is nontrivial for some manifold $M$. 
    Moreover, these classes live in the subspace $\Embp(\D^1,\D^d\sm\bigsqcup_{i=1}^na_i)\subseteq\Embp(\bigsqcup_{i=0}^n\D^1,\D^d)$. In particular, over $\Q$ they generate a subspace
    \[\begin{tikzcd}
        {\mathcal{A}^{\mathsf{T,odd/even}}_n\otimes\Q}\rar[hook ]{\mu_{-,d}} & \pi_{n(d-3)}
        \Embp(\bigsqcup_{i=0}^n\D^1,\D^d)\otimes\Q.
    \end{tikzcd}
    \]
\end{cor}

Indeed, if the displayed class is trivial, then the image of that isotopy under a grasper $\TG_n$ trivialises $\realmap_n(F)\coloneqq\TG_n\circ\mu_{\Gamma,d}(\vec{\tau})$, see Step 2 of the outline in Section~\ref{subsec:overview}.

%~~~~~~~~~~
\subsection{Related work}\label{subsec:related-work}
Let us mention related previous results. 

\subsubsection*{Degree one}
For $n=1$ we have $T_1(C,M)=\Imm(C,M)$ and $\Lie_{\pi_1M}(1)=\Z[\pi_1M]$, and $\realmap_1$ is precisely the realisation map from~\eqref{eq:realmap1}, with the inverse $\evrel_2=\Dax$. This case is studied in detail in \cite{K-Dax}, with the goal of making $\evrel_2=\Dax$ as computable in practice as possible. 

\subsubsection*{Dimension three}
As already mentioned, in \cite{K-thesis-paper} we define $\realmap_n$ for $d=3$ and all $n\geq1$. The purpose of this paper is to do the same for all $d\geq4$, and explain the exact connection between clasper surgery and Whitehead products. 

\subsubsection*{Homology}
Cattaneo, Cotta-Ramusino and Longoni~\cite{Cattaneo-CottaRamusino-Longoni} construct nontrivial de Rham cohomology classes in $H^{n(d-3)}(\Embp(\D^1,\D^d);\R)$ generalising to all $d\geq3$ the Vassiliev filtration of knot invariants. More precisely, they give a map
\[
   H(D^{d,n,m})\ra H^{n(d-3)+m}(\Embp(\S^1,\R^d);\R)
\]
which they show is injective for $m=0$, where $D^{d,n,m}$ is a chain complex generated by graphs (and depending on $d$ only modulo two). 
These maps are constructed using Bott--Taubes configuration space integrals, and injectivity is shown by evaluating them on particular cycles constructed using chord diagrams and ``resolutions'' of double points. 

In fact, Longoni~\cite{Longoni} shows that these cycles factor through a map
\begin{equation}\label{eq:Longoni}
   H(D_{d,n,0})\ra H_{n(d-3)}(\Embp(\S^1,\R^d);\R),
\end{equation}
where $H(D_{d,n,0})$ is the quotient of the linear span of chord diagrams by $1T$ relations and $4T$ relations of either even or odd type. This is related to our classes in $\pi_{n(d-3)}(\Embp(\D^1,\D^d)$ in the same manner the Gusarov--Habiro and Vassiliev filtrations are related in the classical case, see~\cite{Habiro}. Namely, one considers immersions of $\D^1$ into $\D^d$ and defines resolutions of their double points, using sphere-worth of directions to push a strand off (cf.\ Definition~\ref{def:base} below).

The problem in homology analogous to Remark~\ref{rem:generalise-all-degrees} was addressed by Longoni~\cite{Longoni} and Pelatt and Sinha~\cite{Pelatt-Sinha}, following \cite{Cattaneo-CottaRamusino-Longoni}. Namely, both papers give a construction of a nontrivial class in $H_{3(d-3)+1}(\Emb(\S^1,\R^d);\Z)$ for any even $d\geq4$, corresponding to a graph with one fourvalent vertex and all other trivalent, and used configuration space integrals to show their nontriviality. Sakai~\cite{Sakai-nontrivalent} gives such a generator for odd $d$ by a different approach. See also Question~\ref{ques:rel-to-homology}.

\subsubsection*{Balls}

The (rational) homotopy type of the Taylor tower for $\Embp(\D^1,\D^d)$ and the associated homotopy spectral sequence  have been studied by many authors~\cite{Scannell-Sinha,Conant,LT-GC,ALTV,Turchin-Hodge,FTW,BH}, and some ranks of homology and homotopy groups have been computed. By~\cite[Thm.17.4]{Turchin-Hodge} the rank of $\pi_{\leq n}\Embp(\D^1,\D^d)$ grows at least exponentially in $n$.

In particular, it is known that when tensored with the rationals the homotopy spectral sequence for $\Embp(\D^1,\D^d)$ collapses on the second page, with several proofs by now \cite{LT-GC,ALTV,FTW,BH} (the last two apply also to $d=3$, and the last one also proves a partial collapse $p$-locally).

Moreover, \emph{when $d$ is odd} the image of the $d_1$-differential in the diagonal $E^2_{-(n+1),n(d-2)+1}$ of the spectral sequence has been identified by Conant~\cite{Conant} and Lambrechts--Turchin~\cite{LT-GC} independently, and is given by the so-called $\stusq$ relations, depicted in Figure~\ref{fig:stu2}. The result
\[
        \ATodd_n\coloneqq\faktor{\Lie(n)}{\stusq}\cong \faktor{\Z[\Tree(n)]}{\AS,\IHX,\stusq}
\]
is the abelian group of \emph{Jacobi trees} for $n\geq2$, and $\ATodd_1\coloneqq0$, closely related to the (primitives of the) algebra of chord diagram $\mathcal{A}$ from Vassiliev theory for classical knots. Ranks of these groups have been computed in a range by Kneissler~\cite{Kneissler} and by Turchin in \cite[Table 1]{Turchin-Hodge}:
\begin{center}% \renewcommand{\arraystretch}{1.3}
\begin{tabular}{c|cccccccccccc}
    $n$& 1 &2 &3 &4 &5 &6 &7 &8 &9 &10 &11 &12\\\hline
    $\dim(\ATodd_n\otimes\Q)$ &$0$ &$1$ &$1$ &$2$ &$3$ &$5$ &$8$ &$12$ &$18$ &$27$ &$39$ &$55$
\end{tabular}
\end{center}
Furthermore, it is known that for $n\leq6$ there is no torsion, see~\cite[Table 6]{Turchin-bialgebra}.

\begin{figure}[!htbp]
    \centering
    \begin{tikzpicture}[scale=0.8,even odd rule,baseline=0.5cm]
    \draw (-.5,0) -- (4.5,0);
    \draw[blue]
        (-0.25,0) node[circle,fill,inner sep=0.8pt,label=below:$0$]{} -- (0,1) node[red,circle,fill,inner sep=0.9pt]{} node[red,left]{\footnotesize$v_0$}-- (1,1) node[circle,fill,inner sep=0.8pt]{} -- (1,0.5) node[circle,fill,inner sep=0.8pt]{} -- (0.65,0) node[circle,fill,inner sep=0.8pt,label=below:$1$]{}
        (1,0.5)  --  (3,0) node[circle,fill,inner sep=0.8pt,label=below:$2$]{}
        (1,1) -- (2,1) node[red,circle,fill,inner sep=0.9pt]{} node[red,above]{\footnotesize$v_3$} -- (2,0) node[circle,fill,inner sep=0.8pt,label=below:$3$]{}
        (2,1) -- (3,1) node[circle,fill,inner sep=0.8pt]{} -- (4,0) node[circle,fill,inner sep=0.8pt,label=below:$5$]{}
        (3,1) -- (4,1) node[circle,fill,inner sep=0.8pt]{} to[out=-100,in=85,distance=2cm]  (1.35,0) node[circle,fill,inner sep=0.8pt,label=below:$4$]{}
        (4,1) to[out=90,in=90, distance=1.7cm]  (0,1) ;
\end{tikzpicture}$=$%
\begin{tabular}{c c}
\begin{tikzpicture}[scale=0.8,even odd rule,baseline=0.5cm]
    % \draw[white,semithick] (5.6,1) -- (5.8,1) (5.6,0.9) -- (5.8,0.9);
    \draw (6.1,0) -- (11.5,0);
    \draw[blue]
        (6.75,0) node[circle,fill,inner sep=0.8pt,label=below:$0$]{} -- (7,1) node[circle,fill,inner sep=0.8pt]{} -- (8,1) node[circle,fill,inner sep=0.8pt]{}
        -- (8,0.5) node[circle,fill,inner sep=0.8pt]{} -- (7.65,0) node[circle,fill,inner sep=0.8pt,label=below:$1$]{}
        (8,0.5)  -- (10,0) node[circle,fill,inner sep=0.8pt,label=below:$5$]{}
        (8,1) to[out=0,in=90,distance=0.9cm] (8.9,0) node[circle,fill,inner sep=0.8pt,label=below:$3$]{}
        (9.3,0) node[circle,fill,inner sep=0.8pt,label=below:$4$]{} to[out=90,in=180,distance=0.9cm] (10,1) node[circle,fill,inner sep=0.8pt]{} -- (11,0) node[circle,fill,inner sep=0.8pt,label=below:$6$]{}
        (10,1) -- (11,1) node[circle,fill,inner sep=0.8pt]{} to[out=-100,in=85,distance=1.7cm] (8.35,0) node[circle,fill,inner sep=0.8pt,label=below:$2$]{}
        (11,1) to[out=90,in=90, distance=1.6cm]  (7,1) ;
\end{tikzpicture} &
$-$\begin{tikzpicture}[scale=0.8,even odd rule,baseline=0.5cm]
    \draw (6.1,0) -- (11.5,0);
    \draw[blue]
        (6.75,0) node[circle,fill,inner sep=0.8pt,label=below:$0$]{} -- (7,1) node[circle,fill,inner sep=0.8pt]{} -- (8,1) node[circle,fill,inner sep=0.8pt]{}
        -- (8,0.5) node[circle,fill,inner sep=0.8pt]{} -- (7.65,0) node[circle,fill,inner sep=0.8pt,label=below:$1$]{}
        (8,0.5)  -- (10,0) node[circle,fill,inner sep=0.8pt,label=below:$5$]{}
        (8,1) to[out=0,in=90,distance=1cm] (9.3,0)  node[circle,fill,inner sep=0.8pt,label=below:$4$]{}
        (8.9,0) node[circle,fill,inner sep=0.8pt,label=below:$3$]{} to[out=90,in=180,distance=1cm] (10,1) node[circle,fill,inner sep=0.8pt]{} -- (11,0) node[circle,fill,inner sep=0.8pt,label=below:$6$]{}
        (10,1) -- (11,1) node[circle,fill,inner sep=0.8pt]{} to[out=-100,in=85,distance=1.7cm] (8.35,0) node[circle,fill,inner sep=0.8pt,label=below:$2$]{}
        (11,1) to[out=90,in=90, distance=1.6cm]  (7,1) ;
\end{tikzpicture}
\\
\begin{tikzpicture}[scale=0.8,even odd rule,baseline=0.5cm]
    \draw (6.1,0) -- (11.5,0);
    \draw[blue]
        (6.9,0) node[circle,fill,inner sep=0.8pt,label=below:$1$]{} to[out=90,in=180,distance=0.7cm] (8,1) node[circle,fill,inner sep=0.8pt]{} -- (8,0.5) node[circle,fill,inner sep=0.8pt]{} -- (7.65,0) node[circle,fill,inner sep=0.8pt,label=below:$2$]{}
        (8,0.5)  -- (10,0) node[circle,fill,inner sep=0.8pt,label=below:$5$]{}
        (8,1) -- (9,1) -- (9,0) node[circle,fill,inner sep=0.8pt,label=below:$4$]{}
        (9,1) node[circle,fill,inner sep=0.8pt]{} -- (10,1) node[circle,fill,inner sep=0.8pt]{} -- (11,0) node[circle,fill,inner sep=0.8pt,label=below:$6$]{}
        (10,1) -- (11,1) node[circle,fill,inner sep=0.8pt]{} to[out=-100,in=85,distance=1.7cm] (8.35,0) node[circle,fill,inner sep=0.8pt,label=below:$3$]{}
        (11,1) to[out=90,in=90, distance=2.1cm]  (6.6,0) node[circle,fill,inner sep=0.8pt,label=below:$0$]{} ;
\end{tikzpicture} 
&
$-$\begin{tikzpicture}[scale=0.8,even odd rule,baseline=0.5cm]
    \draw (6.1,0) -- (11.5,0);
    \draw[blue]
        (6.6,0) node[circle,fill,inner sep=0.8pt,label=below:$0$]{} to[out=90,in=180,distance=0.8cm] (8,1) node[circle,fill,inner sep=0.8pt]{} -- (8,0.5) node[circle,fill,inner sep=0.8pt]{} -- (7.65,0) node[circle,fill,inner sep=0.8pt,label=below:$2$]{}
        (8,0.5)  -- (10,0) node[circle,fill,inner sep=0.8pt,label=below:$5$]{}
        (8,1) -- (9,1) -- (9,0) node[circle,fill,inner sep=0.8pt,label=below:$4$]{}
        (9,1) node[circle,fill,inner sep=0.8pt]{} -- (10,1) node[circle,fill,inner sep=0.8pt]{} -- (11,0) node[circle,fill,inner sep=0.8pt,label=below:$6$]{}
        (10,1) -- (11,1) node[circle,fill,inner sep=0.8pt]{} to[out=-100,in=85,distance=1.7cm] (8.35,0) node[circle,fill,inner sep=0.8pt,label=below:$3$]{}
        (11,1) to[out=90,in=90, distance=2.1cm]  (6.9,0) node[circle,fill,inner sep=0.8pt,label=below:$1$]{} ;
\end{tikzpicture}
\end{tabular}
    \caption[The \protect$\stusq$ relation.]{The $\stusq$ relation is obtained by equating to zero the four terms on the right. They are in turn obtained from the graph on the left by resolving two vertices $v_0$ and $v_3$ as depicted.}
    \label{fig:stu2}
\end{figure}
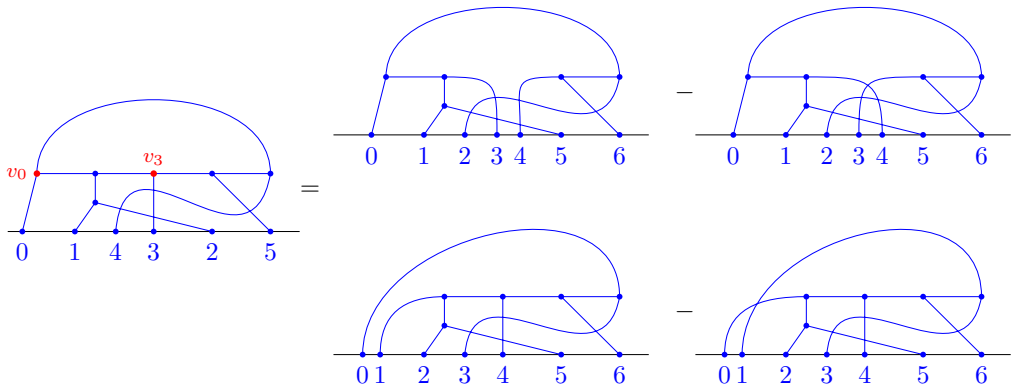
\emph{When $d$ is even} one can repeat Conant's proof to see that $E^2_{-(n+1),n(d-2)+1}$ is given by
\[
        \ATeven_n\coloneqq\faktor{\Lie(n)}{\stusq_{\even}}\cong \faktor{\Z[\Tree(n)]}{\AS,\IHX,\stusq_{\even}}
\]
where $\stusq_{\even}$ is the same as in Figure~\ref{fig:stu2} except that minus signs should be replaced by plus signs.

We do not prove this claim here, but refer to Conant's proof~\cite{Conant} (perhaps more easily this sign can be spotted in the author's thesis \cite[105]{K-thesis}, where unfortunately by oversight only the odd case is discussed). Some computations of this group have been made by Scannell--Sinha~\cite[the lower nonvanishing line in Table 2]{Scannell-Sinha} (cf.\ \cite[Table 5]{Turchin-bialgebra}, which is over $\Z$):
\begin{center}% \renewcommand{\arraystretch}{1.3}
\begin{tabular}{c|cccccccccccc}
    $n$& 1 &2 &3 &4 &5 &6 \\\hline
    $\dim(\ATeven_n\otimes\Q)$ &$0$ &$1$ &$1$ &$0$ &$2$ &$1$ 
\end{tabular}
\end{center}
Moreover, let us mention that in~\cite[624]{Turchin-bialgebra} Turchin computes 
\begin{equation}\label{eq:Turchin-deg-2}
    H_{2(d-3)}\Embp(\D^1,\D^d)\cong\pi_{2(d-3)}\Embp(\D^1,\D^d)\cong\Z,
\end{equation}
and Budney~\cite[60]{Budney-family} constructs a generator. Our grasper surgery for $n=2$ in Corollary~\ref{cor:Dd} can be viewed as a more explicit description of that class (up to sign), and is schematically depicted in Figure~\ref{fig:realmap-2}.
Moreover, \cite[Prop.3.9]{Budney-family} constructs an isomorphism
\begin{equation}\label{eq:Haefliger-foliate}
    \pi_0\Embp(\D^{4k-1},\D^{6k})\cong\pi_{4k-2}\Embp(\D^1,\D^{2k+2})
\end{equation}
as a foliation map. Since grasper surgery is closely related to foliations (Section~\ref{subsec:foliations}), it follows that Budney's isomorphism takes the Haefliger's construction \eqref{eq:Haefliger} to a grasper surgery (Section~\ref{subsec:Haefliger}).

\subsubsection*{Other ambient manifolds}
Budney--Gabai~\cite{Budney-Gabai} studied the groups $\pi_{d-3}\Embp(\S^1,\S^1\times\S^{d-1})$ and $\pi_2\Embp(\D^1,\S^1\times\D^3)$. Our graspers are closely related to their bracket operation \cite[§4]{Budney-Gabai}.

    Starting from a connected trivalent graph $\Pi$ on $2n$ vertices Botvinnik and Watanabe~\cite{Botvinnik-Watanabe} constructed a map $H^\Pi\colon\S^{n(d-3)}\to\Emb^{\mathrm{fr}}(\S^1\sqcup\S^{d-2},X)$,
    where $X$ is any smooth $d$-manifold and the $\S^{d-2}$ component is unknotted and stands still and every $\S^1$ component in the family links it with the linking number one. Thus, they actually have
    \[
        h^\Pi\colon\S^{n(d-3)}\to\Emb^{\mathrm{fr}}(\S^1,X\#\,\S^1\times\ball^{d-1}).
    \]
    since for an unknotted $\S^{d-2}\subseteq X$ we have $X\sm\nu\S^{d-2}\cong X\#\,\S^1\times\ball^{d-1}$. The key ingredient of the construction are foliations of Haefliger's Borromean rings.
    
    By attaching $(d+1)$-dimensional $2$- and $1$-handles (dual to $(d-1)$-handles) along these links, one can associate to $H^\Pi$ an $X$-bundle over $\S^{n(d-3)}$, giving:
    \[
        W(H^\Pi)\colon\S^{n(d-3)}\to B\Diff_\partial(X).
    \]
This recovers the remarkable construction by Watanabe~\cite{Watanabe} of classes in homotopy groups of $\Diff_\partial(X)$, whose nontriviality for $X=\D^d$ and many $\Pi$ he showed using Kontsevich configuration space integrals.  The nontriviality of the homootpy class of $h^\Pi$ follows from this.

%~~~~~~~~~~
\subsection{Open questions}\label{subsec:questions}

    The following is an extremely interesting question:
    \begin{question}\label{quest:Botvinnik-Watanabe}
        How is $h^\Pi$ related to our knotted families $\realmap_n(\Gamma)$ for $C=\S^1$ and $M=X\#\,\S^1\times\ball^{d-1}$ (and either forget the framing or equip our families with a framing)?
    \end{question}
    If a direct relation exists, it is bound to give an interesting relation between trivalent graphs on $2n$ vertices and binary trees with $n$ leaves. Moreover, this can give a guiding step for reproving Watanabe's result using embedding calculus. However, let us note that our approach can describe potential torsion elements, whereas Kontsevich--Watanabe invariants are defined only over $\R$.

A natural next step in the geometric study of embedding spaces of arcs is the following problem:
    \begin{question}\label{quest:real-map-all-deg}
        For each $0<m<d-3$ construct a map
        \[
            \realmap_n^m\colon\pi_{n(d-3)+m}\pF_{n+1}(M)\to\pi_{n(d-3)+m+1}\big(T_n(\D^1,M),\Embp(\D^1,M)\big)
        \]
        such that $\evrel_{n+1}\circ\realmap_n^m=\Id$.
    \end{question}   
This was mentioned in Remark~\ref{rem:generalise-all-degrees}, and can probably be done using similar multi-families of commutators of foliated meridians, but now allowing more than one use of the same meridian.

Continuing on this is the following problem, which would reprove in the case $C=\D^1$ the fundamental theorem of Goodwillie and Klein \cite{GKmultiple} (see Theorem~\ref{thm:rpn}) by different methods.
\begin{question}\label{quest:reprove-GKW}
    For each $0\leq m<d-3$ prove that $\realmap_n^m\circ\evrel_{n+1}=\Id$ as well.
\end{question}

An easier problem is to try to generalise to all $1$-manifolds, as mentioned in Remark~\ref{rem:generalise-nonconnected}:
\begin{question}\label{quest:C-nonconn}
    Does Theorem~\ref{thm:main} hold for any $1$-manifold $C$, i.e.\  $C$ nonconnected?
\end{question}

    Finally, one can generalise the correspondence of Jacobi and chord diagrams to higher dimensions:
    \begin{question}\label{ques:rel-to-homology}
        How is $\realmap_n$ related to \eqref{eq:Longoni}? Study the exact relation of our classes in homotopy to the classes constructed previously in homology.
    \end{question}

%~~~~~~~~~~
\subsection*{Acknowledgements} 
I wish to thank Yuqing Shi and Pete Teichner for our numerous discussions during my PhD studies at the Max Planck Institute in Bonn, that gave rise to the realisation map in dimension three. Many thanks to Ryan Budney and Dave Gabai for citing this paper ahead of time, and to Peter Feller for encouraging me to finish it. Importantly, I am grateful to the referee for patient reading and detailed comments.

%~~~~~~~~~~~~~~~~~~~~~~~~~~~~~~~~~~~~~~
\section{Overview and examples}

%~~~~~~~~~~
\subsection{Outline of the proof}\label{subsec:overview}

We now outline this proof for $C=\D^1$; the proof for $C=\S^1$ is easily derived from the one for $C=\D^1$ in Section~\ref{subsec:main-proof-circle}. Let us fix a decorated tree $\Gamma^{g_{\ul{n}}}$. This is a rooted binary planar tree $\Gamma$ with $n$ leaves each decorated by an element $g_i\in\pi_1M$, $1\leq i\leq n$.

%%%%%%
\textbf{Step 1.}
In Section~\ref{sec:multi-family} we will construct a family
\begin{equation}\label{eq-intro:mu}
    \mu_{\Gamma,d}\colon\S^{n(d-3)}\to\Embp(\D^1,\ball^d)
\end{equation}
as, roughly speaking, an iterated commutator of the loops foliating the meridians $m_i(\S^{d-2})$ of the arc $\u(J_i)$ for $1\leq i\leq n$. One has to make these commutators in an embedded manner, which can be done  by iteratively plumbing together copies of $\ball^d$; this is inspired by the construction in $d=3$ using classical $2$-dimensional gropes $G_\Gamma^c$.

Moreover, in Theorem~\ref{thm:multi-family} we will define a multi-family
\begin{equation}\label{eq-intro:mu-star}
    \mu^\star_{\Gamma,d}\colon \I^n\times\S^{n(d-3)}\to \Embp(\D^1,\ball^d),
\end{equation}
such that $\mu^\star_{\Gamma,d}(\vec{0},-)=\mu_{\Gamma,d}$ and $\mu^\star_{\Gamma,d}(\vec{t},-)=a_0$ if $\vec{t}\in\I^n$ has a coordinate equal to $1$. This will be the key ingredient for lifting our classes to $\pi_{n(d-3)+1}(T_n(\D^1,M),\Embp(\D^1,M))$ in Step 3.

%%%%%%
\textbf{Step 2.}
In Section~\ref{sec:realmap} we will define a family 
\[
    \realmap_n(\Gamma^{g_{\ul{n}}})\colon\S^{n(d-3)}\to\Embp(\D^1,M)
\]
by postcomposing the family $\mu_{\Gamma,d}$ of arcs in $\ball^d$ from \eqref{eq-intro:mu} with any grasper $\TG_n\colon\ball^d\hra M$ that is decorated by $g_{\ul{n}}=(g_1,\dots,g_n)$. In other words, we let
\[
    \realmap_n(\Gamma^{g_{\ul{n}}})(\vec{\tau})\coloneqq\TG_n\circ\mu_{\Gamma,d}(\vec{\tau}).
\]
A grasper was defined in~\eqref{eq:G-condition} (see also Definition~\ref{def:grasper}), and its decoration in Definition~\ref{def:under-decoration}: it consists of the loops $g_i\in\pi_1M$ obtained by going from $\u(-1)$ to $\u(J_i)$ along $\u$ and then back on $\TG_n$.

Note that whereas family $\mu_{\Gamma,d}$ of reimbeddings of the arc $a_0$ can be viewed as the \emph{universal iterated Whitehead product} of the shape $\Gamma$ and dimension $d$, the family $\realmap_n(\Gamma^{g_{\ul{n}}})$ is the mentioned \emph{embedded version of the Whitehead product} according to $\Gamma^{g_{\ul{n}}}$ of meridian spheres for $\u$. Thus,  $\mu_{\Gamma,d}$ plays the role of an abstract clasper, $\TG_n$ of an embedded clasper, while $\realmap_n(\Gamma^{g_{\ul{n}}})$ generalises clasper surgery.

%%%%%%
\textbf{Step 3.}
In Section~\ref{sec:tower} we exhibit a homotopy from $\ev_n\realmap_n(\Gamma^{g_{\ul{n}}})$ to $\ev_n(\const_\u)$ through maps $\S^{n(d-3)}\to T_n(\D^1,M)$, using the multi-family $\mu^\star_{\Gamma,d}$ of arcs in $\ball^d$ from \eqref{eq-intro:mu-star}. In particular, for the space $T_n(\D^1,M)$ we will use the reduced punctured knots model $\rpT_n(M)$ developed in \cite{KT-rpn}, based on the methods of \cite{K-thesis-paper}. Necessary results are listed in Section~\ref{subsec:tower}.

Firstly, in Definition~\ref{def:rpn} we recall from \cite{KT-rpn} that a point in the space $\rpT_n(M)$ is a map
\[
    \Delta^{n-1}\to\Embp(J_0,M_{0\ul{n}})
\]
that on the faces of the simplex restricts to embeddings with a particular intersection pattern with $\u\cap M_{0\ul{n}}$. Here $M_{0\ul{n}}$ is the manifold obtained from $M$ by removing a tubular neighbourhood $\nu\u$ and then adding back the pieces $\nu(\u|_{J_i})$ for $0\leq i\leq n$ (see Figure~\ref{fig:nbhd-of-U-S}), and $\Embp(J_0,M_{0\ul{n}})$ is the space of embeddings $\D^1\hra M_{0\ul{n}}$ which have the same boundary as $\u|_{J_0}$ Moreover, to define the map $\ev_n$ we use $\Embp(\D^1,M)\simeq\Embp(J_0,M_0)$ and the inclusion $M_0\hra M_{0\ul{n}}$. See Section~\ref{subsec:tower}.

Thus, the desired path in $\Omega^{n(d-3)}\rpT_n(M)$ can be defined as
\[\begin{tikzcd}
    {[0,1]}\times\Delta^{n-1}\times\S^{n(d-3)}\rar & \I^n\times\S^{n(d-3)}
    \rar{\mu^\star_{\Gamma,d}} & \Embp(\D^1,\ball^d)\rar{\TG_n\circ-} & 
    \Embp(J_0,M_{0\ul{n}}),
\end{tikzcd}
\]
using~\eqref{eq:G-condition} to ensure the desired intersection pattern on $\partial\Delta^{n-1}$, see Theorem~\ref{thm:grasper-gives-path}. For a map ${[0,1]}\times\Delta^{n-1}\sra\I^n$ see Remark~\ref{rem:previous-approach}. This will complete the construction of the class
\[
    \realmap_n(\Gamma^{g_{\ul{n}}})=[(\TG_n\circ-)\circ\mu^\star_{\Gamma,d}]\in\pi_{n(d-3)+1}(\rpT_n(M),\Embp(\D^1,M)).
\]

%%%%%%
\textbf{Step 4.}
Finally, to prove the main Theorem~\ref{thm:main} it remains to show that $\evrel_{n+1}$ of the resulting relative homotopy class is equal to $[\Gamma^{g_{\ul{n}}}]\in\Lie_{\pi_1M}(n)$. We carry this out in Section~\ref{sec:proof}.
% On the other hand, inside of $\ball^d$ we will define a family $\mu_{\Gamma,d}$ of reimbeddings of the arc $a_0$, that we can postcompose with $\TG_n$ to obtain the knotted family $\realmap_n(\Gamma^{g_{\ul{n}}})$, whereas a multi-family $\mu^\star_{\Gamma,d}$ of reimbeddings of $a_0$ will give rise to the desired nullhomotopy in $(ii)$. 
More precisely, in the group
\[
    \pi_{n(d-3)+1}(\rpT_n(M),\rpT_{n+1}(M))\cong\Lie_{\pi_1M}(n),
\]
where we use the isomorphism from \cite{K-thesis-paper}, we check that $\evrel_{n+1}\realmap_n(\Gamma^{g_{\ul{n}}})=[\Gamma^{g_{\ul{n}}}]$. 

The proof of this will be by induction on the degree of $\Gamma$ and will use Theorem~\ref{thm:analogues}, which establishes the correspondence of $\mu_{\Gamma,d}$ with the Samelson product $x_{\Gamma,d-2}\colon\S^{n(d-3)}\to\Omega\bigvee_n\S^{d-2}$, and of $\mu^\star_{\Gamma,d}$ with $x^\star_{\Gamma,d-1}\colon\I^n\times\S^{n(d-3)}\to\Omega\bigvee_n\ball^{d-1}$ (a certain restriction, defined in Proposition~\ref{prop:suspensions}, of the Samelson product $x_{\Gamma,d-1}$).

\begin{rem}
    In this overview we have made some simplifications. 
    Firstly, instead of $\Embp(\D^1,\ball^d)$ in the rest of the paper we will use $\Emb((\S^1,\S^1_+),(\ball^d,a_0))$ as the target of $\mu^\star_{\Gamma,d}$. Secondly, instead of the unit ball $\ball^d$ we will define its ``extended'' version $\ball^d_{ext}$, which has a distinguished region along which we can conveniently glue when we perform plumbings.
    Thirdly, instead of constructing a path in $\Omega^{n(d-3)}\rpT_n(M)$ we will directly produce a point in $\Omega^{n(d-3)}\pF_{n+1}(M)$, for the layer $\pF_{n+1}(M)$, defined as the fibre of the map $\p_{n+1}\colon\rpT_{n+1}(M)\to\rpT_n(M)$, see Remark~\ref{rem:previous-approach}.
\end{rem}

%~~~~~~~~~~
\subsection{Examples}\label{sec:examples}

Let us clarify our constructions with a few examples.

\subsubsection*{Degree one}
For $n=1$ the only tree $\Gamma$ is a chord, so that $\Lie_{\pi_1M}(1)\cong \Z[\Tree_{\pi_1M}(1)]\cong \Z[\pi_1M]$. For a group element $g\in\pi_1M$ its realisation
\[
    \realmap_1\left(\gchord{g}\right)\in\pi_{d-3}(\Embp(\D^1,M),\u)
\]
is the connected sum of $\u$ with its meridian $m(\S^{d-2})$ based by $g$, as in Figure~\ref{fig:realmap-1}.
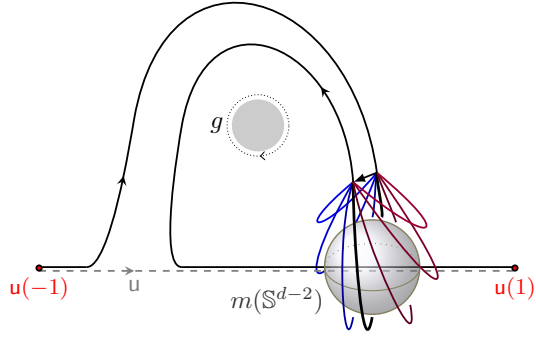
\begin{figure}[!htbp]
    \centering
       \begin{tikzpicture}[scale=0.73,every node/.style={scale=1.1},even odd rule]
        \clip (0.25,-1.6) rectangle (11.75,6);
        % the leftmost
        \draw[blue!90!black,  thick]
                (7.6,1.8) to[out=-140,in=-130, distance=1.8cm] (8.1,2);
        %the second from left
        \draw[blue!80!black,  thick]
                (7.6,1.8) to[out=-120,in=-110, distance=3cm] (8.1,2);
        \fill[white] (8,0) circle (1cm);
        
        \draw[black, thick] 
            (1,0) -- (2,0)
            (4,0) -- (11,0)
            (6,-0.7) node{\color{black!80!white}$m(\mathbb{S}^{d-2})$};
        \draw[gray,thick, dashed,postaction=decorate,decoration = {markings, mark = at position 0.2 with {\arrow{stealth}}}]
            (1,-0.08) -- (11,-0.08) node[below,pos=0.2]{$\mathsf{u}$};
        % \draw[red, thick]
        %     (8, 0) node[below=1pt]{$x$} circle (1.5pt);

        % sphere
        \foreach \y in {8}{
            \shade[ball color=blue!5!white,opacity=0.55] (\y,0) circle (1cm);
            \draw[black!60!yellow] (\y-1,0) arc (180:360:1cm and 0.5cm);
            \draw[black!60!yellow,dotted] (\y-1,0) arc (180:0:1cm and 0.5cm);
            \draw[black!60!yellow,semithick] (\y,0) circle (1cm);
            }
        % the third from left
        \draw[blue!60!black,  thick]
                        (7.6,1.8) to[out=-110,in=-100, distance=1.15cm] (7.6,-1.06)
                        (8,1.06) to[out=40,in=-120, distance=0.07cm] (8.1,2);
        % middle
        \draw[black, very thick]
                (7.6,1.8) to[out=-90,in=-100, distance=1.15cm] (8,-1.06)
                (8.2,1.06) to[out=40,in=-120, distance=0.07cm] (8.1,2);
        % gp element
        \draw[black,  thick, postaction=decorate,decoration = {markings, mark = at position 0.1 with {\arrow{stealth}} , mark = at position 0.9 with {\arrowreversed{stealth}}}]
                        (2,0) to[out=0,in=-100, distance=0.5cm]
                        (3,3.15) to[out=80,in=98, distance=4cm] (8.1,2)
                        (4,0) to[out=180,in=-100, distance=0.5cm]
                        (4,3) to[out=80,in=94, distance=3cm] (7.6,1.8);
        % the grey hole
        \fill[color=black!20!white] (5.6,3) circle (0.55cm);
        \draw[densely dotted,postaction = decorate,decoration = {markings, mark = at position 0.76 with {\arrowreversed{>}} }] (5.6,3) circle (0.65cm) node[left=0.36cm]{$g$};
        
        % the third from right
        \draw[purple!50!black,  thick]
                (7.6,1.8) to[out=-80,in=-80, distance=1.4cm] (8.8,-0.8)
                (8.5,0.95) to[out=100,in=-80, distance=0.1cm] (8.1,2)
        ;
        % the second from right
        \draw[purple!65!black,  thick]
              (7.6,1.8) to[out=-55,in=-55, distance=3.5cm] (8.1,2);
        % the rightmost
        \draw[purple!80!black,  thick]
                (7.6,1.8) to[out=-40,in=-40, distance=2.2cm] (8.1,2);
        \draw[black, thick,-latex]
                (8.1,2) -- (7.6,1.8);
        % \draw[black,  thick]
        %         % (8.6,2.1) node{\textbf{$\alpha_N$}}
        %         % (10,-1.2) node[purple!60!black]{\textbf{}}
        %         (8.8,4) node{\textbf{$\mathfrak{r}_1(g)=\mathsf{u}_{|,d}=\mathbb{G}_1\circ\mu_{|,d}$}};
        \fill[red,draw=black,thick]
            (1,-0.02) circle (1.6pt) node[below]{\small$\mathsf{u}(-1)$}
            (11,-0.02) circle (1.6pt) node[below]{\small$\mathsf{u}(1)$};
    \end{tikzpicture}
    \caption{A degree $1$ grasper surgery on $\u$, giving $\realmap_1(\Gamma^g)$.}
    \label{fig:realmap-1}
\end{figure}

The reason why this class vanishes in $\rpT_1(M)\coloneqq\Embp(J_0,M\sm\nu(\u|_{J_1}))$ is that once the puncture at $J_1$ is made, the arcs foliating the meridian can be dragged up (across the meridian ball) and back to $\u(J_0)$.
These classes are detected by $\ev_2$, or equivalently, by the Dax invariant mentioned in the introduction, see references cited there. The chord decorated by $g=1\in\pi_1M$ is always killed by $\delta_1$, so in particular $\pi_{d-3}(\Embp(\D^1,\D^d),\u)=0$.

\subsubsection*{Degree two}
For $n=2$ there are two trees in $\Tree(2)$, but they are related by the antisymmetry relation $(\AS)$, so $\Lie_{\pi_1M}(2)\cong \Z[(\pi_1M)^2]$. For $g_1,g_2\in\pi_1M$ the realisation
\[
    \realmap_2\left(\igtree{1}{2}{g_1}{g_2}\right)\in\pi_{2(d-3)}(\Embp(\D^1,M),\u)
\]
is the embedded commutator of the arcs foliating the two meridians of $\u$ (depicted in Figure~\ref{fig:realmap-2}) based by $g_1$ and $g_2$. The reason that this class vanishes in $\rpT_2(M)\subseteq\Map(\Delta^1,\Embp(J_0,M_0\cup\nu(\u|_{J_1\sqcup J_2}))$ is that once the punctures at $J_1$ and $J_2$ are made, those arcs can be pulled up and back to $\u(J_0)$ (across the respective meridian balls) in $1$-parameter many ways; this gives rise to the map $\mu_{\Gamma,d-2}^\star$ in \eqref{eq-def:mu-star}, which defines $\u^\star_{\TG,\Gamma^{g_{\ul{n}}}}$ in \eqref{eq:grasper-gives-path}. These classes are detected by $\ev_3$.

\begin{figure}[!htbp]
    \centering
    \begin{tikzpicture}[scale=0.8,every node/.style={scale=1.1},even odd rule]
    \clip (1.4,-1.55) rectangle (15.9,5.7);
\foreach \y/\c in {9/blue!65!green,13/red}{
        % the leftmost
        \draw[\c!40!yellow,  thick]
                (\y-.4,1.8) to[out=-140,in=-130, distance=1.8cm] (\y+.1,2);
        %the second from left
        \draw[\c!60!yellow,  thick]
                (\y-.4,1.8) to[out=-120,in=-110, distance=3cm] (\y+.1,2);
        % spheres
            \shade[ball color=blue!5!white,opacity=0.55] (\y,0) circle (1cm);
            \draw[black!60!yellow] (\y-1,0) arc (180:360:1cm and 0.5cm);
            \draw[black!60!yellow,dotted] (\y-1,0) arc (180:0:1cm and 0.5cm);
            \draw[black!60!yellow,semithick] (\y,0) circle (1cm);
    }
    \foreach \y/\c/\rc/\mid/\i in {9/blue!65!green/blue!80!green/blue!70!green/1,
    13/red/purple/purple/2}{
    % the third from left
    \draw[\c!80!yellow, thick]
            (\y-.4,1.8) to[out=-110,in=-100, distance=1.15cm] (\y-.4,-1.06)
            (\y,1.06) to[out=40,in=-120, distance=0.07cm] (\y+.1,2);

    \draw[\mid, very thick,postaction=decorate,decoration = {markings, mark = at position 0.9 with {\arrow{latex}} }]
            (\y+-3,3) to[out=-20,in=100, distance=1.5cm]
            (\y-.4,1.8);
    % middle end
    \draw[\mid, very thick]
            (\y-.4,1.8) to[out=-90,in=-98, distance=1.5cm]
            (\y+.25,-1.03)
            (\y+.25,1.06) to[out=40,in=-80, distance=0.07cm]
            (\y+.1,2) to[out=100,in=-10, distance=1.35cm] (\y-2.5,3.15);
    % gp element and middle start
    \draw[\mid, very thick]
                    (5-\i,0) to[out=90,in=-160, distance=1.8cm]
                    (\y-5+.7,3.15) to[out=140,in=60, distance=3.8cm]
                    (\y-2.5,3.15)

                    (6-\i,0) to[out=90,in=-160, distance=1.2cm]
                    (\y-4+.2,3) to[out=140,in=60, distance=3.7cm]
                    (\y-3,3) ;
    % the grey hole
    \fill[color=black!20!white] 
        (\y-3.5,4.2) circle (0.55cm);
    \draw[black,densely dotted,postaction = decorate,decoration = {markings, mark = at position 0.65 with {\arrowreversed{>}} }] 
        (\y-3.5,4.2) circle (0.65cm) node[below=0.4cm]{$g_{\i}$};
    % the third from right
    \draw[\rc!75!black,  thick]
            (\y-.4,1.8) to[out=-80,in=-80, distance=1.4cm] (\y+.8,-0.8)
            (\y+.5,0.95) to[out=100,in=-80, distance=0.1cm] (\y+.1,2);
    % the second from right
    \draw[\rc!65!black, thick]
           (\y-.4,1.8) to[out=-55,in=-55, distance=3.5cm] (\y+.1,2);
    % the rightmost
    \draw[\rc!50!black, thick]
            (\y-.4,1.8)to[out=-40,in=-40, distance=2.2cm] (\y+.1,2);
    }
    \draw[gray,thick, dashed,postaction=decorate,decoration = {markings, mark = at position 0.2 with {\arrow{stealth}}}]
            (2,-0.08) -- (15.5,-0.08) node[below,pos=0.2]{$\mathsf{u}$};
    \draw[thick]
        (2,0) -- (3,0) 
        (5,0) -- (15.5,0) 
        (7.3,0) node[below=0.3cm]{\small$m_1(\mathbb{S}^{d-2})$}
        (11.3,0) node[below=0.3cm]{\small$m_2(\mathbb{S}^{d-2})$};
    % \draw[mygreen, -latex, very thick]
    %     (4,0) -- (5,0) node[pos=0.5,below=2pt]{$J_0$};

    \fill[mygreen,draw=black,thick] 
        (3,0) circle (1.5pt) 
        (5,0) circle (1.5pt);
    \fill[red,draw=black,thick]
        (2,-0.02) circle (1.6pt) 
        node[below]{\small$\mathsf{u}(-1)$}
        (15.5,-0.02) circle (1.6pt) 
        node[below]{\small$\mathsf{u}(1)$};
\end{tikzpicture}

    % \vspace{-10pt}
    \caption{The arcs whose embedded commutators form a degree $2$ grasper surgery on $\u$, for $d=4$. The meridian 2-spheres $m_j(\S^{d-2})$ are drawn schematically, but should in fact not be contained in the depicted 3-dimensional slice (they are normal spheres to $\u$ and do not intersect it).}
    \label{fig:realmap-2}
\end{figure}
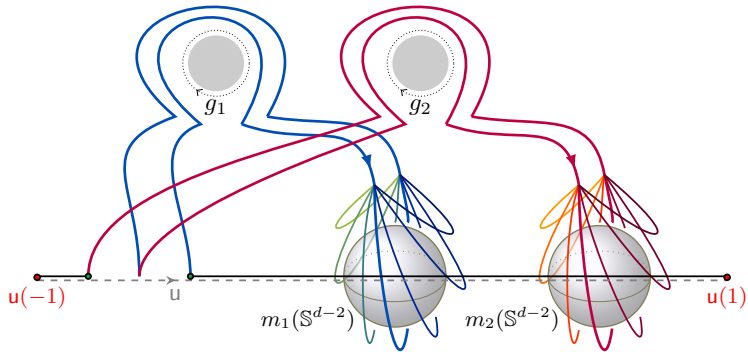
If we instead use the tree with the other tree, then in Figure~\ref{fig:realmap-2} we just need to swap the order in which the red and blue arcs attach to $J_0$. However, by Corollary~\ref{cor:main} the $(\AS)$ relation tells us that
\[
    \realmap_2\left(\igtree{2}{1}{g_2}{g_1}\right)=-
    \realmap_2\left(\igtree{1}{2}{g_1}{g_2}\right),
\]
i.e.\ $\realmap_2$ descends to $\Lie_{\pi_1M}(2)$.
The origins of this are in the antisymmetry of the Samelson bracket. 

In particular, $\pi_{2(d-3)}(\Embp(\D^1,\D^d),\u)=\Lie(2)=\Z$ since $\stusq$ relations are empty in this degree.

\subsubsection*{Degree three}
For $n=3$ there are twelve nondecorated trees in $\Tree(3)$. 
Using the relations $(\AS)$ we can reduce to three of them, which are then related by the Jacobi relation $(\IHX)$:
\[\begin{tikzpicture}[baseline=2ex,scale=0.35,every node/.style={scale=0.8}]
        \clip (-2.5,-0.6) rectangle (2.5,4.45);
        \draw (-0.15,-0.2) rectangle (0.15,0);
        \draw[thick]
            (0,0) -- (0,1) -- (-1,2) -- (-2,3) node[pos=1,above]{$3$}
                    (1,2) -- (0,3) node[pos=1,above]{$2$}
                    (0,1) -- (2,3)  node[pos=1,above]{$1$};
\end{tikzpicture} =
    \begin{tikzpicture}[baseline=2ex,scale=0.35,every node/.style={scale=0.8}]
        \clip (-2.5,-0.6) rectangle (2.5,4.45);
        \draw (-0.15,-0.2) rectangle (0.15,0);
        \draw[thick]
            (0,0) -- (0,1) -- (-1,2) -- (-2,3) node[pos=1,above]{$3$}
                              (-1,2) -- (0,3) node[pos=1,above]{$2$}
                    (0,1) -- (2,3)  node[pos=1,above]{$1$};
\end{tikzpicture}
+ 
\begin{tikzpicture}[baseline=2ex,scale=0.35,every node/.style={scale=0.8}]
        \clip (-2.5,-0.6) rectangle (2.5,4.45);
        \draw (-0.15,-0.2) rectangle (0.15,0);
        \draw[thick]
            (0,0) -- (0,1) -- (-1,2) -- (-2,3) node[pos=1,above]{$2$}
                    (1,2) -- (0,3) node[pos=1,above]{$3$}
                    (0,1) -- (2,3)  node[pos=1,above]{$1$};
\end{tikzpicture}.
\]
Therefore, $\Lie_{\pi_1M}(3)=\Z^2[(\pi_1M)^3]$. An obvious generalisation of Figure~\ref{fig:realmap-2} gives rise to
\[
    \realmap_3\left(\igTree{1}{2}{3}{g_1}{g_2}{g_3}\right)\in\pi_{3(d-3)}(\Embp(\D^1,M),\u).
\]
The fact that $\realmap_3$ vanishes on $(\IHX)$ has origins in the Jacobi relation for the Samelson bracket.

In particular, we have $\Lie(3)=\Z^2$. However, there is now a nontrivial $\stusq$ relation between the two generators, so we have $\pi_{3(d-3)}(\Embp(\D^1,\D^d),\u)=\Z$ (see Corollary~\ref{cor:Dd}).
\begin{rem}\label{rem:graded-signs}
    One can prove $(\AS)$ and $(\IHX)$ for the Samelson bracket by imitating the proof that such relations hold for the commutator bracket in the associated graded of the lower central series of a group; see~\cite[Sec.X.5]{Whitehead} for this proof. Note that the relations for the Samelson bracket have graded signs. In particular, the homotopy groups of the wedge sum of $n$ spheres form a free \emph{graded} Lie algebra $\mathbb{L}_n$ on $n$ letters.

    Alternatively, one could give a direct geometric argument for why $\realmap_n$ vanishes on \emph{nongraded} $(\AS)$ and $(\IHX)$. We save ourselves from doing this by applying the Convergence Theorem at the end of the proof in Section~\ref{subsec:main-proof-arc}. However, see \cite{CST} for a proof in dimension $d=3$ of a similar result. 

    The last two paragraphs are not in contradiction, since the subgroup $\Lie_d(n)$ of $\mathbb{L}_n$, consisting of the words in which each letter appears exactly once, is isomorphic to the group $\Lie(n)$ (of nondecorated trees modulo nongraded $(\AS)$ and $(\IHX)$ as in Section~\ref{subsec:trees}), see \cite[Lem.2.3]{K-thesis-paper}. This is also compatible with our inductive definition of Samelson brackets, see~\cite[App.B]{K-thesis-paper}.
\end{rem}

%~~~~~~~~~~~~~~~~~~~~~~~~~~~~~~~~~~~~~~
\section{Preliminaries}\label{sec:preliminaries}

In Sections~\ref{subsec:notation-gen} and~\ref{subsec:notation-other} we fix some notation. In Section~\ref{subsec:model-ball} we define the extended ball $\ball^d_{ext}$, the model ball $\ball^d_n\coloneqq(\ball^d_{ext},a_0,\bigsqcup_{i=1}^n a_i)$, and the meridians $m\colon\bigvee_n\ball^{d-1}\hra\ball^d_n$. In Section~\ref{subsec:foliations} we define the maps $\mu_d$ and $\mu^\star_d$, and in Section~\ref{subsec:trees} decorated trees $\Gamma^{g_{\ul{n}}}$ and the group $\Lie_{\pi_1M}(n)$. 

%~~~~~~~~~~
\subsection{Notation: general}\label{subsec:notation-gen}

Throughout the paper we fix a smooth $d$-manifold $M$ with $d\geq4$, and $C$ equal either to $\D^1$ or $\S^1$. The space $\Embp(C,M)$ consists of neat smooth embeddings of $C$ into $M$ with a fixed boundary condition $\u|_{\partial C}$ (which is empty if $C=\S^1$), where $\u\colon C\hra M$ is an arbitrary basepoint. We often identify embeddings with their images. 

For spaces $B_i\subseteq A_i\subseteq X_i$, $i=1,2$, the space $\Map((X_1,A_1,B_1),(X_2,A_2,B_2))$ consists of maps $X_1\to X_2$ which take $A_1$ to $A_2$ and $B_1$ to $B_2$. And similarly for embeddings. For a point $x_2\in X_2$ we often write $X_1\times x_2\subseteq X_1\times X_2$, omitting the parentheses for the one-point set.

We denote by $X_1\wedge X_2\coloneqq{X_1\times X_2}/{X_1\vee X_2}$ the smash product. Note that $\S^1\wedge\S^{d-3}$ is the reduced suspension. We have the based loop space $\Omega X\coloneqq\Map((\S^1,e),(X,*_X))$ and the adjunction
\begin{equation}\label{eq:adjunction}
    \Map((\S^1\wedge X_1,e\wedge x_1),(X_2,x_2))\cong\Map((X_1,x_1),(\Omega X_2,\const_{x_2})),
\end{equation}
given by $f\mapsto (x\mapsto (\theta\mapsto f(\theta\wedge x)))$.
% On several occasions we will have a map to a mapping space, which lands in the subspace of embeddings except for a single point in the image. We will write such a target as
% \[
%     \Emb^\star(X_1,X_2),
% \]
% where $\star$ stands for this exceptional point, which will be clear from the context.

%~~~~~~~~~~
\subsection{Notation: balls and spheres}\label{subsec:notation-other}

We use the notation $\I\coloneqq[0,1]$ and $\D^1\coloneqq[-1,1]$. For $-1<L_i<R_i<L_{i+1}$ such that the sequence $R_i$ converges to $R_\infty<1$ we fix a collection of disjoint closed subintervals
\begin{equation}\label{eq-def:J-i}
    J_i=[L_i,R_i]\subseteq \D^1
\end{equation}
Moreover, we let $J_i\subseteq\faktor{\D^1}{\partial}\cong\S^1$.

For a finite set $S$ let $\I^S=\Map(S,\I)$ be the $|S|$-dimensional cube, with coordinates labelled by $S$. Faces of the cube which have a coordinate equal to 0 or 1 will be called \emph{$0$- and $1$-faces} respectively. Similarly, let $\Delta^S$ be the simplex spanned by the set $S$; this means that $\Delta^S\cong\Delta^{|S|-1}$ has vertices labelled by $i\in S$, edges by $\{i,j\}\subseteq S$ and so on. For $n\geq1$ denote $\ul{n}\coloneqq\{1,2,\dots,n\}$.

We write $\vec{t}\in\I^S$, and $t_i\in\I$ for its coordinate that corresponds to $i\in S$. Similarly, we write $\vec{\tau}\in\S^{n(d-3)}\cong\bigwedge_n\S^{d-3}$,
with $\tau_i\in\S^{d-3}$ as the coordinate in the $i$-th factor of the smash product. Moreover, if $n=n_1+n_2$ we fix homeomorphisms
\begin{equation}\label{eq:smash}
    \S^{n(d-3)}\cong\S^{n_1(d-3)}\wedge\S^{n_2(d-3)}\cong\faktor{\I^{n_1(d-3)}\times\I^{n_1(d-3)}}{\partial}.
\end{equation}
On the other hand, sometimes we identify $\S^d$ with the quotient of the cube $(\D^1)^d$ by its boundary; we fix such an identification once and for all. Then, using $\I\coloneqq[0,1]\subseteq[-1,1]\eqqcolon\D^1$ we also have a canonical map $\I^d\sra\I^d/\text{1-faces}\hra\S^d$.

%%%
Next, for any $d\geq2$ we consider the ball
\[
    \ball^d\coloneqq\{(x_1,\dots,x_d)\in\R^d:x_1^2+\dots+x_d^2\leq 1\}.
\]
Let $e^+_d\coloneqq(1,\vec{0})$ and $e^-_d\coloneqq(-1,\vec{0})\in\partial\ball^d$ be the north and south poles (since in our pictures $x_1$ is the vertical axis, see Figure~\ref{fig:balls}), and take $e^+_d$ as the basepoint of $\ball^d$.  We fix the inclusion
\[
    \ball^{d-1}\cong\{x_d=0\}\subseteq\ball^d.
\]
Since $e^\pm_{d-1}$ is identified with $e^\pm_d$, we will often simply write $e^\pm$.

\begin{figure}[!htbp]
    \centering
    \begin{tikzpicture}[scale=1.3,every node/.style={scale=0.7}]
\clip (-1.1,-1.2) rectangle (1.1,1.2);
    \draw[->] 
        (0,0) -- (0,0.4) node[left]{$x_1$};
    \draw[->] 
        (0,0) -- (0.4,0) node[below, pos=1.1]{$x_d$};
`   \draw[->] 
        (0,0) -- (-0.2,-0.25) node[left, pos=0.9]{$x_2$};
        
    \draw[black] (-1,0) arc (180:360:1cm and 0.5cm);
    \draw[black,dotted] (-1,0) arc (180:0:1cm and 0.5cm);
    \shade[ball color=blue!20!,opacity=0.3] (0,0) circle (1cm);
    \draw[black,semithick] (0,0) circle (1cm);
    \fill (0,0) circle (0.5pt) ;

    \draw[blue]
        (0,1) to[out=-170,in=100,distance=0.55cm] (-0.7,-0.358) ;
    \draw[blue,densely dashed]
        (0,1) to[out=-5,in=100,distance=0.3cm] 
        (0.65,0.37);

    \draw[red]
        (0,-1) to[out=170,in=-80,distance=0.4cm] (-0.7,-0.358) ;
    \draw[red,densely dashed]
        (0,-1) to[out=5,in=-80,distance=0.5cm] 
        (0.65,0.37);
        
    \fill
        (0,1)   circle (0.5pt) 
                node[above=-2pt]{\small$e^+$}
        (0,-1)  circle (0.5pt) 
                node[below=-2pt]{\small$e^-$}
        (-0.7,-0.358) circle (0.5pt)
        (0.65,0.37) circle (0.5pt)
        (-0.5,0.1) node[blue]{$\mathbb{S}^1_+$}
        (-0.3,-0.75) node[red]{$\mathbb{S}^1_-$};

\end{tikzpicture}
    \caption{The ball $\ball^d$ for $d=3$.}
    \label{fig:balls}
\end{figure}
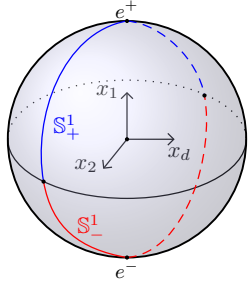
Let $\S^1_+\coloneqq\{(x_1,x_2)\in\partial\ball^2:x_1\geq0\}$ and $\S^1_-\coloneqq\{(x_1,x_2)\in\partial\ball^2:x_1\leq0\}$ be the northern, respectively southern, semicircles, which we view as arcs in boundaries of balls of all dimensions $d\geq2$, using our inclusions $\S^1_+\subseteq\ball^2\subseteq\ball^3\dots\subseteq\ball^d$. They also have the basepoint $e^+$.

%~~~~~~~~~~
\subsection{The model ball}\label{subsec:model-ball}
Fix $d\geq3$ and $n\geq1$. Firstly, we define the \emph{extended ball} as the space
\[
    \ball^d_{ext}\cong\ball^d\cup (\ball^{d-3}\times\D^1\times\D^1\times\D^1)
\]
where we glue $\vec{0}\times-1\times\D^1\times0$ to $\S^1_+\subseteq\partial\ball^d$, and the face $\ball^{d-3}\times-1\times\D^1\times\D^1$ to a tubular neighbourhood of $\S^1_+\subseteq\partial\ball^d$, so that the last factor $\D^1$ is orthogonal to the meridian sphere $\partial\ball^{d-1}\subseteq\partial\ball^d$ (and so $\ball^{d-3}$ is tangent to it); see Figure~\ref{fig:ball-extended}. We refer to the part $\ball^{d-3}\times\D^1\times\D^1\times\D^1$ as the \emph{extension}.

\begin{rem}\label{rem:smooth}
    We have opted for defining the extension as a thickened cube instead a ball for the ease of describing the gluing later on. This gives a space $\ball^d_{ext}$ with specified coordinates which are not smooth. Moreover, the curves and bands that we construct in this space (in Definition~\ref{def:emb-commutator}, for example) will pass through some of these corners. 
    However, see Remarks~\ref{rem:smooth-a0} and~\ref{rem:smooth-emb-comm}.
\end{rem}
\begin{figure}[!htbp]
    \centering
\begin{minipage}[c]{.38\linewidth}
    \flushright
    \begin{tikzpicture}[scale=1.3,every node/.style={scale=0.8}]
    \clip (-1.15,-1) rectangle (1.2,1);
    \draw[->,black!50] 
        (-0.4,-0.15) -- (-0.4,0.75) node[left]{$x_1$};
    \draw[->,black!50] 
        (-0.4,-0.15) -- (-1,-0.75) node[below]{$x_2$};
    \draw[->,black!50] 
        (-0.4,-0.15) -- (0.9,-0.15) node[right]{$x_3$};

    \draw
        (-0.25,-0.25) rectangle (0.75,0.75);
    \fill[blue!10,fill opacity=0.6,draw=black]
        (-0.75,-0.75) -- (0.25,-0.75) -- 
        (0.75,-0.25) -- (-0.25,-0.25) -- (-0.75,-0.75);
    \fill[red!10,fill opacity=0.6,draw=red,-latex]
          (0.2,-0.25) -- (-0.3,-0.75) -- (-0.3,0.25) --
         (0.2,0.75) -- (0.2,-0.25);
    \fill[white,fill opacity=0,draw=black]
        (-0.75,0.25) -- (0.25,0.25) -- 
        (0.75,0.75) -- (-0.25,0.75) -- (-0.75,0.25);

    \draw (-0.75,-0.75) rectangle (0.25,0.25);

    \draw  (0.1,0) node[red]{$a_0$};
    \draw[blue,-latex] (0.2,-0.25) -- (-0.3,-0.75) node[right,pos=0.4,below,scale=0.65,rotate=45]{$-1\times\mathbb{D}^1\times0\,$}; 
\end{tikzpicture}

\begin{tikzpicture}[scale=1.3,every node/.style={scale=0.8},even odd rule]
\clip (-1.15,-1.25) rectangle (1.15,1.2);
    \draw[black] (-1,0) arc (180:360:1cm and 0.5cm);
    \draw[black,dotted] (-1,0) arc (180:0:1cm and 0.5cm);
    \shade[ball color=blue!20!,opacity=0.3] (0,0) circle (1cm);
    \draw[black,semithick] (0,0) circle (1cm);
    \fill (0,0) circle (0.5pt) ;

    \fill[blue!20,opacity=0.6]
        (0.5,0.45) to[in=0,out=100,distance=0.3cm]  (-0.2,0.98) to[out=-170,in=100,distance=0.5cm] (-0.85,-0.26) to[out=-20,in=170,distance=0.1cm] (-0.55,-0.4)
        to[in=-170,out=100,distance=0.5cm] (0.2,0.98)  to[out=0,in=100,distance=0.3cm]  (0.8,0.3) to[in=-20,out=170,distance=0.1cm] (0.5,0.45) ;
            
    \draw[blue]
        (0,1) to[out=-170,in=100,distance=0.55cm] (-0.7,-0.358) ;
    \draw[blue,densely dashed]
        (0,1) to[out=-5,in=100,distance=0.3cm] 
        (0.65,0.37);

    \fill
        (0,1)   circle (0.5pt) 
                node[above=-1pt]{\small$e^+$}
        (0,-1)  circle (0.5pt) 
                node[below=-2pt]{\small$e^-$}
        (-0.7,-0.358) circle (0.5pt)
        (0.65,0.37) circle (0.5pt)
        (-0.5,0.1) node[blue]{$\mathbb{S}^1_+$}
        ;
    
\end{tikzpicture}
\end{minipage}
$\qquad\rightsquigarrow\qquad$
\begin{minipage}[c]{.17\linewidth}
    \centering
    \begin{tikzpicture}[scale=1.3,every node/.style={scale=0.7}]
\clip (-0.8,-2) rectangle (0.8,0.8);
    \draw
        (-0.25,-0.25) rectangle (0.75,0.75);
    \fill[blue!10,fill opacity=0.6,draw=black]
        (-0.75,-0.75) -- (0.25,-0.75) -- 
        (0.75,-0.25) -- (-0.25,-0.25) -- (-0.75,-0.75);
    \fill[red!10,fill opacity=0.6,draw=red,-latex]
          (0.2,-0.25) -- (-0.3,-0.75) -- (-0.3,0.25) --
         (0.2,0.75) -- (0.2,-0.25);
    \fill[black!15,fill opacity=0.4,draw=black,semithick]
        (-0.75,0.25) -- (0.25,0.25) -- 
        (0.75,0.75) -- (-0.25,0.75) -- (-0.75,0.25);

    \fill[black!15,fill opacity=0.4] (-0.75,-0.75) rectangle (0.25,0.25);

    \draw  (0.1,0) node[red]{$a_0$};
    \draw[blue,-latex] (0.2,-0.25) -- (-0.3,-0.75) node[right,pos=0.4,below,scale=0.5,rotate=45]{$-1\times\mathbb{D}^1\times0\,$}; 

    \fill[black!15,fill opacity=0.8]
        (0.25,-0.75) to[out=-90,in=-90,distance =1.8cm] (0.75,-0.25)
        -- (0.25,-0.75);
    \draw[semithick]
        (-0.75,-0.75) -- (-0.75,0.25) 
        (0.25,0.25) -- (0.25,-0.75) to[out=-90,in=-90,distance =1.8cm] (0.75,-0.25) -- (0.75,0.75);
    \draw[dashed]
        (-0.75,-0.75) to[out=-90,in=-90,distance =1.8cm] (-0.25,-0.25);
     \fill[black!15,fill opacity=0.8]
        (-0.75,-0.75) to[out=-90,in=-180,distance =0.2cm] (-0.55,-1.88)
        -- (0.45,-1.88) to[in=-90,out=-180,distance =0.2cm] (0.25,-0.75) 
        -- (-0.75,-0.75);
     \draw[semithick]
        (-0.75,-0.75) to[out=-90,in=-180,distance =0.2cm] (-0.55,-1.88)
        -- (0.45,-1.88) to[in=-90,out=-180,distance =0.2cm] (0.25,-0.75) 
        ;
    \fill[black!15,fill opacity=0.4]
        (0.25,0.25) -- (0.25,-0.75) -- (0.75,-0.25) -- (0.75,0.75) --(0.25,0.25) ;

\end{tikzpicture}
\end{minipage}
    \caption{The extended $d$-ball $\ball^d_{ext}$ for $d=3$ is the union of the cube and the ball as on the left, glued along blue strips. We depict it as on the right.}
    \label{fig:ball-extended}
\end{figure}
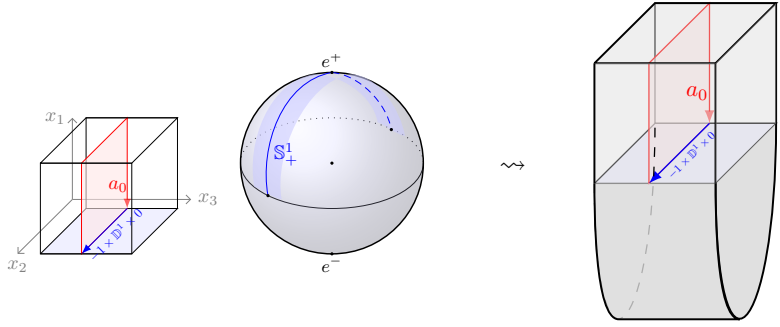

Moreover, as in Figure~\ref{fig:ball-extended} we define the \emph{arc}
\begin{equation}\label{eq-def:a-0}
     a_0\colon\D^1\hra\partial\ball^d_{ext}
\end{equation}
to have the image $\vec{0}\times(\D^1\times1\,\cup\,1\times\D^1\,\cup\,\D^1\times-1)\times0$, and its \emph{strip} neighbourhood
\begin{equation}\label{eq-def:A-0}
     A_0\colon\D^1\times[-\e,\e]\hra\partial\ball^d_{ext}
\end{equation}
to have the image $\vec{0}\times(\D^1\times1\,\cup\,1\times\D^1\,\cup\,\D^1\times-1)\times\D^1$. Finally, define the \emph{circle} $c_0\colon\S^1\hra\ball^d_{ext}$
as the union of $a_0$ and $\vec{0}\times-1\times\D^1\times0$ (so $c_0$ is the boundary of the middle vertical square), and $C_0\colon\D^1\times[-\e,\e]\hra\partial\ball^d_{ext}$ as its \emph{annulus} neighbourhood.
\begin{rem}\label{rem:smooth-a0}
    Note that $a_0(\D^1)\subseteq\partial\ball^d_{ext}\subseteq\R^d$ is not smooth as it passes through the corners of the cube $(\D^1)^3$. However, we will often write
    \[
         \Emb((\S^1,\S^1_+),(\ball^d_{ext},a_0))
    \]
    where we use the following convention: a map belongs to this space if its a smooth embedding apart from the corners of $a_0$. This is topologised using any homeomorphism $\ball^d_{ext}\cong\ball^d$ which identifies $a_0$ with a smooth curve.
\end{rem}

Next, for $i\geq 1$ we choose a collection of \emph{neat} arcs $a_i\colon(\D^1,\partial\D^1)\hra(\ball^d_{ext},\partial\ball^d_{ext})$ that miss the extension and are smoothly neatly embedded in the part $\ball^d\subseteq\ball^d_{ext}$. We also pick their small tubular neighbourhoods $\nu(a_i)$. Finally, let the \emph{model ball of degree $n$} be the tuple
\begin{equation}\label{eq-def:ball-d-n}
    \ball^d_n\coloneqq \big(\ball^d_{ext},a_0,\bigsqcup_{i=1}^n a_i\big).
\end{equation}
% We define the boundary of the model ball as $\partial\ball^d_n\coloneqq (\partial\ball^d_{ext})\sm\bigsqcup_{i=1}^n\nu(a_i)|_{\pm1}$, the $2n$-times punctured $(d-1)$-sphere.
% Moreover, by an embedding 
% \[
%     f\colon(\D^2,\partial)\hra(\ball^d_n,\partial)
% \]
% we mean a neat embedding $f\colon\D^2\hra\ball^d$ such that $f(\partial\D^2)\subseteq\partial\ball^d_n\subseteq\partial\ball^d$, i.e.\ we just require that $f(\partial\D^2)$ misses the $2n$ disks $\bigsqcup_{i=1}^n\nu(a_i)|_{\pm1}$. 
% 
Note that for a fixed $d\geq4$ and $n\geq1$ the tuple $\ball^d_n$ is unique up to isotopy, even relative to the boundary. See Section~\ref{subsec:gropes} for a motivation of this definition.

In particular, for $n=1$ we make a particular choice for $a_1$, so that it has image
\begin{equation}\label{eq:a-1}
    a_1(\D^1)\coloneqq\ball^d\cap x_d\text{-axis}.
\end{equation}
We have a homeomorphism $\ball^d\cong a_1\times\ball^{d-1}$, so we can view $\ball^d_1$ as a tubular neighbourhood of the arc $a_1$. Let $a_1(0)\times\frac{1}{2}\ball^{d-1}$ be a smaller ball, which we call the (oriented) \emph{meridian ball} to the arc $a_1$, and its boundary $a_1(0)\times\partial\frac{1}{2}\ball^{d-1}$ is the \emph{meridian (sphere)}. 
% Restricting to the points $a_i(\pm1)\in\partial\ball^d$ we obtain $\nu(a_i)|_{+1}\colon\ball^{d-1}\hra\partial\ball^d$ called the \emph{(oriented) meridian ball} to $a_i$, and $\nu(a_i)|_{-1}\colon\ball^{d-1}\hra\partial\ball^d$ the \emph{oppositely oriented meridian ball}. 
Similarly, for each $a_i$ we pick a meridian ball $m_i\colon\ball^{d-1}\hra\ball^d_n$ at $a_i(0)$, together with a path, called the \emph{whisker}, to $e^+$. We thicken these whiskers to obtain an inclusion
\begin{equation}\label{eq-def:m}
    m\colon\bigvee_n\ball^{d-1}\hra\ball^d_n.
\end{equation}
% Observe that $\partial\ball^d_n\subseteq\partial(\ball^d_{ext}\sm\bigsqcup_{i=1}^n\nu(a_i))\cong(\S^{d-2}\times\S^1)^{\# n}$, where each $\S^{d-2}\times pt $ corresponds to the oriented meridian sphere $\partial\nu(a_i)_{a_i(1)}$. 
 Note that $\partial m\colon\bigvee_n\S^{d-2}\hra\ball^d\sm\bigsqcup_{i=1}^na_i$ is a homotopy equivalence.

%~~~~~~~~~~
\subsection{Foliations}\label{subsec:foliations}

Pick a homeomorphism $\S^{d-2}\cong\S^1\wedge\S^{d-3}$ with the reduced suspension
and consider its adjoint map (via \eqref{eq:adjunction}):
\begin{equation}\label{eq-def:x}
    x_{d-2}\colon\S^{d-3}\ra\Omega\S^{d-2},
\end{equation}
for which $x_{d-2}(e^+_{d-3})=\const_{e^+_{d-2}}$.

We parametrise $x_{d-2}$ so that it satisfies the following inductive property: the restriction of $x_{d-1}$ to the meridian $\S^{d-3}\subseteq\S^{d-2}$ has image in the meridian $\S^{d-2}\subseteq\S^{d-1}$, and agrees with $x_{d-2}$. Moreover, we can view $x_{d-1}\colon\S^{d-2}\ra\Omega\S^{d-1}$ as a ``suspension'' of $x_{d-2}$: by gluing together two maps of the shape
\begin{equation}\label{eq-def:x-star}
    x^\star_{d-1}\colon\I\times\S^{d-3}\ra\Omega\ball^{d-1}
\end{equation}
each of which homotopes $x_{d-2}$ to $\const_{e^+}$, one using the eastern, the other the western hemisphere of $\S^{d-1}$. Namely, we use $\S^{d-2}\cong\ball^{d-2}_E\cup_\partial\ball^{d-2}_W$ and $\ball^{d-3}_{E/W}\cong\I\times\S^{d-3}/1\times\S^{d-3}$, see Figure~\ref{fig:foliation-x}.
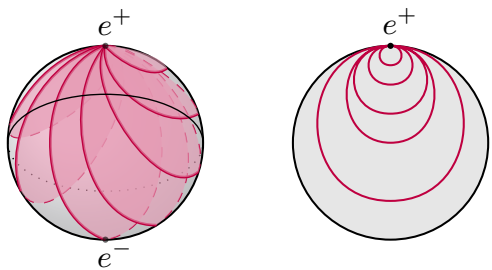
\begin{figure}[!htbp]
    \centering
    \begin{tikzpicture}[fill opacity=0.6]
    \begin{scope}
        \clip (3.95,-1.43) rectangle (6.05,1.43);
        \foreach \y in {5}{
        \shade[ball color=blue!3!white,opacity=0.4] (\y,0) circle (1cm);
        \draw[dotted] (\y-1,0) arc (180:360:1cm and 0.5cm);
        \draw[semithick] (\y,0) circle (1cm);
        }

        \draw[purple,thick]
            (4.03,0.21) to[in=-175,out=65,distance=0.45cm] (5,1);
        \draw[purple,dashed]
            (5,1) to[out=-40,in=-110,distance=0.57cm] (4.03,0.21);
            \fill[purple!40] 
                (4.03,0.21) to[in=-175,out=65,distance=0.45cm](5,1) to[out=-40,in=-110,distance=0.57cm] (4.03,0.21);
                
        \draw[purple,thick] 
            (5,1) to[out=-170,in=110,distance=0.6cm] (4.135,-0.5);
        \draw[purple,dashed]
            (5,1) to[out=-25,in=-60,distance=0.7cm] (4.135,-0.5);
            \fill[purple!40]
                (4.135,-0.5) to[in=-170,out=110,distance=0.6cm]
                (5,1) to[out=-25,in=-60,distance=0.7cm] (4.135,-0.5);

        \draw[purple,thick] 
            (5,1) to[out=-160,in=160,distance=0.8cm] (5,-1);
        \draw[purple,dashed]
            (5,1) to[out=-20,in=10,distance=0.8cm] (5,-1);
                \fill[purple!40]
                    (5,-1) to[in=-160,out=160,distance=0.8cm]
                    (5,1) to[out=-20,in=10,distance=0.8cm] (5,-1);

        \draw[purple,thick] 
            (5,1) to[out=-130,in=180,distance=0.75cm] (5.52,-0.85);
        \draw[purple,dashed]
            (5,1) to[out=-15,in=35,distance=0.8cm] (5.52,-0.85);
                \fill[purple!40]
                    (5.52,-0.85) to[in=-130,out=180,distance=0.75cm]  
                    (5,1) to[out=-15,in=35,distance=0.8cm] (5.52,-0.85);

        \draw[purple,thick] 
            (5,1) to[out=-120,in=-140,distance=0.5cm] (5.95,-0.3);
        \draw[purple,dashed]
            (5,1) to[out=5,in=70,distance=0.45cm] (5.95,-0.3);
                \fill[purple!40]
                     (5.95,-0.3) to[in=-120,out=-140,distance=0.5cm] 
                     (5,1) to[out=5,in=70,distance=0.45cm] (5.95,-0.3);

        \draw[purple,thick] 
            (5,1) to[out=-100,in=-150,distance=0.4cm] (5.95,0.3);
        \draw[purple,dashed]
            (5,1) to[out=0,in=100,distance=0.25cm] (5.95,0.3);
                \fill[purple!40] 
                (5.95,0.3) to[in=-100,out=-150,distance=0.4cm] 
                (5,1) to[out=0,in=100,distance=0.25cm] (5.95,0.3);
        \draw[purple,thick] 
            (5,1) to[out=-50,in=-140,distance=0.2cm] 
            (5.63,0.77);
        \fill[purple!40] 
            (5,1) to[out=-50,in=-140,distance=0.2cm] 
            (5.63,0.77) to[in=0,out=110,distance=0.15cm] (5,1);

        \draw (4,0) arc (180:0:1cm and 0.5cm);

        \fill 
            (5,-1) circle (0.9pt)
            (5,1) circle (0.9pt);
        \draw[opacity=1]
            (5.1,-1.15) node{\small$e^-$}
            (5.1,1.25) node{\small$e^+$};
    \end{scope}
\end{tikzpicture}~\qquad
\begin{tikzpicture}[rotate=180]
    \clip (3.95,-1.43) rectangle (6.05,1.43);
    \fill[black!10,draw=black,semithick] (5,0) circle (1cm);
    \draw[purple,semithick] 
        (5,-1) to[out=0,in=0,distance=0.15cm] (5,-0.8) to[out=180,in=180,distance=0.15cm] (5,-1)
        (5,-1) to[out=0,in=0,distance=0.3cm] (5,-0.6) to[out=180,in=180,distance=0.3cm] (5,-1)
        (5,-1) to[out=0,in=0,distance=0.5cm] (5,-0.3) to[out=180,in=180,distance=0.5cm] (5,-1)
        (5,-1) to[out=0,in=0,distance=0.6cm] (5,0) to[out=180,in=180,distance=0.6cm] (5,-1)
        (5,-1) to[out=0,in=0,distance=1cm] (5,0.6) to[out=180,in=180,distance=1cm] (5,-1)
        ;
    \fill 
        (5,-1) circle (0.9pt);
    \draw[opacity=1]    
         (4.9,-1) node[above]{\small$e^+$};
\end{tikzpicture}
    \caption{The foliations $x_{d-2}$ of $\S^{d-2}$ and $x^\star_{d-2}$ of $\ball^{d-2}$ respectively, for $d=4$. The latter is the restriction of the former to any of the two hemispheres. Similarly, the foliation $x_{d-1}$ of $\S^{d-1}$ can be obtained by gluing together two copies of $x^\star_{d-1}$ along their boundaries, and $x^\star_{d-1}$ is given as the foliated ball on the left but where each pink disk should be foliated using $x^\star_2$, depicted on the right.}
    \label{fig:foliation-x}
\end{figure}
\begin{figure}[!htbp]
    \centering
    \begin{tikzpicture}[fill opacity=0.6]
    \begin{scope}
        \clip (3.95,-1.35) rectangle (6.05,1.5);
        \foreach \y in {5}{
        \shade[ball color=blue!3!white,opacity=0.4] (\y,0) circle (1cm);
        \draw[dotted] (\y-1,0) arc (180:360:1cm and 0.3cm);
        \draw[semithick] (\y,0) circle (1cm);
        }
        \draw[opacity=1,blue,dashed,thick]
            (5.58,-.22)
            to[out=80,in=-5,distance=0.35cm]
            (5,1);
        \fill[blue]
            (5.58,-.23) circle (0.8pt);
            
            % \fill[purple!40] 
            %     (4.03,0.21) to[in=-175,out=65,distance=0.45cm](5,1) to[out=-40,in=-110,distance=0.57cm] (4.03,0.21);
            \fill[purple!40]
                (4.135,-0.5) to[in=-170,out=110,distance=0.6cm]
                (5,1) to[out=-25,in=-60,distance=0.7cm] (4.135,-0.5);
\draw[purple,dashed]
    (5,1) to[out=-25,in=-60,distance=0.7cm] (4.135,-0.5);
            \fill[purple!40]
                (5,-1) to[in=-160,out=160,distance=0.8cm]
                (5,1) to[out=-20,in=10,distance=0.8cm] (5,-1);
\draw[purple,dashed]
    (5.58,-.22) to[out=-100,in=10,distance=0.38cm] (5,-1); 
                % \fill[purple!40]
                %     (5.52,-0.85) to[in=-130,out=180,distance=0.75cm]  
                %     (5,1) to[out=-15,in=35,distance=0.8cm] (5.52,-0.85);
                \fill[purple!40]
                     (5.95,-0.3) to[in=-120,out=-140,distance=0.7cm] 
                     (5,1) to[out=5,in=70,distance=0.45cm] (5.95,-0.3);
 \draw[purple,dashed]
    (5,1) to[out=5,in=70,distance=0.45cm] (5.95,-0.3);                
                % \fill[purple!40] 
                %     (5.95,0.3) to[in=-100,out=-150,distance=0.4cm] 
                %     (5,1) to[out=0,in=100,distance=0.25cm] (5.95,0.3);
                \fill[purple!40] 
                    (5,1) to[out=-50,in=-140,distance=0.2cm] 
                    (5.63,0.77) to[in=0,out=110,distance=0.15cm] (5,1);

        \draw[purple,thick] 
            (5,1) to[out=-160,in=160,distance=0.8cm] (5,-1);
        \draw[purple,thick] 
            (4.49,0.42) to[out=-110,in=70,distance=0.12cm] 
            (4.2,0.23) to[out=-110,in=110,distance=0.3cm]
            (4.135,-0.5);
        % \draw[purple,thick] 
        %     (4.5,0.4) to[out=-100,in=180,distance=0.5cm] (5.52,-0.85);
        \draw[purple,thick] 
            (4.49,0.42) to[out=-110,in=120,distance=0.2cm]
            (4.9,0.4) to[out=-60,in=-140,distance=0.55cm] (5.95,-0.3);
        % \draw[purple,thick] 
        %     (4.5,0.42) to[out=-80,in=-150,distance=0.4cm] (5.95,0.3);
        \draw[purple,thick] 
            (4.49,0.42) to[out=-110,in=-140,distance=0.05cm] 
            (4.55,0.4) to[out=40,in=-140,distance=0.5cm] 
            (5.63,0.77);

        \draw (4,0) arc (180:0:1cm and 0.5cm);

        \draw[opacity=1,blue,thick]
            (5,1) to[out=-160,in=75,distance=0.25cm]
            (4.49,0.42)
            (4.9,0.73) node{\small$\alpha$} ;
        \fill[blue]
            (4.49,0.42) circle (0.8pt);

        \fill 
            (5,-1) circle (0.9pt)
            (5,1) circle (0.9pt);

        \draw[opacity=1]
            (5.1,-1.15) node{\small$e^-$}
            (5.1,1.25) node{\small$e^+$};
    \end{scope}
\end{tikzpicture}~\qquad
\begin{tikzpicture}[rotate=180]
    \clip (3.95,-1.4) rectangle (6.05,1.34);
    \fill[black!10,draw=black,semithick] (5,0) circle (1cm);
    \draw[purple,semithick] 
        (4,0) to[out=90,in=180,distance=0.5cm] (5,-0.8) to[out=0,in=90,distance=0.5cm] (6,0)
        (4,0) to[out=90,in=180,distance=0.5cm] (5,-0.6) to[out=0,in=90,distance=0.5cm] (6,0)
        (4,0) to[out=90,in=180,distance=0.5cm] (5,-0.3) to[out=0,in=90,distance=0.5cm] (6,0)
        (4,0) to[out=90,in=180,distance=0.5cm] (5,0.1) to[out=0,in=90,distance=0.5cm] (6,0)
        (4,0) to[out=90,in=180,distance=0.5cm] (5,0.5) to[out=0,in=90,distance=0.5cm] (6,0)
        (4,0) to[out=90,in=180,distance=0.5cm] (5,0.8) to[out=0,in=90,distance=0.5cm] (6,0)
        ;
    \draw[blue,semithick]
        (4,0) to[out=-90,in=-90,distance=1.3199cm] (6,0) ;
    \draw[blue]
        (5,-1) node[above]{\small$\alpha$};
    \fill[blue]
        (4,0) circle (0.9pt)
        (6,0) circle (0.9pt);
\end{tikzpicture}
    \caption{The foliation $\mu^\star_d$ of $\ball^{d-1}\subseteq\ball^d$ for $d=4$ is given on the left, but where each pink disk $\ball^2$ should be foliated using $\mu^\star_2$, depicted on the right.}
    \label{fig:foliation-modified}
\end{figure}
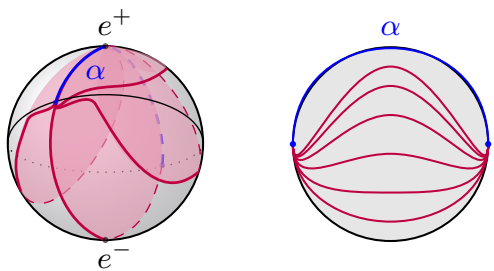

In fact, we will use the following modifications $\mu_d$ and $\mu^\star_d$ of the foliations $x_{d-2}$ and $x^\star_{d-1}$.

We first homotope $x_{d-2}$ so that each circle agrees on $\S^1_+$ with our fixed arc $\alpha\coloneqq\S^1_+\subseteq\S^{d-2}$, see Figure~\ref{fig:foliation-modified}. Note that every circle in the family is embedded, except at the basepoint $e^+_{d-3}$ where we have $\alpha\alpha^{-1}$. 
Secondly, we include this into our extended ball $\ball^d_{ext}$ as the foliation of the meridian sphere $\S^{d-2}=\partial\frac{1}{2}\ball^{d-1}\subseteq\partial\ball^d$ (see the paragraph after \eqref{eq:a-1}), and we replace $\alpha$ in each circle by our arc $a_0$, together with a fixed path connecting their endpoints in a smooth manner. We denote the resulting foliation by $\mu_d$, and note that at the basepoint we have
\[
    \mu_d(e^+_{d-3})=a_0\alpha^{-1}=c_0,
\]
by the definition of $c_0$ in Section~\ref{subsec:model-ball}. Therefore, we have
\begin{equation}\label{eq-def:mu}
    \mu_d\colon\S^{d-3}\ra\Emb((\S^1,\S^1_+),(\ball^d_{ext},a_0)).
\end{equation}
By an analogous modification of $x^\star_{d-1}$, we obtain a map
\begin{equation}\label{eq-def:mu-star}
    \mu^\star_d\colon\I\times\S^{d-3}\ra\Emb((\S^1,\S^1_+),(\ball^d_{ext},a_0)).
\end{equation}
This foliates the meridian ball $\ball^{d-1}\subseteq\ball^d$, so that $\mu^\star_d(0,-)=\mu_{d-1}$ and $\mu^\star_d(1,-)=\const_{c_0}$.

%~~~~~~~~~~
\subsection{Trees}\label{subsec:trees}
Let $\Tree(n)$ be the set of (rooted planar binary) trees with $n$ ordered leaves (nonroot univalent vertices). Define the set of \emph{decorated trees} as the product
\[
    \Tree_{\pi_1M}(n)\coloneqq \Tree(n)\times(\pi_1M)^n.
\]
We view a $\pi_1M$-decorated tree $\Gamma^{g_{\ul{n}}}\in\Tree_{\pi_1M}(n)$ as a tree $\Gamma\in\Tree(n)$ of \emph{degree} $\deg\Gamma=n$, together with the decoration $g_i\in\pi_1M$ for the $i$-th leaf, $1\leq i\leq n$.

We can assign to a tree $\Gamma$ a bracketed word with $n$ distinct letters $\mathsf{x}_i$ (corresponding to the leaves) and $n-1$ brackets (corresponding to trivalent vertices). Therefore, we will often write 
\[
    \Gamma=[\Gamma_1,\Gamma_2]
\]
to mean that the tree $\Gamma\in\Tree(n)$ is obtained by gluing together the roots of $\Gamma_j\in\Tree(S_j)$, where $S_1\sqcup S_2=\ul{n}$, and sprouting a new root. By $\Tree(S)$ we mean trees whose leaves are labelled by the elements of the set $S$.

We define $\Lie_{\pi_1M}(n)$ as the quotient of the free abelian group $\Z[\Tree_{\pi_1M}(n)]$ by the relations $(\AS)$ and $(\IHX)$ locally given as follows:
\[\begin{aligned}
    $
\mathsf{AS}:\begin{tikzpicture}[baseline=2ex,scale=0.45,every node/.style={scale=0.85}]
        \clip (-1.7,-0.6) rectangle (1.7,3.6);
        \draw
            (0,0) node[]{$\vdots$};
        \draw[thick]
            (0,0) -- (0,1) -- (-1.25,2.25) node[pos=1,above]{$\Gamma_2$}
                    (0,1) -- (1.25,2.25)  node[pos=1,above]{$\Gamma_1$};
\end{tikzpicture} 
+ 
\begin{tikzpicture}[baseline=2ex,scale=0.45,every node/.style={scale=0.85}]
        \clip (-1.7,-0.6) rectangle (1.7,3.6);
        \draw (0,0) node[]{$\vdots$};
        \draw[thick]
            (0,0) -- (0,1) -- (-1.25,2.25) node[pos=1,above]{$\Gamma_1$}
                    (0,1) -- (1.25,2.25)   node[pos=1,above]{$\Gamma_2$};
\end{tikzpicture} 
=\: 0,\quad
\mathsf{IHX}:\begin{tikzpicture}[baseline=2ex,scale=0.35,every node/.style={scale=0.85}]
        \clip (-2.5,-0.6) rectangle (2.5,4.45);
        \draw (0,0) node[]{$\vdots$};
        \draw[thick]
            (0,0) -- (0,1) -- (-1,2) -- (-2,3) node[pos=1,above]{$\Gamma_3$}
                              (-1,2) -- (0,3) node[pos=1,above]{$\Gamma_2$}
                    (0,1) -- (2,3)  node[pos=1,above]{$\Gamma_1$};
\end{tikzpicture} 
-
\begin{tikzpicture}[baseline=2ex,scale=0.35,every node/.style={scale=0.85}]
        \clip (-2.5,-0.6) rectangle (2.5,4.45);
        \draw (0,0) node[]{$\vdots$};
        \draw[thick]
            (0,0) -- (0,1) -- (-1,2) -- (-2,3) node[pos=1,above]{$\Gamma_3$}
                    (1,2) -- (0,3) node[pos=1,above]{$\Gamma_2$}
                    (0,1) -- (2,3)  node[pos=1,above]{$\Gamma_1$};
\end{tikzpicture} 
+ 
\begin{tikzpicture}[baseline=2ex,scale=0.35,every node/.style={scale=0.85}]
        \clip (-2.5,-0.6) rectangle (2.5,4.45);
        \draw (0,0) node[]{$\vdots$};
        \draw[thick]
            (0,0) -- (0,1) -- (-1,2) -- (-2,3) node[pos=1,above]{$\Gamma_2$}
                    (1,2) -- (0,3) node[pos=1,above]{$\Gamma_3$}
                    (0,1) -- (2,3)  node[pos=1,above]{$\Gamma_1$};
\end{tikzpicture} 
=\: 0.
$
\end{aligned}
\]
We have that $\Lie_{\pi_1M}(n)\cong\Lie(n)\otimes\Z[(\pi_1M)^n]$, where $\Lie(n)=\Lie_{\{1\}}(n)$ is the subgroup of the free Lie algebra on $n$ letters consisting of Lie words written with exactly $n$ distinct letters. 

\begin{rem}
    We thank the referee for pointing out that usually (for example in the operadic construction of free algebras), one additionally mods out the diagonal action of the symmetric group on $n$ letters. In our case, $C=\I$ is the interval or $C=\S^1$ is the based circle, so that subintervals of $C$ can be ordered, giving a canonical order for the leaves of a tree.
\end{rem}

%~~~~~~~~~~~~~~~~~~~~~~~~~~~~~~~~~~~~~~~~~~~~~~~
\section{The realisation map}\label{sec:realmap}

In Section~\ref{subsec:realisation-map} we define a grasper $\TG_n$ modelled on a tree $\Gamma\in\Tree(n)$, and surgery on $\TG_n$, which produces a knotted family $\realmap_n(\Gamma^{g_{\ul{n}}})$ of arcs or circles in a $d$-manifold $M$ (this is Step 2 of Section~\ref{subsec:overview}). The surgery uses the family
\[
    \mu_{\Gamma,d}\coloneqq\mu^\star_{\Gamma,d}(\vec{0},-)\colon\S^{n(d-3)}\ra \Emb((\S^1,\S^1_+),(\ball^d_n,a_0)),
\]
which will be constructed in the next Section~\ref{sec:multi-family}, see~\eqref{eq-def:mu-Gamma} (this is Step 1 of Section~\ref{subsec:overview}). We opted for this reversed exposition in order to clarify the motivation behind the somewhat involved definition of $ \mu_{\Gamma,d}$. To this end, in Section~\ref{subsec:gropes} we explain the background inspiration coming from gropes and claspers, that gave rise to the name ``grasper'', and in Section~\ref{subsec:Haefliger} we explain how the Haefliger trefoil is also a geometric analogue of a Samelson product.

%~~~~~~~~~~
\subsection{Graspers and knotted families}\label{subsec:realisation-map}

Recall that $C=\D^1$ or $C=\S^1$, and $J_i\subseteq C$ are fixed intervals from~\eqref{eq-def:J-i}, and the model ball $\ball^d_n=(\ball^d_{ext},a_0,\bigsqcup_{i=1}^na_i)$ was defined in~\eqref{eq-def:ball-d-n}.
\begin{defn}[Grasper]\label{def:grasper}
    For a smooth manifold $M$ of dimension $d\geq3$ and a smooth embedding $\u\colon C\hra M$, a \emph{grasper} of degree $n$ in $M$ relative to $\u$, written
    \[
        \TG_n\colon\ball^d_n\hra M,
    \]
    is a smooth embedding of $\ball^d$ into $M$ such that
    $\TG_n(a_0)=\u(J_0)$ and $\TG_n(a_i)=\u(J_i)$ and otherwise $\TG_n(\ball^d)$ is disjoint from $\u(C)$. 
\end{defn}
\begin{defn}[Underlying decoration]\label{def:under-decoration}
    For a grasper $\TG_n\colon\ball^d_n\hra M$ and $1\leq i\leq n$ denote by $p_i\in C$ the point such that $\u(p_i)=\TG_n(a_i(0))$. We define $g_i\in\pi_1M$ as the homotopy class of the loop given by $\u|_{[0,p_i]}$ followed by the image under $\TG_n$ of the whisker for the meridian ball to $a_i$. We say that the tuple $g_{\ul{n}}\coloneqq(g_1,\dots,g_n)\in(\pi_1M)^n$ \emph{decorates} the grasper $\TG_n$.
\end{defn}
We use graspers and the map $\mu_{\Gamma,d}$
from~\eqref{eq-def:mu-Gamma} to define knotted families of arcs or circles in an arbitrary $d$-manifold $M$. 
Namely, for a tree $\Gamma\in\Tree(n)$ we have a family of arcs
\[\begin{tikzcd}[column sep=large]
        \S^{n(d-3)}\rar{\mu_{\Gamma,d}} & 
        \Emb(\S^1,\ball^d_n)\rar{\TG_n\circ-} & 
        \Emb(\S^1,M)\arrow{r}{-|_{\S^1_-}} &
        \Embp(\D^1,M).
    \end{tikzcd}
\]
All these arcs have the same endpoints as the subarc $\u|_{J_0}$ of $\u$, and are disjoint from $\u|_{C\sm J_0}$, so they can be used to replace $\u|_{J_0}$.
\begin{defn}[Grasper surgery]\label{def:knotted-family}
    We say that the family of embeddings
    \begin{align*}
       \u_{\TG,\Gamma^{g_{\ul{n}}}}\colon\S^{n(d-3)} & \to\Embp(C,M),\\ 
        \vec{\tau} & \mapsto\u|_{C\sm J_0}\cup(\TG_n\circ\mu_{\Gamma,d}(\vec{\tau}))|_{\S^1_-}
    \end{align*}
    is the \emph{result of surgery of type $\Gamma$ along the grasper $\TG_n$ on the trivial family $\const_\u$}.
    
    Moreover, we define a homomorphism 
    \[
    \realmap_n\colon
    \Z[\Tree_{\pi_1M}(n)]\ra\pi_{n(d-3)}\Embp(C,M)
    \]
    by linearly extending the map sending a decorated tree $\Gamma^{g_{\ul{n}}}\in\Tree_{\pi_1M}(n)$ to the homotopy class $[\u_{\TG,\Gamma^{g_{\ul{n}}}}]$,
    for any grasper $\TG_n$ of degree $n$ in $M$ relative to $\u$ and decorated by $g_{\ul{n}}$.
\end{defn}
\begin{rem}\label{rem:meridians-agree}
    Observe that the meridian balls $m\colon\bigvee_n\ball^{d-1}\hra\ball^d_n$ from~\eqref{eq-def:m} precisely become (based) meridians to the arc $\u$ at the points $\u(p_i)=\TG_n(a_i(0))$.
\end{rem}
    
\begin{rem}\label{rem:arc-to-circle}
    If $\u\colon\S^1\hra M$ and $\D^d\subseteq M$ is a small ball containing $\u|_{[R_\infty,L_0]}$, there is an inclusion
    \[
        \Embp(\D^1,M\sm\D^d)\hra\Emb(\S^1,M)
    \]
    which adds the ball together with $\u|_{[R_\infty,L_0]}$. Then graspers on the circle $\u$ clearly correspond to graspers on the arc $\u|_{[L_0,R_\infty]}$, so the realisation map factors as
    \[
    \realmap_n\colon
    \Z[\Tree_{\pi_1M}(n)]\ra\pi_{n(d-3)}\Embp(\D^1,M\sm\D^d)\ra\pi_{n(d-3)}\Emb(\S^1,M).\qedhere
    \]
\end{rem}
\begin{rem}\label{rem:nonconn-graspers}
    This construction clearly extends to any $1$-manifold $C$, that is, not necessarily connected. The grasper can have its root and leaves on various components.
\end{rem}

%~~~~~~~~~~
\subsection{Relation to gropes and claspers}\label{subsec:gropes}

    In dimension $d=3$ a grasper $\TG_n\colon\ball^3\hra M$ is essentially determined by a choice of a capped grope cobordism $\G\colon G_\Gamma^c\hra M$ relative to $\u$. Namely, $\TG_n$ is a tubular neighbourhood of $\G$ and the arc $a_i$ is dual to the $i$-th cap of $\G$, see Figure~\ref{fig:grope-thick}. 
    The map $\mu_{\Gamma,3}\colon\S^0\to\Emb(\S^1,\ball^3_n)$ sends the basepoint $e^+$ to $c_0$ and $e^-$ to the standard circle $\S^1=\partial\ball^2\subseteq\ball^d$. The surgery on the grasper is the knot 
    \[
        \u_{\TG,\Gamma^{g_{\ul{n}}}}\coloneqq\u|_{C\sm J_0}\cup(\TG_n\circ\mu_{\Gamma,d}(e^-)|_{\S^1_-},
    \]
    which is said to be Gusarov--Habiro $n$-equivalent to $\u$, and the string link $\TG(\bigsqcup_{i=1}^na_i)\subseteq\TG(\ball^3)$ represents the obstruction for this cobordism to be trivial, i.e.\ for $K$ to be isotopic to $\u$. When $M=\I^3$ this notion of $n$-equivalence agrees with the one of Vassiliev, namely that all finite type invariants of type $\leq n$ agree on $K$ and $\u$, as shown independently by Gusarov~\cite{Gusarov-main} and Habiro~\cite{Habiro}.
    
\begin{figure}[!htbp]
    \centering
    \begin{tikzpicture}[scale=0.78,fill opacity=0.7]
       \clip (-1.55,-1.2) rectangle (3.35,2.55);
        %vertical band vertical piece
        \fill[yellow!30!,draw=black]
            (1.35,-1.2)  to[out=20,in=-20,distance=1.8cm] (2.5,2.1)
            -- (1.8,2.45) to[out=-20,in=20,distance=1.8cm]  (0.6,-0.85) ;
        %middle vertical disk
        \fill[blue!40!,draw=blue]
            (1,-1) -- (1,2.05) to[out=35,in=160,distance=0.45cm] (2.1,2.25)
            to[out=-20,in=20,distance=1.8cm]  (1,-1) ;
        %vertical flat
        \draw[black] (0.6,-0.85) -- (1.35,-1.2) -- (1.35,1.9) -- (0.6,2.25) --  cycle ;
        %horizontal flat
        \draw[black] (0,1.5) -- (2,0.5) -- (2,-0.5) -- (0,0.5) ;
        %vertical flat
        \fill[yellow!30!,opacity=0.7](0.6,-0.85) -- (1.35,-1.2) -- (1.35,1.9) -- (0.6,2.25) --  cycle ;
        %horizontal flat
        \fill[yellow!30!,opacity=0.7] (0,1.5) -- (2,0.5) -- (2,-0.5) -- (0,0.5) ;
        %back band
        \fill[yellow!30!,draw=black]
        (0,0.5) to[out=-170, in=90, distance=0.45cm] (-1.52,-0.2)
        -- (-1.52,0.8) to[out=90, in=-170, distance=0.45cm] (0,1.5) -- (0,0.5) ;
        %vertical band horizontal piece
        \fill[yellow!30!,draw=black]
            (1.35,1.9) -- (0.6,2.25) to[out=35,in=160,distance=0.45cm] (1.8,2.45)
            -- (2.5,2.1) to[out=160,in=35,distance=0.45cm] (1.35,1.9) ;
        % middle horizontal disk
        \fill[orange!50!,draw=orange!50!black]
            (0,1) to[out=-170,in=-170,distance=3.1cm] (2,0) -- (0,1) ;
        %horizontal band
        \fill[yellow!30!,draw=black]
            (-1.51,-0.2) to[out=-60,in=-170,distance=0.76cm] (2,-0.5)
            -- (2,0.5) to[out=-170,in=-60,distance=0.76cm] (-1.51,0.8) --(-1.51,-0.2) ;
        \draw[-stealth,mygreen,very thick,opacity=1] (2,0.5) -- (1.35,0.82) -- (1.35, 1.9) node[pos=0.4,right=-2pt]{$a_0$};
\end{tikzpicture}
    \qquad
\begin{tikzpicture}[scale=0.78,fill opacity=0.65,draw opacity=0.9]
       \clip (-2,-1.2) rectangle (3.9,2.55);
       \draw[orange!50!black] 
            (-1.6,1.5) node[]{$\mathbb{B}_{\Gamma_2}$};
        \draw[blue]
            (3.5,0) node[]{$\mathbb{B}_{\Gamma_1}$};
        %back disk
        \fill[blue!20!,draw=blue]
            (0.6,-0.85) -- (0.6,2.25) to[out=35,in=160,distance=0.45cm] (1.8,2.45)
            to[out=-20,in=20,distance=1.8cm]  (0.6,-0.85) ;
        % middle vertical disk
        \fill[blue!20!,draw=blue]
            (1,-1) -- (1,2.05) to[out=35,in=160,distance=0.45cm] (2.1,2.25)
            to[out=-20,in=20,distance=1.8cm]  (1,-1) ;
        %vertical flat
        \fill[blue!20!,draw=blue]
            (0.6,-0.85) -- (1.35,-1.2) -- (1.35,1.9) -- (0.6,2.25) --  cycle ;
        %vertical band vertical piece
        \fill[blue!20!,draw=blue]
            (1.35,-1.2)  to[out=20,in=-20,distance=1.8cm] (2.5,2.1)
            -- (1.8,2.45) to[out=-20,in=20,distance=1.8cm]  (0.6,-0.85) ;
        % a_1
        \fill[draw,thick]
            (2.5,0.5) circle (0.5pt) -- 
            (1.9,0.8)  circle (0.5pt) node[pos=0.4,above,opacity=1]{$a_1$};
        \draw (2.2,0.65)  circle (0.5pt);
        %front disk
        \fill[blue!40!,draw=blue]
            (1.35,-1.2) -- (1.35,1.9) to[out=35,in=160,distance=0.45cm] (2.5,2.1)
            to[out=-20,in=20,distance=1.8cm] (1.35,-1.2) ;
        %horizontal flat
        \fill[orange!20!,draw=orange!50!black]
            (0,1.5) -- (2,0.5) -- (2,-0.5) -- (0,0.5) ;
        %back band
        \fill[orange!20!,draw=orange!50!black]
            (0,0.5) to[out=-170, in=90, distance=0.45cm] (-1.52,-0.2)
            -- (-1.52,0.8) to[out=90, in=-170, distance=0.45cm] (0,1.5) -- (0,0.5) ;
        %lower disk
        \fill[orange!30!,draw=orange!50!black]
            (0,0.5) to[out=-170,in=-170,distance=3.1cm] (2,-0.5)
            -- (0,0.5) ;
        % middle horizontal disk
        \fill[orange!50!,draw=orange!50!black]
            (0,1) to[out=-170,in=-170,distance=3.1cm] (2,0) -- (0,1) ;
        %vertical band horizontal piece
        \fill[blue!20!,draw=blue]
            (1.35,1.9) -- (0.6,2.25) to[out=35,in=160,distance=0.45cm] (1.8,2.45)
            -- (2.5,2.1) to[out=160,in=35,distance=0.45cm] (1.35,1.9) ;
        % green arcs
        \draw[-stealth,mygreen,thick]  
            (0.6,0.7) -- (2,0);
        \draw[-stealth,mygreen,thick] 
            (0.98,0) -- (0.98,2.08);
        % a_2
        \fill[draw,thick]
            (-0.3,-0.3) circle (0.5pt) -- 
            (-0.3,0.8)  circle (0.5pt) node[pos=0.3,right=-2pt,opacity=1]{$a_2$};
        \draw (-0.3,0.3)  circle (0.5pt);
        % orange lying band
        \fill[orange!20!,draw=orange!50!black]
            (-1.51,-0.2) to[out=-60,in=-170,distance=0.76cm] (2,-0.5)
            -- (2,0.5) to[out=-170,in=-60,distance=0.76cm] (-1.51,0.8) --(-1.51,-0.2)  ;
        %upper disk
        \fill[orange!40!,draw=orange!50!black]
            (2,0.5) to[out=-170,in=-170,distance=3.1cm] (0,1.5) -- (2,0.5) ;
        \draw[-stealth,mygreen,very thick,opacity=1] 
            (2,0.5) -- (1.35,0.82) -- (1.35, 1.9) node[pos=0.4,right=-2pt]{$a_0$} ;
\end{tikzpicture}
    \caption{Thickening the abstract capped grope $G_\Gamma^c$ of degree $\deg\Gamma=2$ (on the left) gives the model ball $\ball^3_2$ (on the right).}
    \label{fig:grope-thick}
\end{figure}
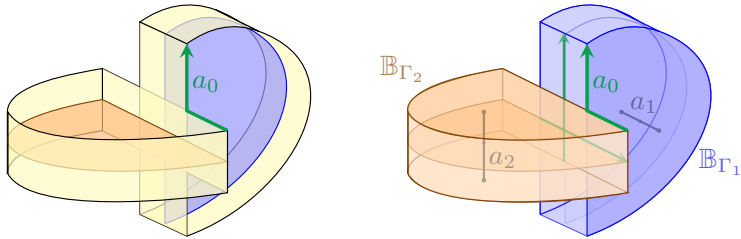
    In fact, Gusarov and Habiro defined the $n$-equivalence relation in terms of claspers, whereas the interpretation in terms of gropes is due to Conant and Teichner~\cite{CT1}. A clasper is an embedding of a tree $\Gamma$ into the $3$-manifold $M$, with leaves on $\u$ and a thickening (framing) of all edges. It can also be embedded into $G_\Gamma^c$. Moreover, this tree describes the isotopy class rel.\ boundary of the string link $\bigsqcup_{i=1}^na_i$: it is obtained from the Hopf link by iterated Bing doubles according to $\Gamma$.
    
    Therefore, one can think of a grasper as a thickening of a clasper to $d\geq3$ dimensions, where the tree structure gets lost. However, the tree reappears in the family $\mu_{\Gamma,d}$ embedded into the grasper. 
    In other words, $\ball^d_n$ is equivalent to $\ball^3_\Gamma\times\ball^{d-3}$, the thickening of $G_\Gamma^c$ to $d$ dimensions, for any tree $\Gamma$ of degree $n$. However, in that viewpoint we see at best a $(d-3)$-family of $G_\Gamma^c$ inside of $\ball^d_n$, and not an $n(d-3)$-parameter family as one can in fact define, and which is related to our map $\mu^\star_{\Gamma,d}$ (cf.\ the plumbing construction in Definition~\ref{def:plumbing}).
    Nevertheless, we hope to explore in future work connections between our families and the more classical role of $G_\Gamma^c\times\ball^2$ in dimension $d=4$ (related to the failure of the Whitney trick and work of Freedman and Quinn).

%~~~~~~~~~~
\subsection{About the Haefliger trefoil}\label{subsec:Haefliger}

As mentioned in the introduction, our realisation maps $\realmap_n$ are generalisations to knotted families of the classical Borromean rings. The key property of this link that we use is the following: two components form an unlink, whereas the third component is homotopic to the commutator $[x_1,x_2]$ of their meridians $x_1$ and $x_2$. In fact, the complement of the 2-component string unlink in $\ball^3$ is homotopy equivalent to the wedge of two circles, so up to homotopy, the third component corresponds to the map $\S^1\to\S^1\vee\S^1$ given as the commutator.

On the other hand, for $d=3$ and the tree $\Gamma$ of degree $n=2$ the Samelson product \eqref{eq-def:x-Gamma} is a map 
\[
    x_{\Gamma,3}\colon\S^0\to\Omega(\S^1\vee\S^1).
\]
By definition, this sends the basepoint $e^+\in\S^0$ to the constant loop, and the remaining point $e^-\in\S^0$ to the commutator loop $[x_1,x_2]$ where $x_i$ are the inclusions of the two wedge factors. Thus, the Borromean rings can be viewed as an embedded commutator, or a geometric analogue of the basic Samelson product. This is generalised to all $d\geq3$ (and $n\geq1$) by our Theorem~\ref{thm:analogues}: the complement of  the $n$-component string unlink in $\ball^d$ is homotopy equivalent to $\bigvee_n\S^{d-2}$ and under this equivalence $\mu_{\Gamma,d}$ corresponds to $x_{\Gamma,d-2}$.

Furthermore, as we saw in Figure~\ref{fig:Borr} the trefoil knot is obtained by ambiently connect summing the unknot into the Borromean rings. Equivalently, the trefoil is obtained from $\u$ by grasper surgery of degree $2$. For higher $d$ we obtain analogues of the trefoil in $\pi_{2(d-3)}\Embp(C,M)$.

In fact, one can consider more general Samelson or Whitehead products, of spheres of varying dimensions. For example, in degree $n=2$ we have the Samelson product
\[
    [y_1,y_2]\colon\S^{2d-p-q-4}\to\Omega(\S^{d-p-1}\vee\S^{d-q-1}),
\]
whose value at $\vec{\tau}_1\wedge\vec{\tau}_2\in\I^{d-p-2}\times\I^{d-q-2}\subseteq\S^{2d-p-q-4}$ is the commutator of the loops $y_1(\vec{\tau}_1)$ and $y_2(\vec{\tau}_2)$, which respectively foliate the first and second wedge factors. The adjoint (cf.\ \eqref{eq:adjunction}) of this map is the Whitehead product $[y_1,y_2]\colon\S^{2d-p-q-3}\to\S^{d-p-1}\vee\S^{d-q-1}$,
which is geometrically realised by the following link. Consider the coordinates
\[
    (z_1,z_2,z_3)\in\R^d=\R^{d-p-1}\times\R^{d-q-1}\times\R^{p+q+2-d}
\]
and form the hyperellipsoids in each of the three coordinate subspaces:
\begin{align*}
   L_i\colon\quad z_i=0,\quad |z_{i+1}|^2+\tfrac{1}{4}|z_{i+2}|^2=1,
\end{align*}
where $i=1,2,3$, and we consider $i+2$ modulo $3$. This forms Haefliger's Borromean rings~\cite{Haefliger}:
\[
    L\colon\S^p\sqcup\S^q\sqcup\S^{2d-p-q-3}\hra\R^d,
\]
Haefliger first used this link in the case $d=6k$ and $p=q=4k-1$ in order to construct a generator of $\Emb(\S^{4k-1},\S^{6k})$ or $\Embp(\D^{4k-1},\D^{6k})$, by an ambient connect sum analogous to the one for the classical trefoil. Sakai~\cite{Sakai-Haefliger-invt} defined finite type invariants of such knots and showed that the Haefliger trefoil is detected by a type 2 invariant; an invariant can be defined by foliating by arcs and then taking $\ev_2\colon\pi_{4k-2}\Embp(\D^1,\D^{2k+2})\to\Z$, see  \eqref{eq:Haefliger-foliate}.

%~~~~~~~~~~~~~~~~~~~~~~~~~~~~~~~~~~~~~~~~~~~~~~~~~~~~~~~
\section{Multi-families of arcs}\label{sec:multi-family}

In Section~\ref{subsec:plumbings} we define plumbings and embedded commutators.
In Section~\ref{subsec:main-construction} for a given tree $\Gamma$ of degree $n$ we find inside of $\ball^d_n$ an $\I^n\times\S^{n(d-3)}$-parameter family $\mu^\star_{\Gamma,d}$ of arcs intersecting $a_i$ in a particular pattern. We define $\mu_{\Gamma,d}$ as a restriction of $\mu^\star_{\Gamma,d}$.

%~~~~~~~~~~
\subsection{Plumbings}\label{subsec:plumbings}

For the purpose of constructing multi-families of circles $\mu^\star_{\Gamma,d}$ in $\ball^d$, we will use the standard construction called \emph{plumbing}, see Definition~\ref{def:plumbing}. In essence, this glues together two $d$-balls along some thickened squares in their boundaries using the swap map $(s,t)\mapsto(t,s)$. The result is just a $d$-ball again, parameterised in a particular (non-smooth) way. We will fix a homeomorphism to the standard ball, to which we can glue the extension to get $\ball^d_{ext}$ again, see Definition~\ref{def:plumbing-ext}. In our inductive procedure defining $\mu^\star_{\Gamma,d}$ we will then iterate this, following the shape of a tree $\Gamma$, and we will also perform iterated embedded commutators, see Definition~\ref{def:emb-commutator}.

\begin{defn}[Plumbing]\label{def:plumbing}
    Fix $d\geq3$ and let $\ball^{(j)}\cong\ball^d_{ext}$ for $j=1,2$ be two copies of the extended $d$-ball (as in Section~\ref{subsec:model-ball}). We define the \emph{plumbing of $\ball^{(1)}$ and $\ball^{(2)}$} as the quotient
        \[
            \ball^{(1)}\natural\,\ball^{(2)}\coloneqq \faktor{\ball^{(1)}\sqcup\ball^{(2)}}{
            (\vec{x},1,s,t)\sim (\vec{x},1,t,s)\text{ for }\vec{x}\in\ball^{d-3},s,t\in\D^1}\quad.
        \]
    The resulting space is homeomorphic to a $d$-ball, see Figure~\ref{fig:plumbing} for $d=3$. We fix now once and for all a homeomorphism  to the standard smooth ball:
    \[
        \xi\colon\ball^{(1)}\natural\,\ball^{(2)}\to\ball^d
    \]
    In order to specify $\xi$ we think of $\ball^{(1)}\natural\,\ball^{(2)}$ as sitting inside of $\R^d$ and with the origin $O$ as $a_0^{(1)}\cap a_0^{(2)}$ (for example, perform the ambient gluing of the standardly embedded extensions and balls, as in Figure~\ref{fig:plumbing}), and we take $\ball^d\subseteq\R^d$ to be a ball of large enough radius centred at $O$ and containing $\ball^{(1)}\natural\,\ball^{(2)}$. Let $\xi$ be the radial stretching: $O$ is fixed, and on each half-line $H$ starting at $O$ apply the stretching that sends $H\cap(\ball^{(1)}\natural\,\ball^{(2)})$ to $H\cap\ball^d$.
    % and use the same notation for this standard smooth structure on $\ball^{(1)}\natural\,\ball^{(2)}$.
\end{defn}
\begin{figure}[!htbp]
    \centering
    \begin{tikzpicture}[scale=1.3,every node/.style={scale=0.7}]
\clip (-1.9,-0.9) rectangle (3.1,0.9);

% B_1
    \draw[dashed]
        (-1.84,-0.54)  to[in=180,out=90,distance=0.2cm]
        (-0.25,-0.25);
    \fill[blue!45!red!80!white,fill opacity=0.4]
        (-0.75,0.25) to[out=180,in=-90,distance=0.2cm]
        (-1.84,0.47) --
         (-1.84,-0.54)  to[in=180,out=-90,distance=0.2cm]
         (-0.75,-0.75);
    \fill[blue!45!red!80!white,fill opacity=0.2,draw=black]
        (-0.25,0.75) to[out=180,in=180,distance=1.75cm] (-0.75,0.25);
    \draw
        (-1.84,-0.54) -- (-1.84,0.47);
    \draw
        (-1.84,-0.54)  to[in=180,out=-90,distance=0.2cm] (-0.75,-0.75);

    \draw[dashed,red]
        (-1.84,0)  to[in=180,out=90,distance=0.2cm] (-0.25,0.2);
    \draw[red,-latex]
        (-1.84,0)  to[in=180,out=-90,distance=0.2cm] (-0.75,-0.2);
    \draw[red] (-1.7,-0.2) node{$\alpha_1$};
% ext region for B_1
    \draw[dashed]
        (-0.25,-0.25) rectangle (0.75,0.75);
    \draw[dashed]
        (-0.75,-0.75) -- (0.25,-0.75) -- 
        (0.75,-0.25) -- (-0.25,-0.25) -- (-0.75,-0.75);
    \fill[red!20,fill opacity=0.6,draw=red,-latex]
         (-0.25,0.2) -- (-0.75,-0.2) -- (0.25,-0.2) --
         (0.75,0.2) -- (-0.25,0.2);
    \fill[blue!45!red!80!white,fill opacity=0.2,draw=black]
        (-0.75,0.25) -- (0.25,0.25) -- 
        (0.75,0.75) -- (-0.25,0.75) -- (-0.75,0.25);

    \fill[blue!45!red!80!white,fill opacity=0.4,draw=black]
        (-0.75,-0.75) rectangle (0.25,0.25);

    \draw (-0.05,0.06) node[red]{$a_0^{\mathbb{B}_1}$};

    \draw[very thick] (0.75,0.2) -- (0.75,0.75) -- (0.5,0.5) ;
% B_2
    \draw 
        (2.4, -0.5) to[out=30,in=-50,distance=0.35cm] (3.01,0.39);

% ext region for B_2
    \draw[dashed]
        (1.75,-0.25) -- (1.75,0.75);
    \draw[dashed]
        (0.25,-0.75) -- (1.25,-0.75) -- 
        (1.75,-0.25) -- (0.75,-0.25) -- (0.25,-0.75);
    \fill[red!20,fill opacity=0.6,draw=red,-latex]
         (1.5,0.5) -- (1.5,-0.5) -- (0.5,-0.5) -- (0.5,0.5) --  (1.5,0.5);
    \fill[yellow!45!green!90!white,fill opacity=0.2,draw=black]
        (0.25,0.25) -- (1.25,0.25) -- 
        (1.75,0.75) -- (0.75,0.75) -- (0.25,0.25);

    \fill[yellow!45!green!90!white,fill opacity=0.4,draw=black]
        (0.25,-0.75) rectangle (1.25,0.25);
        
% front of B_2
    \draw[red,dashed,-latex]
        (2.8,-0.2)  to[out=-140,in=0,distance=0.4cm] (1.5,-0.5);
    \draw[red,]
        (1.5,0.5) to[out=0,in=50,distance=0.6cm] (2.84,-0.15) ;
    \draw[red] (2.5,0.45) node{$\alpha_2$};
    
    \draw[dashed]
        (3.01,0.39)  to[out=-60,in=0,distance=0.55cm] (1.75,-0.25);
    \fill[yellow!45!green!90!white,fill opacity=0.4,draw=black]
        (1.25,0.25) to[out=0,in=0,distance=1.75cm] (1.25,-0.75) -- (1.25,0.25);
    \fill[yellow!45!green!90!white,fill opacity=0.2]
        (1.75,0.75) to[out=0,in=130,distance=0.4cm] 
        (3.01,0.39) to[in=30,out=-50,distance=0.35cm] 
        (2.4, -0.5) to[out=40,in=0,distance=0.72cm] 
        (1.25,0.25) -- (1.75,0.75) ;
    \draw
        (1.75,0.75) to[out=0,in=130,distance=0.4cm] 
        (3.01,0.39) ;

    \draw 
        (1.2,0.62) node[red]{$a_0^{\mathbb{B}_2}$}
        (2.45,-0.68) node[green!45!red]{$\mathbb{B}_2$}
        (-1.7,0.74) node[yellow!45!red]{$\mathbb{B}_1$};

\end{tikzpicture}
    \caption{The plumbing of two extended $3$-balls. The new extension will be glued onto the arc in bold.}
    \label{fig:plumbing}
\end{figure}
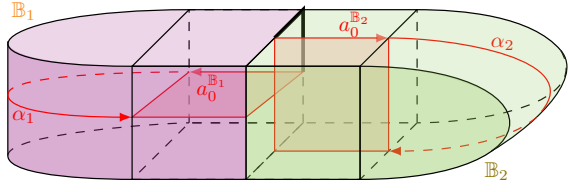
\begin{defn}[Extended plumbing]\label{def:plumbing-ext}
    We define $(\ball^{(1)}\natural\,\ball^{(2)})_{ext}$ by gluing a copy of $\ball^{d-3}\times(\D^1)^3$ to the $d$-ball $\xi(\ball^{(1)}\natural\,\ball^{(2)})$ so that $\ball^{d-3}\times-1\times\D^1\times\D^1$ is identified with a tubular neighbourhood of the arc obtained by applying $\xi$ to the union of the two arcs $\vec{0}\times1\times[0,1]\times1\subseteq\partial\ball^{(j)}$ for $j=1,2$ (drawn in bold in Figure~\ref{fig:plumbing}). In particular, we have the new $a_0\subseteq\partial(\ball^{(1)}\natural\,\ball^{(2)})_{ext}$. 
    
    Moreover, if $\ball^{(j)}_{n_j}=(\ball^{(j)}, a_0^{(j)}, \bigsqcup_{i=1}^{n_j}a_i^{(j)})$ are model balls of degrees $n_1,n_2\geq 1$, then we define the \emph{plumbed model ball} of degree $n\coloneqq n_1+n_2$ as the tuple
    \[
        \Big(
            (\ball^{(1)}\natural\,\ball^{(2)})_{ext},\; 
            a_0,\;
            \bigsqcup_{i=1}^{n_1}a_i^{(1)} 
            \sqcup\bigsqcup_{i=1}^{n_2}a_i^{(2)}
        \Big).\qedhere
    \]
\end{defn}

\begin{defn}[Embedded commutator]\label{def:emb-commutator}
    Fix $d\geq3$ and let $\ball^{(j)}\cong\ball^d_{ext}$ for $j=1,2$ be two copies of the extended $d$-ball, with $a_0^{(j)}\subseteq A_0^{(j)}\subseteq\partial\ball^{(j)}$ as in \eqref{eq-def:a-0}--\eqref{eq-def:A-0}. 
    Given some smoothly embedded annuli $C_j\colon\S^1\times[-\e,\e]\hra\ball^{(j)}$ for $j=1,2$, $\e>0$, such that $C_j|_{\S^1_+\times[-\e,\e]}=A_0^{(j)}$,
    we define their \emph{embedded commutator}
    \[
        [C_1,C_2]_\natural \colon\,
        \S^1\times[-\e,\e]
        \hra(\ball^{(1)}\natural\,\ball^{(2)})_{ext}.
    \]
    First, consider the two ``push-offs'', in the negative and positive $\e$-directions, of the nontrivial half of $C_j$, namely the strips 
    $A_j^-\coloneqq C_j|_{\S^1_-\times[-\e,-\frac{\e}{2}]}$ and $A_j^+\coloneqq C_j|_{\S^1_-\times[\frac{\e}{2},\e]}$.

    Second, let $e_i\subseteq\partial(\ball^{(1)}\natural\,\ball^{(2)})$ for $1\leq i\leq 5$ be the arcs in the two glued extensions as in Figure~\ref{fig:plumbing-commutator}.
    Let $E_i$ be the strip neighbourhood of the arc $e_i$ such that the ends $\partial E_i$ can glue to the appropriate $A_j^+$ or $A_j^-$. We glue the strips along their normal directions to define
    \[
        [C_1,C_2]_\natural\coloneqq
        \xi(A_0\cdot E_1\cdot
        A_1^+\cdot E_2\cdot A_2^+\cdot  E_3\cdot (A^-_1)^{-1}\cdot  E_4\cdot (A_2^-)^{-1}\cdot  E_5),
    \]
    where the inverse means the opposite orientation in the arc coordinate. We parameterise this so that it is defined on $\S^1\times[-\e,\e]$ and satisfies $[C_1,C_2]_\natural|_{\S^1_+\times[-\e,\e]} =A_0\subseteq\partial(\ball^{(1)}\natural\,\ball^{(2)})_{ext}$.
\end{defn}
\begin{figure}[!htbp]
    \centering
    \begin{tikzpicture}[scale=1.3,every node/.style={scale=0.7}]
\clip (-2.25,-1) rectangle (3.35,1.4);
% A_1
    \draw[blue!45!red!80!white,thick,-latex]
        (-0.25,0.75) to[out=180,in=180,distance=1.75cm] (-0.75,0.25);
    \draw[blue!45!red!80!white,thick,latex-]
       (-0.25,-0.25)  to[in=180,out=180,distance=1.75cm] (-0.75,-0.75);
    \fill[white]
         (-0.83,-0.35) rectangle (-0.68,-0.2);
    \draw[blue!45!red!80!white]
        (-1.8,0.18) node{$\alpha_1^+$}
        (-1.9,-0.85) node{$(\alpha_1^-)^{-1}$};
% A_2
    \draw[yellow!45!green!90!white,thick,latex-]
        (1.75,0.75) to[out=0,in=0,distance=1.75cm] (1.75,-0.25);
    \fill[white]
        (2.45,-0.25) rectangle (2.65,-0.1);
    \draw[yellow!45!green!90!white,thick,-latex]
        (1.25,0.25) to[out=0,in=0,distance=1.75cm] (1.25,-0.75);

    \draw[blue,thick,-latex,dashed]
        (-0.25,-0.25) -- (1.75,-0.25) node[above,pos=0.15]{$e_4$};

    \draw[yellow!45!green!90!white]
        (3,0.8) node{$(\alpha_2^-)^{-1}$}
        (2.4,-0.7) node{$\alpha_2^+$};

% ext region for B_1
    \fill[blue!45!,fill opacity=0]
        (-0.25,-0.25) rectangle (0.75,0.75);
    \draw[dashed]
        (0.25,-0.75) -- (0.75,-0.25)  
        (-0.25,-0.25) -- (-0.75,-0.75)
        (-0.25,-0.25) -- (-0.25,0.75)
        (0.75,-0.25) -- (0.75,0.75);

    \fill[blue!45!,fill opacity=0,draw=black]
        (-0.75,0.25) -- (0.25,0.25) -- 
        (0.75,0.75) -- (-0.25,0.75) -- (-0.75,0.25);

    \fill[blue!45!,fill opacity=0,draw=black]
        (-0.75,-0.75) rectangle (0.25,0.25);

    \draw[blue,thick,latex-] 
        (0.75,0.2)  
        to[out=-140,in=-90,distance=0.6cm]
        (1.1,0.6) to[out=90,in=180,distance=0.15cm]
        (1.3,0.75) -- (1.75,0.75) node[above,pos=0.7]{$e_5$} ; 
    \fill[white]  (0.95,0.2) rectangle  (1.1,0.3);

    \fill[white]  (0.45,0.7) rectangle  (0.55,0.8);
    \draw[red,thick,latex-] (0.5,0.5) -- (0.5,1) -- (0.87,1.35) -- (0.87,0.35) -- (0.75,0.2);
    \fill[white] (0.82,0.7) rectangle (0.92,0.8) ;
    \draw (1.05,1.2) node[red]{$a_0$};
% ext region for B_2
    \fill[blue!45!,fill opacity=0]
        (0.75,-0.25) rectangle (1.75,0.75);
    \draw (0.75,0.75) -- (1.75,0.75);

    \fill[blue!45!,fill opacity=0]
        (0.25,0.25) -- (0.75,0.75) -- 
        (1.75,0.75) -- (1.25,0.25) -- (0.25,0.25);

    \fill[blue!45!,fill opacity=0,draw=black]
        (0.25,-0.75) rectangle (1.25,0.25);
    \fill[blue!45!,fill opacity=0,draw=black]
         (1.25,0.25) -- (1.25,-0.75) -- (1.75,-0.25) -- (1.75,0.75) -- (1.25,0.25);

    \draw[blue,thick,-latex] 
        (0.5,0.5) to[out=-90,in=-90,distance=0.3cm] 
        (0.4,0.6) to[out=90,in=0,distance=0.1cm]
        (0.3,0.75) -- (-0.25,0.75) node[above,pos=0.7]{$e_1$} ; 
    \draw[blue,thick,-latex]
        (-0.75,0.25) -- (1.25,0.25) node[above,pos=0.18]{$e_2$};
    \draw[blue,thick,-latex]
        (1.25,-0.75) -- (-0.75,-0.75) node[below,pos=0.15]{$e_3$};

    \draw[very thick] (0.75,0.2) -- (0.75,0.75) -- (0.5,0.5) ;

\end{tikzpicture}
    \caption{The embedded commutator of the circles $c_j\coloneqq C_j|_{\S^1\times0}$, where $a_j^\pm=C_j|_{\S^1_-\times\pm\e}$ and $d=3$.}
    \label{fig:plumbing-commutator}
\end{figure}
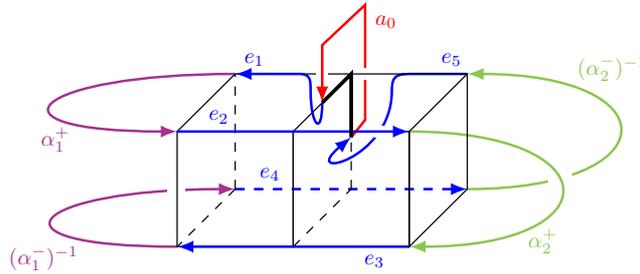

\begin{rem}\label{rem:smooth-emb-comm}
    Let us give a few important remarks about this construction.
\begin{enumerate}
    \item 
        For simplicity in Figure~\ref{fig:plumbing-commutator}  we drew $\ball^{(1)}\natural\,\ball^{(2)}$ instead of $\xi(\ball^{(1)}\natural\,\ball^{(2)})\subseteq(\ball^{(1)}\natural\,\ball^{(2)})_{ext}$.
    \item 
        Also for simplicity we drew only $[C_1,C_2]_\natural|_{\S^1\times0}=[c_1,c_2]_\natural$, the commutator of the circles $c_j\coloneqq C_j|_{\S^1\times0}=a_0^{(j)}\cdot \alpha_j$, where $\alpha_j$ are as in Figure~\ref{fig:plumbing}.
    \item 
        To obtain the annulus $[C_1,C_2]_\natural$ from $[c_1,c_2]_\natural$, each $\alpha$-curve should be thickened into the $\e$-direction (for $\alpha_1$ this is up/down, for $\alpha_2$ this is forward/backward), and each $e$-curve should receive a right-handed half-twist. 
    \item
        Note that $[c_1,c_2]_\natural$ is homotopic to the commutator $c_1c_2c_1^{-1}c_2^{-1}$.
    \item
        Note that the curve $[c_1,c_2]_\natural$ is smooth in $\ball^d_{ext}\subseteq\R^d$, apart from the corners in $a_0$. Therefore, we can use Remark~\ref{rem:smooth-a0} to view it as an element in
        \[
         [c_1,c_2]_\natural\in\Emb((\S^1,\S^1_+),(\ball^d_n,a_0)).
        \]
    \item
        On the other hand, $[C_1,C_2]_\natural$ has ``ridges'', which needs to be straightened-out by $\xi$.
    \item
        Finally, we remark that although in the figure we used arcs $\alpha_j$ which are in the boundary of $\ball^{(j)}$, in general the embedded commutator construction will use annuli $C$ which will consist of $C(\S^1_+)=A_0$ and a neatly embedded annulus $C(\S^1_-)\subseteq \ball^d$. Thus, one should image that depicted are only some pieces of the balls $\ball^{(j)}$.
        \qedhere
\end{enumerate}
\end{rem}

%~~~~~~~~~~
\subsection{The main construction}\label{subsec:main-construction}

Recall from Section~\ref{subsec:notation-other} that we use notation $\vec{t}=(t_1,\dots,t_n)\in\I^n$ and $\vec{\tau}=\tau_1\wedge\dots\wedge\tau_n\in(\S^{d-3})^{\wedge n}\cong\S^{n(d-3)}$. The following is our main construction.

\begin{thm}\label{thm:multi-family}
    For any $d\geq3$ and $\Gamma\in\Tree(n)$ with $n\geq1$ the map
    \begin{equation}\label{eq-def:ol-mu}
        \mu^\star_{\Gamma,d}\colon\quad
        \I^n\times\S^{n(d-3)}\to \Emb\big((\S^1,\S^1_+),(\ball^d_{ext},a_0)\big),
    \end{equation}
    defined in Definitions~\ref{def:base}--\ref{def:partIV} satisfies the following properties.
    \begin{enumerate}
        \item
    For $\vec{\tau}=e^+$ we have the trivial family of circles:
    \[
        \mu^\star_{\Gamma,d}(-,e^+)=\const_{c_0}.
    \]
        \item
    For $\vec{t}\in\I^n$, if $t_i=1$ for some $1\leq i\leq n$, then we have the trivial family of circles:
    \[
        \mu^\star_{\Gamma,d}(\vec{t},-)=\const_{c_0}.
    \]
        \item
    For $1\leq i\leq n$ the intersection of the arc $a_i(\D^1)\subseteq\ball^d_{ext}$ and the circle $\mu^\star_{\Gamma,d}(\vec{t},\vec{\tau})\subseteq\ball^d_{ext}$ is
    \[
    a_i(\D^1) \cap \mu^\star_{\Gamma,d}(\vec{t},\vec{\tau})=\begin{cases}
        \{a_i(0)\}, & t_i=\tfrac{1}{2},\;\tau_i=e^-,\;\text{and }t_j\leq\tfrac{3}{4}\text{ for all } 1\leq j\leq n,\\
        \emptyset, & \text{otherwise}.
    \end{cases}
    \]
\end{enumerate}
\end{thm}
In particular, note that $(iii)$ implies that restricting to $\vec{t}=\vec{0}=(0,\dots,0)\in\I^n$ defines a family
    \begin{equation}\label{eq-def:mu-Gamma}
        \mu_{\Gamma,d}
        \coloneqq
        \mu^\star_{\Gamma,d}(\vec{0},-)\colon\quad
        \S^{n(d-3)}\to \Emb\big((\S^1,\S^1_+),(\ball^d\sm\bigsqcup_{i=1}^na_i,a_0)\big),
    \end{equation}
    % which does not intersect any of the arcs $a_i$.
    % In other words, for every $\vec{\tau}\in\S^{n(d-3)}$ the circle $\mu^\star_{\Gamma,d}(\vec{0},\vec{\tau})$ lies in the boundary of $\ball^d_{ext}$ and misses the endpoints of all arcs $a_i\subseteq\ball^d_{ext}$, $1\leq i\leq n$.
    % its restriction to is precisely the arc $a_0$.
The definition of  $\mu^\star_{\Gamma,d}$ will be by induction and in several steps, Definitions~\ref{def:base}--\ref{def:partIV}. First we define a map from $\I^n\times \S^{n_1(d-3)}\times \S^{n_2(d-3)}$ into a space of embedded annuli. Then for its restriction to the subspace $\I^n\times(\S^{n_1(d-3)}\times e^+ \vee e^+ \times\S^{n_2(d-3)})$ we find a homotopy (isotopy) to $\const_{c_0}$, so that adding this on a collar gives a map out of the space $\I^n\times(\S^{n_1(d-3)}\wedge\S^{n_2(d-3)})=\I^n\times\S^{n(d-3)}$.

In other words, this gives a multi-family of annuli:
        \begin{equation}\label{eq-def:mu-star-e}
            \mu^{\star,\e}_{\Gamma,d}\colon\quad
            \I^n\times \S^{n(d-3)}\ra
            \Emb\big(
                (\S^1\times[-\e,\e],
                \S^1_+\times[-\e,\e]),
                (\ball^d_{ext},A_0)\big).
       \end{equation}
such that $\mu^{\star,\e}_{\Gamma,d}(-,e^+)=\const_{C_0}$.
Finally, we will restrict to the core $\S^1\times0\subseteq\S^1\times[-\e,\e]$ to obtain
        \[
            \mu^\star_{\Gamma,d}\colon\quad
            \I^n\times \S^{n(d-3)}\ra 
            \Emb((\S^1,\S^1_+),(\ball^d_{ext},a_0)),
        \]
for which we check the desired properties of the theorem. The purpose of using annuli is to simplify the inductive step, since for the embedded commutators (from Definition~\ref{def:emb-commutator}) we need the $\e$-direction for pushing circles off themselves.
\begin{defn}[Induction base]\label{def:base}
    For $n=1$ we define 
        \begin{equation}\label{eq-def:ol-mu-1}
            \mu^{\star,\e}_{|,d}\colon\I\times\S^{d-3}\to \Emb(\S^1\times[-\e,\e],\ball^d_{ext})
        \end{equation}
    using $\mu^\star_d\colon\I\times\S^{d-3}\to \Emb(\S^1,\ball^{d}_{ext})$ from~\eqref{eq-def:mu-star} and thickening all the circles to annuli using the normal direction to $\ball^{d-1}\subseteq\ball^d\subseteq\ball^d_{ext}$ (so $x_d$-axis, parallel to $a_1$, see~\eqref{eq:a-1}).
    In particular, for $n=1$ we will simply have $\mu^\star_{|,d}=\mu^\star_d$ from~\eqref{eq-def:mu-star}.
\end{defn}

\begin{defn}[Part I -- commutators]
    Assuming now that $\mu^{\star,\e}_{\Gamma,d}$ is defined for trees $\Gamma$ with $<n$ leaves, for some $n\geq1$, let $\Gamma=[\Gamma_1,\Gamma_2]$ be a tree with $\deg(\Gamma)=n=n_1+n_2$ leaves, for $\Gamma_j\in\Tree(S_j)$ and $n_j=|S_j|$, where $S_1\sqcup S_2=\ul{n}$.
    By the induction hypothesis for $\vec{t}_j\in\I^{n_j}$, $\vec{\tau}_j\in\S^{n_j(d-3)}$ we have
    \[
        C_j\coloneqq\mu^{\star,\e}_{\Gamma_j,d}(\vec{t}_j,\vec{\tau}_j)\colon\;
        \S^1\times[-\e,\e]\hra\ball^{(j)}
    \]
    such that $C_j|_{\S^1_+\times[-\e,\e]}=A_0^{(j)}$. We can thus put
    \[
        \mu^{\star,\e}_{\Gamma,d}(\vec{t}_1\times\vec{t}_2,\vec{\tau}_1\times\vec{\tau}_2)
        \coloneqq [C_1,C_2]_\natural\colon\S^1\times[-\e,\e]\hra (\ball^{(1)}\natural\,\ball^{(2)})_{ext},
    \]
    where the target is the extended plumbing from Definition~\ref{def:plumbing-ext}. 
    This defines a continuous map on $\I^{n_1}\times\I^{n_2}\times\S^{n_1(d-3)}\times\S^{n_2(d-3)}$. Indeed, $\mu^{\star,\e}_{\Gamma_j,d}$ are continuous, and the plumbing is performed independently of the parameters. Moreover, by the definition of the embedded commutator the restriction of this to $\S^1_+\times[-\e,\e]$ is $A_0$.
    Finally, this is indeed a smooth embedding, with the caveat that $A_0$ has corners, see Remark~\ref{rem:smooth-a0}.
\end{defn}
\begin{figure}[!htbp]
    \centering
    \begin{tikzpicture}[scale=1.3,every node/.style={scale=0.7}]
\clip (-0.9,-1) rectangle (3.1,1);

    \draw[blue,thick,-latex,dashed]
        (-0.25,-0.25) -- (1.75,-0.25) node[above,pos=0.15]{$e_4$};

% ext region for B_1
    \draw[dashed]
        (-0.25,-0.25) rectangle (0.75,0.75);

    \fill[yellow!45!green!90!white,fill opacity=0.4,draw=black]
        (-0.75,0.25) -- (0.25,0.25) -- 
        (0.75,0.75) -- (-0.25,0.75) -- (-0.75,0.25);

    \fill[yellow!45!red,fill opacity=0,draw=black]
        (-0.75,-0.75) rectangle (0.25,0.25);

    \draw[very thick] (0.75,0.2) -- (0.75,0.75) -- (0.5,0.5) ;
% B_2
    \draw 
        (2.4, -0.5) to[out=30,in=-50,distance=0.35cm] (3.01,0.39);

    \draw[blue,thick,latex-] 
        (0.75,0.2)  
        to[out=-140,in=-90,distance=0.6cm]
        (1.1,0.6) to[out=90,in=180,distance=0.15cm]
        (1.3,0.75) -- (1.75,0.75) node[above,pos=0.7]{$e_5$} ; 
    \draw[blue,thick,-latex] 
        (0.5,0.5) to[out=-90,in=-90,distance=0.3cm] 
        (0.4,0.6) to[out=90,in=0,distance=0.1cm]
        (0.3,0.75) -- (-0.25,0.75) node[above,pos=0.7]{$e_1$} ; 

% ext region for B_2
    \draw[dashed]
        (1.75,-0.25) -- (1.75,0.75);
    \draw[dashed]
        (0.25,-0.75) -- (1.25,-0.75) -- 
        (1.75,-0.25) (0.75,-0.25) -- (0.25,-0.75);

    \fill[yellow!45!green!90!white,fill opacity=0.4,draw=black]
        (0.25,0.25) -- (1.25,0.25) -- 
        (1.75,0.75) -- (0.75,0.75) -- (0.25,0.25);

    \fill[green!45!red,fill opacity=0,draw=black]
        (0.25,-0.75) rectangle (1.25,0.25);
        
% front of B_2
    % \fill[yellow!45!green!90!white,fill opacity=0]
    %     (1.25,0.25) to[out=0,in=0,distance=1.75cm] (1.25,-0.75) -- (1.25,0.25);
    \draw[blue,thick]
        (1.25,0.25) to[out=0,in=0,distance=1.75cm] (1.25,-0.75);
    \fill[yellow!45!green!90!white,fill opacity=0.4]
        (1.75,0.75) to[out=0,in=130,distance=0.4cm] 
        (3.01,0.39) to[in=30,out=-50,distance=0.35cm] 
        (2.4, -0.5) to[out=40,in=0,distance=0.72cm] 
        (1.25,0.25) -- (1.75,0.75) ;
    \draw[blue,thick]
        (1.75,0.75) to[out=0,in=130,distance=0.4cm] 
        (3.01,0.39) ;
    \draw[dashed,blue,thick]
        (3.01,0.39)  to[out=-60,in=0,distance=0.55cm] (1.75,-0.25);

     \fill[yellow!45!green!90!white,fill opacity=0.2]
        (3.01,0.39)  to[out=-60,in=0,distance=0.95cm] (-0.25,-0.25)
        -- (-0.75,-0.75) to[out=0,in=-150,distance=0.75cm]  (2.4, -0.5) to[out=30,in=-50,distance=0.35cm]  (3.01,0.39) ;

    \draw[blue,thick,-latex]
        (-0.75,0.25) -- (1.25,0.25) node[above,pos=0.15]{$e_2$};
    \draw[blue,thick,-latex]
        (1.25,-0.75) -- (-0.75,-0.75) node[below,pos=0.15]{$e_3$};
    \draw[blue,thick,-latex,dashed]
        (-0.75,-0.75) -- (-0.25,-0.25) ;

\end{tikzpicture}
\begin{tikzpicture}[scale=1.3,every node/.style={scale=0.6}]
\clip (-1.1,-1) rectangle (1.9,1);

    \draw[blue,thick,-latex,dashed]
        (-0.25,-0.25) -- (1.75,-0.25) node[above,pos=0.15]{$e_4$};
    \draw[blue,thick,-latex]
        (1.75,-0.25) -- (1.75,0.75);

% ext region for B_1
    \fill[blue!45!red!80!white,fill opacity=0.3,draw=black,dashed]
        (-0.25,-0.25) rectangle (0.75,0.75);
    \fill[yellow!45!green!90!white,fill opacity=0.3,draw=black,dashed]
        (-0.75,-0.75) -- (0.25,-0.75) -- 
        (0.75,-0.25) -- (-0.25,-0.25) -- (-0.75,-0.75);
    \fill[blue!45!red!80!white,fill opacity=0.3]
        (-0.75,-0.75) -- (-0.75,0.25) -- (-0.25,0.75)-- (-0.25,-0.25) ; 
    \fill[yellow!45!green!90!white,fill opacity=0.4,draw=black]
        (-0.75,0.25) -- (0.25,0.25) -- 
        (0.75,0.75) -- (-0.25,0.75) -- (-0.75,0.25);

    \fill[blue!45!red!80!white,fill opacity=0.4,draw=black]
        (-0.75,-0.75) rectangle (0.25,0.25);

    \draw[very thick] (0.75,0.2) -- (0.75,0.75) -- (0.5,0.5) ;
% B_2
    \draw[blue,thick,latex-] 
        (0.75,0.2)  
        to[out=-140,in=-90,distance=0.6cm]
        (1.1,0.6) to[out=90,in=180,distance=0.15cm]
        (1.3,0.75) -- (1.75,0.75) node[above,pos=0.7]{$e_5$} ;
    \draw[blue,thick,-latex] 
        (0.5,0.5) to[out=-90,in=-90,distance=0.3cm] 
        (0.4,0.6) to[out=90,in=0,distance=0.1cm]
        (0.3,0.75) -- (-0.25,0.75) node[above,pos=0.7]{$e_1$} ; 

% ext region for B_2
    \fill[yellow!45!green!90!white,fill opacity=0.3]
        (1.25,-0.75) -- (1.75,-0.25) -- (0.75,-0.25) -- (0.25,-0.75);
    \fill[blue!45!red!80!white,fill opacity=0.3,draw=black]
        (0.75,-0.25) rectangle (1.75,0.75);

    \fill[yellow!45!green!90!white,fill opacity=0.4]
        (0.25,0.25) -- (0.75,0.75) -- 
        (1.75,0.75) -- (1.25,0.25) -- (0.25,0.25);

    \fill[blue!45!red!80!white,fill opacity=0.4,draw=black]
        (0.25,-0.75) rectangle (1.25,0.25);

    \draw[blue,thick,-latex]
        (-0.75,0.25) -- (1.25,0.25) node[above,pos=0.15]{$e_2$};
    \draw[blue,thick,-latex]
         (1.25,0.25) -- (1.25,-0.75);
    \draw[blue,thick,-latex]
        (1.25,-0.75) -- (-0.75,-0.75) node[below,pos=0.15]{$e_3$};

    \fill[yellow!45!green!90!white,fill opacity=0.4,draw=black]
         (1.25,0.25) -- (1.25,-0.75) -- (1.75,-0.25) -- (1.75,0.75) -- (1.25,0.25);

\end{tikzpicture}
\begin{tikzpicture}[scale=1.3,every node/.style={scale=0.6}]
\clip (-2.1,-1) rectangle (1.9,1);

    \draw[blue,thick,-latex,dashed]
        (-0.25,-0.25) -- (1.75,-0.25) node[above,pos=0.15]{$e_4$};
    \draw[blue,thick,-latex]
        (1.75,-0.25) -- (1.75,0.75);
% B_1
    \draw[blue,thick,dashed]
        (-1.84,-0.54)  to[in=180,out=90,distance=0.2cm]
        (-0.25,-0.25);
    \fill[blue!45!red!80!white,fill opacity=0.4]
        (-0.75,0.25) to[out=180,in=-90,distance=0.2cm]
        (-1.84,0.47) --
         (-1.84,-0.54)  to[in=180,out=-90,distance=0.2cm]
         (-0.75,-0.75);
    \fill[blue!45!red!80!white,fill opacity=0.1,draw=blue,thick]
        (-0.25,0.75) to[out=180,in=180,distance=1.75cm] (-0.75,0.25);
    \draw
        (-1.84,-0.54) -- (-1.84,0.47);
    \draw[blue,thick]
        (-1.84,-0.54)  to[in=180,out=-90,distance=0.2cm]
        (-0.75,-0.75);
% ext region for B_1
    \draw[dashed]
        (-0.25,-0.25) rectangle (0.75,0.75);
    \draw[dashed]
        (-0.75,-0.75) -- (0.25,-0.75) -- 
        (0.75,-0.25) -- (-0.25,-0.25) -- (-0.75,-0.75);

    \fill[blue!45!red!80!white,fill opacity=0.1,draw=black]
        (-0.75,0.25) -- (0.25,0.25) -- 
        (0.75,0.75) -- (-0.25,0.75) -- (-0.75,0.25);

    \fill[blue!45!red!80!white,fill opacity=0.4,draw=black]
        (-0.75,-0.75) rectangle (0.25,0.25);

    \draw[very thick] (0.75,0.2) -- (0.75,0.75) -- (0.5,0.5) ;
% B_2

    \draw[blue,thick,latex-] 
        (0.75,0.2)  
        to[out=-140,in=-90,distance=0.6cm]
        (1.1,0.6) to[out=90,in=180,distance=0.15cm]
        (1.3,0.75) -- (1.75,0.75) node[above,pos=0.7]{$e_5$} ;
    \draw[blue,thick,-latex] 
        (0.5,0.5) to[out=-90,in=-90,distance=0.3cm] 
        (0.4,0.6) to[out=90,in=0,distance=0.1cm]
        (0.3,0.75) -- (-0.25,0.75) node[above,pos=0.7]{$e_1$} ; 

% ext region for B_2
    \fill[blue!45!red!80!white,fill opacity=0.1,draw=black]
        (0.75,-0.25) rectangle (1.75,0.75);
    % \draw[]
    %     (1.25,-0.75) -- 
    %     (1.75,-0.25) -- (0.75,-0.25);

    % \fill[blue!45!red!80!white,fill opacity=0,draw=black]
    %     (0.25,0.25) -- (1.25,0.25) -- 
    %     (1.75,0.75) -- (0.75,0.75) -- (0.25,0.25);
    \fill[blue!45!red!80!white,fill opacity=0.1]
        (0.25,0.25) -- (0.75,0.75) -- (0.75,0.25) --  (0.25,0.25);

    \fill[blue!45!red!80!white,fill opacity=0.4,draw=black]
        (0.25,-0.75) rectangle (1.25,0.25);

    \draw[blue,thick,-latex]
        (-0.75,0.25) -- (1.25,0.25) node[above,pos=0.15]{$e_2$};
    \draw[blue,thick,-latex]
         (1.25,0.25) -- (1.25,-0.75);
    \draw[blue,thick,-latex]
        (1.25,-0.75) -- (-0.75,-0.75) node[below,pos=0.15]{$e_3$};

\end{tikzpicture}
    \caption{In the isotopy $h_{\Gamma,d}$ for $\vec{\tau}_1=e^+$ we use the green strip as in the leftmost picture, in order to homotope its boundary onto the bold arc $a_0$. In the isotopy for $\vec{\tau}_2=e^+$ we use the purple strip as in the rightmost picture, in order to homotope its boundary onto the bold arc $a_0$. When $\vec{\tau}_1=\vec{\tau}_3=e^+$ we have the situation as in the middle picture, where we can run both isotopies in parallel.}
    \label{fig:plumbing-interpolating}
\end{figure}
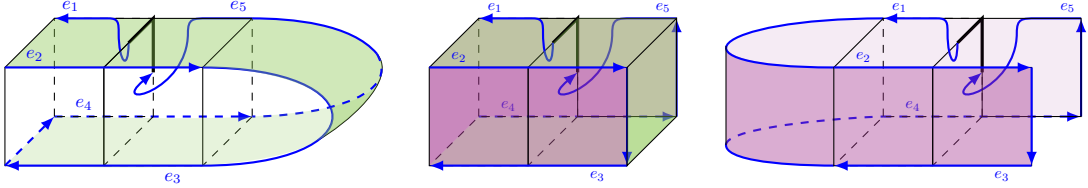
\begin{defn}[Part II -- isotoping circles]
    We define an isotopy
    \[
        h_{\Gamma,d}\colon[0,1]\times\I^n\times(\S^{n_1(d-3)}\times e^+  \vee e^+ \times\S^{n_2(d-3)})\ra \Emb((\S^1,\S^1_+),(\ball^d_{ext},a_0)),
    \]
    from the restriction $\mu^\star_{\Gamma,d}$ of $\mu^{\star,\e}_{\Gamma,d}$ (both to the wedge and the annulus core) to $\const_{c_0}$.  Note that of course any two circles in $\ball^d$ are based isotopic; however, we want to make sure that this forms a continuous family, and that in our isotopy all circles satisfy the conditions of Theorem~\ref{thm:multi-family}.
    
    First consider $\vec{\tau}_1=e^+$ and $\vec{\tau}_2\in \S^{n_1(d-3)}$, so  $\mu^\star_{\Gamma,d}(\vec{t}_1\times\vec{t}_2,e^+\times\vec{\tau}_2)$ is the commutator $[c_0^{\ball^{(1)}},c_2]_\natural$ by the induction hypothesis.
Note that $[c_0^{\ball^{(1)}},c_2]_\natural$ is the boundary of the embedded strip $\D^1\times[0,1]$ that consists of $A_2$ and the sides of the extensions of the cubes $\ball^{(j)}$, see the left hand side of Figure~\ref{fig:plumbing-interpolating}. Note that we have not drawn the strips in Figure~\ref{fig:plumbing-interpolating} so that $e_1$ and $e_5$ are in their boundaries; this is only due to our inability to draw and we must leave this to reader's imagination.
    
    There is an obvious isotopy of this strip that pulls the lower left horizontal edge $(c_0^{\ball^{(1)}})^-$ to the right and up, and thus isotopes $[c_0^{\ball^{(1)}},c_2]_\natural$ to $c_0^{\ball^{(1)}\natural\,\ball^{(2)}}$. We do the analogous isotopy for $\vec{\tau}_2=e^+$, by pulling the front right vertical edge to the left and back, see the right hand side of Figure~\ref{fig:plumbing-interpolating}.
    
    Finally, for $\vec{\tau}_1=\vec{\tau}_2=e^+$ the isotopy is not the same in the two cases. However, there is an isotopy interpolating between them, that runs both isotopies in parallel, cf.\ the middle of Figure~\ref{fig:plumbing-interpolating}. We omit writing this down explicitly.
\end{defn}

\begin{defn}[Part III  -- isotoping annuli]
    We extend the isotopy $h_{\Gamma,d}$ of the last definition to an isotopy $h_{\Gamma,d}^\e$ of $\S^1\times[-\e,\e]$, simply by observing that the strips that we use for isotoping are built out of 1) pieces of our starting annuli $C_j$, and 2) the neighbourhoods $E_i$ of the edges of the two extended cubes. Once we pick appropriate isotopies of $E_i$, we get a continuous family $h_{\Gamma,d}^\e$.
\end{defn}

\begin{defn}[Part IV -- putting all together]\label{def:partIV}
    Finally, we define the desired map~\eqref{eq-def:mu-star-e} by gluing the isotopy $h_{\Gamma,d}^\e$ from Part III to the map $\mu^{\star,\e}_{\Gamma,d}$ from Part I. More precisely, let $\I_j\coloneqq\I^{n_j(d-3)}$ and $\I_{j,\e}\coloneqq[0,1+\e]^{n_j(d-3)}$, and pick a homeomorphism $\S^{n_j(d-3)}\cong \I_j/\partial$. We define a map on $\I^n\times\I_{1,\e}\times \I_{2,\e}$ by applying $\mu^{\star,\e}_{\Gamma,d}$ on $\I^n\times \I_1\times \I_2$ and $h_{\Gamma,d}^\e$ on $\I^n\times(\I_{1,\e}\times \I_{2,\e}\sm \I_1\times\I_2)$, using the homotopy coordinate as the radial coordinate. The resulting map is constant on the boundary, so defines a map out of $\I^n\times\S^{n(d-3)}$, see~\eqref{eq:smash}.
\end{defn}
\begin{proof}[Proof of Theorem~\ref{thm:multi-family}]
    Let us now verify the properties of $\mu^\star_{\Gamma,d}$ claimed in Theorem~\ref{thm:multi-family}.

    The condition  $(i)$, that $\mu^\star_{\Gamma,d}(\vec{t},\vec{\tau})=\const_{c_0}$ if $\vec{\tau}=e^+$, is fulfilled by construction, since this is the value of $h_{\Gamma,d}$ on $\I^n\times\partial(\I_{1,\e}\times \I_{2,\e})$, equivalently $\{1\}\times\I^n\times(\S^{n_1(d-3)}\times e^+  \vee e^+ \times\S^{n_2(d-3)})$.

    For the condition $(ii)$, we need to show that if $t_i=1$ for some $1\leq i\leq n$ then $\mu^\star_{\Gamma,d}(\vec{t},\vec{\tau})=\const_{c_0}$. Note that for $n=1$ this holds by \eqref{eq-def:mu}. Therefore, the second condition holds by induction.
    % The third condition says that for every $\vec{\tau}\in\S^{n(d-3)}$ the circle $\mu_{\Gamma,d}(\vec{\tau})\coloneqq\mu^\star_{\Gamma,d}(\vec{0},\vec{\tau})$ misses all the arcs $a_i\subseteq\ball^d$, $1\leq i\leq n$. This follows from the last claim. 
    % Indeed, when each $t_i=0$ in the first part of the construction we are using only $\mu_d\colon\S^{d-3}\to\Emb(\S^1,\ball^d)$ (instead of $\mu^\star_d$) which foliates the meridian sphere of $a_i$. Moreover, in the last step of the construction, we perform isotopies which remain in $\partial\ball^d_n$: we use the remaining parts of the annuli $C_i$, and these were obtained by small thickening of the $i$-th copy of $\mu_d$ inside of $\partial\ball^d_n$ (in the direction normal to $\S^{d-2}\subseteq\partial\ball^d$).
    
    For the last claim $(iii)$, we have to show that the circle $\mu^\star_{\Gamma,d}(\vec{t},\vec{\tau})$ intersects the neat arc $a_i\subseteq\ball^d$ if and only if $t_i=\tfrac{1}{2}$ and $\tau_i=e^-$ and $t_j\leq \frac{3}{4}$ for all $1\leq j\leq n$, and in that case the intersection is a single point, the midpoint $a_i(0)$. This follows by induction. For the base case $n=1$ this is true by \eqref{eq:a-1} and \eqref{eq-def:mu-star}. For the induction step simply use the fact that each $a_j$ belongs to exactly one of the balls $\ball^{(1)}$ or $\ball^{(2)}$ in the plumbing $\ball^d_{ext}=(\ball^{(1)}\natural\,\ball^{(2)})_{ext}$, and that the isotopies $h_{\Gamma,d}$ do not intersect any $a_i$. Namely, they use the remaining parts of the annuli $C_i$, and these were obtained by small thickening of the $i$-th copy of $\mu_d$ inside of $\partial\ball^d$ (in the direction normal to $\S^{d-2}\subseteq\partial\ball^d$).
\end{proof}

%~~~~~~~~~~~~~~~~~~~~~~~~~~~~~~~~~~~~~~~~~~~~~~~~~~~~~~~~~~
\section{Relation to Samelson products}\label{sec:relation}

% In Section~\ref{subsec:Samelson} we define Samelson products $x_{\Gamma,d-2}$ and $x^\star_{\Gamma,d-1}$. 
% In Section~\ref{subsec:analogue} we show how they relate to $\mu_{\Gamma,d}$ and $\mu^\star_{\Gamma,d}$.

As mentioned in the introduction, our families $\mu_{\Gamma,d}$ and $\mu^\star_{\Gamma,d}$ of embedded circles are geometric analogues of certain Samelson (or Whitehead) products, namely $x_{\Gamma,d-2}$ and $x^\star_{\Gamma,d-1}$. After defining these maps in Section~\ref{subsec:Samelson}, we show in Section~\ref{subsec:analogue} that $\mu^\star_{\Gamma,d}$ is homotopic to $x^\star_{\Gamma,d-1}$ (when both are included into an appropriate target), whereas $\mu_{\Gamma,d}$ is homotopic to $x_{\Gamma,d-2}$.

\subsection{Homotopy theoretic families: the Samelson products}\label{subsec:Samelson}

As before, fix $d\geq3$. Recall from Section~\ref{subsec:trees} that a tree $\Gamma\in\Tree(n)$, for some $n\geq1$, corresponds to a bracketed word in which each letter $\mathsf{x}^i$ for $1\leq i\leq n$  appears exactly once. We will now replace the letter $\mathsf{x}^i$ by the map
\begin{equation}\label{eq-def:x-i}
    x^i_{d-2}\colon\S^{d-3}\to\Omega\bigvee_n\S^{d-2}.
\end{equation}
This is defined as the map $x_{d-2}\colon\S^{d-3}\to\Omega\S^{d-2}$ from~\eqref{eq-def:x} (the adjoint of the identity map on $\S^{d-3}$), followed by the loop map of the inclusion $\S^{d-2}\hra\bigvee_n\S^{d-2}$ of the $i$-th wedge summand. Moreover, we replace the bracket by the Samelson bracket, defined as follows.

\begin{defn}
    For a tree $\Gamma\in\Tree(n)$ we define inductively in $n\geq1$ the Samelson product of the maps $x^i_{d-2}$ from \eqref{eq-def:x-i} as a map
    \begin{equation}\label{eq-def:x-Gamma}
        x_{\Gamma,d-2}\colon\S^{n(d-3)}\to\Omega\bigvee_n\S^{d-2}.
    \end{equation}
    For $n=1$ there is only one tree $\Gamma=|$ and 
    \[
        x_{|,d-2}\coloneqq x^1_{d-2}=x_{d-2}.
    \]
    For the induction step, let $\Gamma=[\Gamma_1,\Gamma_2]$ be a tree with $\deg(\Gamma)=n=n_1+n_2$ leaves, for $\Gamma_j\in\Tree(S_j)$ and $n_j=|S_j|$, where $S_1\sqcup S_2=\ul{n}$. Let us denote $\I_j\coloneqq\I^{n_j(d-3)}$ and fix a homeomorphism $\S^{n(d-3)}\cong\S^{n_1(d-3)}\wedge\S^{n_2(d-3)}\cong\faktor{\I_1\times\I_2}{\partial}$, as in~\eqref{eq:smash}.
    Then for $\vec{\tau}_j\in\I_j$ we let
    \[
        x_{\Gamma,d-2}(\vec{\tau}_1,\vec{\tau}_2) \coloneqq [x_{\Gamma_1}(\vec{\tau_1}),x_{\Gamma_2}(\vec{\tau}_2)]
    \]
    be the commutator of the loops $x_{\Gamma_j,d-2}(\vec{\tau}_j)\in\Omega\bigvee_{S_j}\S^{d-2}\subseteq \Omega\bigvee_{n}\S^{d-2}$. 
    This is not trivial on the boundary of the cube $\partial(\I_1\times\I_2)$, but at each of its points we have a commutator with the trivial loop, so of the shape $\alpha\alpha^{-1}$ for a loop $\alpha$, and there is a canonical homotopy from such a loop to $\const_{e^+_{d-2}}$. On $\partial\I_1\times\partial\I_2$ the homotopy is constant, since $\alpha$ is trivial as well. Thus, we can 
    use all these homotopies on a collar of $(\partial\I_1)\times\I_2\cup \I_1\times(\partial\I_2)$, to get a map on a slightly bigger $\I_1\times\I_2$, that is constant on the boundary, so defines a map on $\S^{n(d-3)}$.
\end{defn}
\begin{rem}
    One can equivalently define $x_{\Gamma,d-2}$ by directly taking iterated Samelson brackets of $x^i_{d-2}$, see~\cite[App.B]{K-thesis-paper}.
    Moreover, the Samelson bracket is adjoint, up to a sign, to the perhaps more widely known Whitehead bracket $x_{\Gamma,d-2}^{Wh}\colon\S^{n(d-3)+1}\to\bigvee_n\S^{d-2}$, which is defined inductively using the map
    \[
        Wh\colon\S^{n(d-3)+1}\to \S^{n_1(d-3)+1}\vee\S^{n_2(d-3)+1},
    \]
    given as the attaching map of the top cell in $\S^{n_1(d-3)+1}\times\S^{n_2(d-3)+1}$. See \cite[Ch.X.7]{Whitehead}.
\end{rem}

In the following lemma we show that $x_{\Gamma,d-1}$ can be obtained from $x_{\Gamma,d-2}$ by a suspension-like construction. We will use
the homeomorphisms
\[
    \faktor{(\D^1)^n}{\partial}\,\wedge\S^{n(d-3)}\cong
    \S^n\wedge\S^{n(d-3)}\cong
    \bigwedge_n\S^1\wedge\bigwedge_n\S^{d-3}\cong
    \bigwedge_n(\S^1\wedge\S^{d-3})\cong
    \S^{n(d-2)}
\]
and the maps (see the paragraph after~\eqref{eq:smash})
\[
    \I^n\times\S^{n(d-3)}\subseteq(\D^1)^n\times\S^{n(d-3)}\sra\faktor{(\D^1)^n}{\partial}\,\wedge\S^{n(d-3)}.
\]

\begin{prop}\label{prop:suspensions}\hfill
\begin{enumerate}
\item 
    The restriction of $x_{\Gamma,d-1}$ to $e^-\wedge\S^{n(d-3)}\subseteq\S^n\wedge\S^{n(d-3)}\cong\S^{n(d-2)}$ has image in the subspace $\Omega\bigvee_n\S^{d-2}\subseteq\Omega\bigvee_n\S^{d-1}$, using the inclusions of the meridian $\incl\colon\S^{d-2}\hra\S^{d-1}$. Moreover, this restriction agrees with $x_{\Gamma,d-2}$.
    In other words, there is a commutative diagram
    \[
    \begin{tikzcd}
        \S^{n(d-3)}\ar{rr}{x_{\Gamma,d-2}}\dar[hook]{e^-\wedge\Id} && \Omega\bigvee_n\S^{d-2}\dar[hook]{\Omega\bigvee_n\incl} \\
        \S^n\wedge\S^{n(d-3)}\rar{\cong} & \S^{n(d-2)}\rar{x_{\Gamma,d-1}} & \Omega\bigvee_n\S^{d-1}
    \end{tikzcd}
    \]
\item
    The restriction of $x_{\Gamma,d-1}$ to $\I^n\times\S^{n(d-3)}\subseteq(\D^1)^n\times\S^{n(d-3)}$ has image in the subspace $\Omega\bigvee_n\ball^{d-1}\subseteq\Omega\bigvee_n\S^{d-1}$, using the inclusions of the western hemisphere $\mathfrak{j}_W\colon\ball^{d-1}\hra\S^{d-1}$. In other words, if we denote this restriction by $x^\star_{\Gamma,d-1}$, then there is a commutative diagram
    \[
    \begin{tikzcd}
        \I^n\times\S^{n(d-3)}\ar{rr}{x^\star_{\Gamma,d-1}}\dar[hook]{} && \Omega\bigvee_n\ball^{d-1}\dar[hook]{\Omega\bigvee_n\mathfrak{j}_W} \\
        (\D^1)^n\times\S^{n(d-3)}\rar[two heads]{} & \S^n\wedge\S^{n(d-3)}\cong \S^{n(d-2)}\rar{x_{\Gamma,d-1}} & \Omega\bigvee_n\S^{d-1}
    \end{tikzcd}
    \]
\end{enumerate}
\end{prop}
\begin{proof}
   This follows by induction on $\deg\Gamma=n\geq 1$, with the case $n=1$ in~\eqref{eq-def:x-star}. 
\end{proof}

In other words, the map $x^\star_{\Gamma,d-1}$ is the iterated commutator of the arcs foliating the balls $\ball^{d-1}_i$, which is the western hemisphere of $\S^{d-1}_i$ respectively. In fact, we can generalise this to include all possible restrictions of $x_{\Gamma,d-1}$ to subcubes of $(\D^1)^n$, using the following definition.
Let $\mathfrak{j}_E\colon\ball^{d-1}\hra\S^{d-1}$ be the inclusion of the eastern, and $\mathfrak{j}_W\colon\ball^{d-1}\hra\S^{d-1}$ of the western hemisphere as above. Then for any $S\subseteq\ul{n}=\{1,\dots,n\}$ we have the corresponding inclusion
\[
    j_{W/E}^S\coloneqq
    \bigvee_{i\in S}\mathfrak{j}_W
    \vee\bigvee_{i\in\ul{n}\sm S} \mathfrak{j}_E
    \colon\;
    \bigvee_n\ball^{d-1} \hra \bigvee_n\S^{d-1}. 
\]
\begin{defn}\label{def:glueOp}
    Let $A\colon \I^n\to  \Omega\bigvee_n\ball^{d-1}$ be a map such that $A|_{\partial\I^n}$ has image in $\Omega\bigvee_n\partial\ball^{d-1}$. Then we define the map
    \[
        \glueOp_{S\subseteq\ul{n}} \refl_S(A)\colon (\D^1)^n\ra  \Omega\bigvee_n\S^{d-1}
    \]
    by writing $(\D^1)^n=([-1,0]\cup\I)^n=\glueOp_{S\subseteq\ul{n}}[-1,0]^S\times\I^{\ul{n}\sm S}$,
    and for $S\subseteq\ul{n}=\{1,\dots,n\}$ defining:
    \[\begin{tikzcd}
        \refl_S(A)\colon\; 
        [-1,0]^S\times\I^{\ul{n}\sm S}\rar{\cong} &
        \I^{\ul{n}}\rar{A} &
        \Omega\bigvee_n\ball^{d-1}\rar{\Omega j_{W/E}^S} &
        \Omega\bigvee_n\S^{d-1},
    \end{tikzcd}
    \]
    where the first map is the product of the \emph{reflections} $[-1,0]\to[0,1]$ around $0$ and $\Id\colon\I\to\I$.
\end{defn}
In other words, $\glueOp_{S\subseteq\ul{n}} \refl_S(A)$ is obtained by gluing together along the $1$-faces of $\I^n$ suitable ``reflections'' of the cube $A$. 
Completely analogously to Proposition~\ref{prop:suspensions}~(ii) we can check that the restriction of $x_{\Gamma,d-1}$ to a cube $[-1,0]^S\times\I^{\ul{n}\sm S}\subseteq(\D^1)^n$ is a reflection of $x^\star_{\Gamma,d-1}$, proving the following.

\begin{prop}\label{prop:glueOp}
    The map $x_{\Gamma,d-1}$ can be obtained by gluing together reflections of $x^\star_{\Gamma,d-1}$, i.e.\ 
    \[
        x_{\Gamma,d-1}=\glueOp_{S\subseteq\ul{n}} \refl_S(x^\star_{\Gamma,d-1}).
    \]
\end{prop}

For example, if $n=2$ and $d=4$ and $\Gamma=[\mathsf{x}^1,\mathsf{x}^2]$ then $x_{\Gamma,d-2}\colon\S^2\to\Omega(\S^2\vee\S^2)$ maps $(\tau_1,\tau_2)\in(\D^1)^2$ to the commutator of the arcs $m_1(\tau_1)$ and $m_2(\tau_2)$ depicted in Figure~\ref{fig:realmap-2}. If $\tau_j>0$ we use the western hemisphere of $m_j$, whereas if $\tau_j<0$ we use the eastern hemisphere. See also~\eqref{eq-def:x-star} and Figure~\ref{fig:foliation-x}. Thus, $x_{\Gamma,d-2}$ splits into a $\glueOp$-sum of four maps.

%~~~~~~~~~~
\subsection{Multi-families are geometric analogues of Samelson products}
\label{subsec:analogue}

Recall from Section~\ref{subsec:model-ball} that $\ball^d_n$ is the extended $d$-ball $\ball^d_{ext}$ with an arc $a_0\subseteq\partial\ball^d_{ext}$ and a set of neat arcs $a_i\colon\D^1\hra\ball^d_{ext}$, $1\leq i\leq n$. For each arc we fixed a based meridian ball, giving $m\colon\bigvee_n\ball^{d-1}\hra\ball^d_n$, see~\eqref{eq-def:m}. Recall also the map $\mu_{\Gamma,d}\coloneqq\mu^\star_{\Gamma,d}(\vec{0},-)$ from \eqref{eq-def:mu-Gamma}.

\begin{thm}\label{thm:analogues}
    For the map $\mu^\star_{\Gamma,d}$ from \eqref{eq-def:ol-mu} the following square commutes up to homotopy:
    \[
    \begin{tikzcd}
        \I^n\times\S^{n(d-3)}
        \dar[swap]{x^\star_{\Gamma,d-1}}\rar{\mu^\star_{\Gamma,d}} 
        & \Emb((\S^1,\S^1_+),(\ball^d_n,a_0))
        \dar[hook]\\
        \Omega\bigvee_n\ball^{d-1}
        \rar[hook]{m}
        & \Omega\ball^d_n,
    \end{tikzcd}
    \]
    with the homotopy that on $\vec{0}\in\I^n$ restricts to a homotopy making the following commute:
\[
\begin{tikzcd}
    \S^{n(d-3)}
    \dar[swap]{x_{\Gamma,d-2}}\rar{\mu_{\Gamma,d}} 
    & \Emb((\S^1,\S^1_+),(\ball^d\sm\bigsqcup_{i=1}^n a_i,a_0))
    \dar[hook]\\
    \Omega\bigvee_n\S^{d-2}
    \rar[hook]{\partial m}[swap]{\sim}
    & \Omega(\ball^d\sm\bigsqcup_{i=1}^n a_i).
\end{tikzcd}
\]
\end{thm}
\begin{proof}
    % \begin{figure}[!htbp]
    %     \centering
    %     \includestandalone[mode=buildmissing,width=0.4\linewidth]{fig-knotted/fig-wedge-of-disks}
    %     \includestandalone[mode=buildmissing,width=0.4\linewidth]{fig-knotted/fig-flaps}
    %     \caption{Collapsing $\D^2$ into quarters.}
    %     \label{fig:wedge-of-disks}
    %     \label{fig:flaps}
    % \end{figure}
    % The second claim follows immediately from the first by evaluating at $\vec{t}=0$ and using \eqref{eq-def:mu-Gamma} and Proposition~\ref{prop:suspensions}.
    We prove this by induction on $n\geq1$.
    
    The base case $n=1$ is clear from Section~\ref{sec:preliminaries}: the map $\mu^\star_{|,d}=\mu_d^\star$ from Definition~\ref{def:base} was defined in \eqref{eq-def:mu-star} by homotoping the foliation $x^\star_{d-1}$ of $\ball^{d-1}$, and then using the inclusion $\ball^{d-1}\subseteq\ball^d$ of the meridian. Since by definition $x^\star_{|,d-1}=x^\star_{d-1}$ we have the claimed homotopy commutative square.

    Assuming that the statement holds for $\mu^\star_j\coloneqq\mu^\star_{\Gamma_j,d}$ with $\deg(\Gamma_j)=n_j<n$, we check that it also holds for $\mu^\star_{\Gamma,d}$ where $\Gamma=[\Gamma_1,\Gamma_2]$ has $\deg(\Gamma)=n=n_1+n_2$.

    %check on $\I^{n(d-3)}$ first. 
    Since by the induction hypothesis $x^\star_{\Gamma_j,d-1}$ is homotopic to $\mu^\star_j$ after including both targets into $\Omega\ball^d_{n_j}$, we have that $x^\star_{\Gamma,d-1}\coloneqq[x^\star_{\Gamma_1,d-1},x^\star_{\Gamma_2,d-1}]$ is homotopic to the Samelson bracket $[\mu^\star_1,\mu^\star_2]$. Thus, it suffices to show that $\mu^\star_{\Gamma,d}$ is homotopic to $[\mu^\star_1,\mu^\star_2]$. The latter is defined from the commutator map
    \[
        \I^{n_1}\times\I^{n_2}\times\I^{n_1(d-3)}\times\I^{n_2(d-3)}\ni(\vec{t}_1,\vec{t}_2,\vec{\tau}_1,\vec{\tau}_2)
        \quad\mapsto\quad
        [\mu^\star_1(\vec{\tau}_1),\mu^\star_2(\vec{\tau}_2)]\in \Map((\S^1,\S^1_+),(\ball^d_n,a_0)),
    \]
    by canonically nullhomotoping the boundary $\I^{n_1}\times\I^{n_2}\times\partial(\I^{n_1(d-3)}\times\I^{n_2(d-3)})$.
    On the other hand, the multi-family $\mu^\star_{\Gamma,d}$ was defined by canonically nullhomotoping the embedded commutator map which used the push-offs of $c_j=\mu^\star_j(\vec{t}_j,\vec{\tau}_j)$ and some auxiliary strips $E_i$.
    Since we are now working in the space of maps instead of embeddings, we can contract $E_i$ and collapse the push-offs back to $c_j$, so that we precisely get their commutator.
    
    %Then glue in nullhomotopies.
    Finally, this collapse also provide a homotopy between the canonical nullhomotopies in the two cases. Namely, the push-offs bounded the remaining part of the annuli $C_j$, and we constructed the nullhomotopy by dragging along them. After this annulus is collapsed onto the core, this dragging clearly corresponds to the nullhomotopy of the shape $t\mapsto \alpha|_{[0,t]}\alpha|_{[0,t]}^{-1}$.
\end{proof}

%~~~~~~~~~~~~~~~~~~~~~~~~~~~~~~~~~~~~~~~~~~
\section{The Taylor tower}\label{sec:tower}

In Section~\ref{subsec:tower} we discuss the reduced punctured knots model for the Taylor tower for $\Embp(\D^1,M)$. In Section~\ref{subsec:paths} we construct a homotopy from $\ev_{n+1}\realmap_n(\Gamma^{g_{\ul{n}}})$ to a point in the layer $\pF_{n+1}(M)$.

%~~~~~~~~~~
\subsection{The Taylor tower for the arc}\label{subsec:tower}

Recall from Section~\ref{subsec:notation-gen} that $\ul{n}\coloneqq\{1,\dots,n\}$ and that we fix a collection $J_i=[L_i,R_i]\subseteq \D^1$ with $0<L_i<R_i<L_{i+1}$ for $i\geq0$, see~\eqref{eq-def:J-i}. For a set $T$ we denote by $\Delta^T$ the simplex spanned by the vertex set $T$ (so it has dimension $|T|-1$), and if $k\notin T$ we write $Tk\coloneqq T\sqcup\{k\}$. We will use the following definitions and results from \cite{KT-rpn} and~\cite{K-thesis-paper}.

\begin{defn}
\label{def:M-T}
    For a $d$-manifold $M$ with boundary and a neat arc $\u\colon\D^1\hra M$ 
    with a tubular neighbourhood $\nu\u$
    define
    \[
        M_T\coloneqq (M\sm\nu\u)\cup\bigsqcup_{i\in  T}\nu(\u|_{J_i})=M\sm\nu(\u|_{\D^1\sm \bigsqcup_{i\in  T}J_i}).
    \]
    For any $S\subseteq\ul{n}$ define $\Embp(J_0,M_{0S})$ as the space of neat embeddings $\D^1\hra M_{0S}$ which on the boundary agree with $\u|_{J_0}$, see Figure~\ref{fig:nbhd-of-U-S}.
    For $k\notin S\subseteq\ul{n}$ define the map \[
    \rho_{0S}^k\colon\Embp(J_0,M_{0S})\to\Embp(J_0,M_{0Sk})
    \]
    as the postcomposition with the inclusion $i_{0S}^k\colon M_{0S}\hra M_{0Sk}$.
\end{defn}
\begin{figure}[!htbp]
    \centering
    \begin{tikzpicture}[scale=0.8]
    \clip (-2.2,-2.1) rectangle (20.2,2.1);
    % boundary of M
    \fill[blue!50!black,opacity=0.15,draw=black,draw opacity=0.5]
        (18.2,-1.2) -- (20.1,-2) -- (20.1,2) --  (18.2,1.2) ;
    \fill[blue!50!black,opacity=0.15,draw=black,draw opacity=0.5]
        (-.2,1.2) -- (-2.1,2) -- (-2.1,-2) -- (-.2,-1.2) ;
    \draw[semithick]
        (-1,0) -- (2,0) 
        (4,0) -- (6,0) 
        (7,0) -- (9,0) 
        (10,0) -- (12,0) 
        (13,0) -- (15,0) 
        (16,0) -- (19,0) ;
    \draw[color=red, dashed,draw opacity=0]
        (6,0) -- (7,0) node[pos=0.5,below=2pt]{$J_1$}
        (9,0) -- (10,0) node[pos=0.5,below=2pt]{$J_2$}
        (12,0) -- (13,0) node[pos=0.5,below=2pt]{$J_3$}
        (15,0) -- (16,0) node[pos=0.5,below=2pt]{$J_4$};
    \fill[red,draw=black]
        (-1,0) circle (1.5pt) (-0.78,0) node[below]{\small$\mathsf{u}({-}1)$}
        (19,0) circle (1.5pt) node[below]{\small$\mathsf{u}(1)$};

    % spheres
    \foreach \y/\i/\j in {1/-1,5/0,8/1,11/2,14/3,17/4}{
        \draw[blue] (\y-1,0) arc (180:360:1cm and 0.5cm);
        \draw[blue,dotted] (\y-1,0) arc (180:0:1cm and 0.5cm);
        \shade[ball color=blue!20!,opacity=0.3] (\y,0) circle (1cm);
        \draw[blue,semithick] (\y,0) circle (1cm);
        \fill (\y,0) circle (1.5pt) %node[below]{\small$u_{\i}$}
        ;
        \draw (\y,-1) node[below=2pt]{$\mathbb{S}_{\i}$};
        }
    % cylinder left
    \fill[draw=blue,semithick,fill=blue!50!white,fill opacity=0.3]
        (-1,-1) -- (1,-1) to[out=145,in=-145, distance=0.6cm] 
        (1,1) -- (-1,1) to[out=-145,in=145, distance=0.6cm] (-1,-1);
    \draw[semithick,blue] 
        (-1,1) -- (1,1) 
        (1,-1) -- (-1,-1);
    % right cylinder
    \fill[draw=blue,semithick,fill=blue!50!white,fill opacity=0.3]
        (17,1) -- (19,1) to[out=-5,in=5, distance=0.7cm] 
        (19,-1) -- (17,-1) to[out=5,in=-5, distance=0.7cm] (17,1);
    % vertical arcs
    \foreach \x in {5,8,11,14,17,19}{
        \draw[blue]
            (\x,1) to[out=-5,in=5, distance=0.7cm] (\x,-1);
        \draw[blue,dotted]
            (\x,1) to[out=-145,in=145, distance=0.6cm] (\x,-1);
    }
    % vertical arcs on the first sphere
    \foreach \x in {-1,1}{
        \draw[blue,semithick]
            (\x,1) to[out=-145,in=145, distance=0.6cm] (\x,-1);
        \draw[blue, dotted]
            (\x,1) to[out=-5,in=5, distance=0.7cm] (\x,-1);
    }
    \draw[mygreen, thick]
            (2,0) -- (4,0) node[fill=mygreen!30,draw, pos=0.5]{}
            (3, 0) node[below=2pt]{$J_0$};
    \fill[mygreen,draw=black,semithick] 
        (2,0) circle (1.5pt) (4,0) circle (1.5pt);
\end{tikzpicture}
    \caption{A knot $J_0\hra M_{01234}=M\sm\nu\u|_{\D^1\sm (J_0\sqcup J_1\sqcup J_2\sqcup J_3\sqcup J_4)}$ is depicted by a square, with $d=3$.}
    \label{fig:nbhd-of-U-S}
\end{figure}
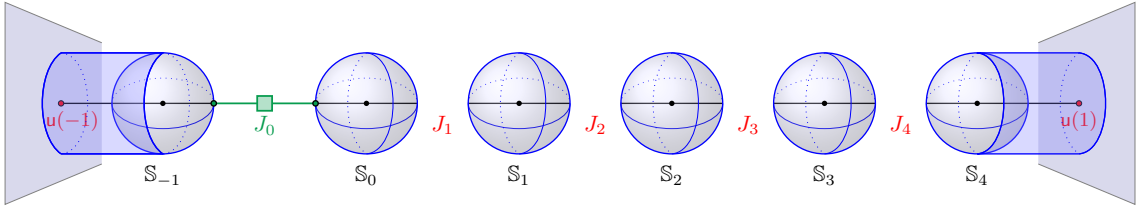
\begin{rem}
\label{rem:knots-J-0}
    For $S=\emptyset$ the inclusion
    \[
        \Embp(J_0,M_0)\hra\Embp(\D^1,M),
    \]
    that sends $K$ to $\u|_{\D^1\sm J_0}\cup K$,
    is a homotopy equivalence, see \cite{KT-rpn}.
\end{rem}

The following definition can be obtained from the standard punctured knots definition by taking (homotopy) fibres in one direction. 
\begin{defn}[{\cite{KT-rpn}}]
\label{def:rpn}
    Let $n\geq1$.
    \begin{enumerate}
\item 
    The $n$-th stage of the Taylor tower $\rpT_n(M)$ is the space of maps
    \[
        f^{\ul{n}}\colon\Delta^{\ul{n}}\to\Embp(J_0,M_{0\ul{n}})
    \]
    such that for $S\subseteq\ul{n}$ the face $\Delta^S\subseteq\Delta^{\ul{n}}$ lands in the subspace 
    \[
        \rho_S^{\ul{n}\sm S}(\Embp(J_0,M_{0S}))\subseteq\Embp(J_0,M_{0\ul{n}}).
    \]
\item
    Define $\p_n\colon \rpT_n(M)\to\rpT_{n-1}(M)$ as the restriction along the face $\Delta^{\ul{n-1}}\subseteq\Delta^{\ul{n}}$.
\item
    Define $\ev_n\colon \Embp(J_0,M_0)\to\rpT_n(M)$ as $\ev_n(K)\coloneqq\const_{\rho^{\ul{n}}_\emptyset(K)}$.\qedhere
\end{enumerate}
\end{defn}
\begin{thm}[Convergence Theorem~{\cite{GKmultiple}}]
\label{thm:rpn}
    If $d\geq4$, $n\geq1$, $0\leq k< (n+1)(d-3)$, the evaluation map induces an isomorphism
    \[
        \pi_k\ev_{n+1}\colon\quad\pi_k\Embp(J_0,M_0)\cong\pi_k \rpT_{n+1}(M).
    \]
\end{thm}
Note that this implies $\evrel_{n+i}\colon\pi_{n(d-3)+i}(\rpT_n(M),\Embp(J_0,M_0))\cong\pi_{n(d-3)+i}(\rpT_n(M),\rpT_{n+1}(M))$ for $0\leq i< d-3$.
We define  the \emph{layer} $\pF_{n+1}(M)\coloneqq\fib_{\ev_n(\u)}(\p_{n+1}\colon \rpT_{n+1}(M)\to \rpT_n(M))$, and in the next lemma describe its points. Its homotopy type will be discussed in Section~\ref{subsec:thesis}.

\begin{lem}[{\cite{KT-rpn}}]
\label{lem:rpn-fiber}
    The map $\p_{n+1}$ is a fibration and a point $f\in \rpT_{n+1}(M)$ belongs to the layer $\pF_{n+1}(M)$
    if and only if the restriction of $f^{\ul{n+1}}\colon\Delta^{\ul{n+1}}\to\Embp(J_0,M_{0\ul{n+1}})$ to $\Delta^{\ul{n}}\subseteq\Delta^{\ul{n+1}}$, included as the union of faces not containing $n+1$, is equal to $\ev_n(\u|_{J_0})^{\ul{n}}=\const_{\rho^{\ul{n}}_\emptyset(\u|_{J_0})}$. 
    
    In fact, $f$ is equivalently described by a map
    \begin{equation}\label{eq:tofib-pt}
        F\colon \I^{\ul{n}}\to \Embp(J_0,M_{0\ul{n+1}})
    \end{equation}
    such that $F(\vec{t})\in\Embp(J_0,M_{0Sn+1})$ if $t_i=0,\forall i\in\ul{n+1}\sm S$, and $F(\vec{t})=\u|_{J_0}$ if $\exists i\in\ul{n+1},t_i=1$.
\end{lem}

%~~~~~~~~~~
\subsection{Knotted families and the Taylor tower for the arc}\label{subsec:paths}

Recall from Definition~\ref{def:knotted-family} that the realisation map $\realmap_n$
sends a decorated tree $\Gamma^{g_{\ul{n}}}\in\Tree_{\pi_1M}(n)$ to the homotopy class of
\[
    \u_{\TG,\Gamma^{g_{\ul{n}}}}=\u|_{\D^1\sm J_0}\cup(\TG_n\circ\mu_{\Gamma,d})|_{\S^1_-}\colon\S^{n(d-3)}\to\Embp(\D^1,M),
\]
defined using the multi-family $\mu_{\Gamma,d}\colon\S^{n(d-3)}\to\Emb(\S^1,\ball^d_n)$ and any grasper $\TG_n\colon\ball^d_n\hra M$ on $\u$ decorated by $g_{\ul{n}}$. In fact, since $\u_{\TG,\Gamma^{g_{\ul{n}}}}$ has image in the subspace of arcs that agree with $\u$ on $\D^1\sm J_0$ (recall from Remark~\ref{rem:knots-J-0} that this subspace is homotopy equivalent to the whole space $\Embp(\D^1,M)$), we have, abusing notation:
\begin{align}\label{eq:realmap}
    \realmap_n\colon
    \Z[\Tree_{\pi_1M}(n)] &\ra\pi_{n(d-3)}\Embp(J_0,M_0),\\
    \Gamma^{g_{\ul{n}}} &\mapsto[\u_{\TG,\Gamma^{g_{\ul{n}}}}=(\TG_n\circ\mu_{\Gamma,d})|_{\S^1_-}].\nonumber
\end{align}
This can be postcomposed with the map $\ev_{n+1}\colon\Embp(J_0,M_0)\to\rpT_{n+1}(M)$ from Definition~\ref{def:rpn}.

\begin{thm}\label{thm:grasper-gives-path}
    The map $\ev_{n+1}(\u_{\TG,\Gamma^{g_{\ul{n}}}})\colon\S^{n(d-3)}\to\rpT_{n+1}(M)$ lifts to the fibre $\pF_{n+1}(M)$. Moreover, it is homotopic to the map
    \[
    \u^\star_{\TG,\Gamma^{g_{\ul{n}}}}\colon\S^{n(d-3)}\ra
    \pF_{n+1}(M)\subset \Map\big(\I^{\ul{n}},\Embp(J_0,M_{0\ul{n+1}})\big)
    \]
    whose adjoint is the composite:
    \begin{equation}\label{eq:grasper-gives-path}
    \begin{tikzcd}
        \u^\star_{\TG,\Gamma^{g_{\ul{n}}}}\colon\I^{\ul{n}}\times\S^{n(d-3)}\rar{\mu^\star_{\Gamma,d}} & 
        \Emb(\S^1,\ball^d_n)\rar{\TG_n\circ-} & 
        \Emb(\S^1,M)\rar{-|_{\S^1_-}} &
        \Embp(J_0,M_{0\ul{n+1}}).
    \end{tikzcd}
    \end{equation}
\end{thm}
Roughly speaking, $\u^\star_{\Gamma,\TG}$ takes $(\vec{t},\vec{\tau})\in\I^{\ul{n}}\times\S^{n(d-3)}$ to the arc $J_0\hra M_{0\ul{n}}$ that near $\u(J_i)$ looks like $\mu^\star_d(t_i,\tau_i)|_{\S^1_-}$, the time $t_i$ of the foliation of the meridian ball bounded by $\mu_d(\tau_i)$, see Section~\ref{subsec:foliations}.
\begin{proof}
    The properties $(ii)$ and $(iii)$ of the family $\mu^\star_{\Gamma,d}$ from Theorem~\ref{thm:multi-family} ensure that:
    \begin{enumerate}[label=(\arabic*)]
            \item for $\vec{t}\in\I^{\ul{n}}$ with some $t_i=1$ we have $\u^\star_{\Gamma,\TG}(\vec{t},-)=\const_{\u|_{J_0}}$;
            \item for $\vec{t}\in\I^{\ul{n}}$ and $\vec{\tau}\in\S^{n(d-3)}$, the embedding $\u^\star_{\Gamma,\TG}(\vec{t},\vec{\tau})\colon J_0\hra M_{0\ul{n+1}}$ intersects $J_i$ only if 
    $t_i=\frac{1}{2}$. In particular, if $t_i=0$ then $\u^\star_{\Gamma,\TG}(\vec{t},\vec{\tau})$ does not intersect $J_i$;
            \item for $\vec{0}\in\I^{\ul{n}}$ we have $\u^\star_{\Gamma,\TG}(\vec{0},-)=\u_{\TG,\Gamma^{g_{\ul{n}}}}=(\TG_n\circ\mu_{\Gamma,d})|_{\S^1_-}$.
    \end{enumerate}
    
    The conditions $(1)$ and $(2)$ ensure that $\u^\star_{\Gamma,\TG}(-,\vec{\tau})$ defines a point in $\pF_{n+1}(M)$ as in Lemma~\ref{lem:rpn-fiber}, for any $\vec{\tau}\in\S^{n(d-3)}$. Recall from that lemma that this is viewed as a point in $\rpT_{n+1}(M)$ using the inclusion $\I^{\ul{n}}\subseteq\Delta^{\ul{n+1}}$ as the subcube containing the vertex $n+1$ in the barycentric subdivision.
    
    On the other hand, for any $\vec{\tau}\in\S^{n(d-3)}$ the point $\ev_{n+1}(\u_{\TG,\Gamma^{g_{\ul{n}}}})(\vec{\tau})\in\rpT_{n+1}(M)$ is given as the constant map $\Delta^{\ul{n+1}}\to\Embp(J_0,M_{0\ul{n+1}})$ with value $\u_{\TG,\Gamma^{g_{\ul{n}}}}(\vec{\tau})$. Since the value of $\u^\star_{\Gamma,\TG}$ at the vertex $n+1$ is $\u_{\TG,\Gamma^{g_{\ul{n}}}}(\vec{\tau})$ by the condition $(2)$, we can define a homotopy from $\u^\star_{\Gamma,\TG}$ to $\ev_{n+1}(\u_{\TG,\Gamma^{g_{\ul{n}}}})$ by gradually extending by this value from the vertex $n+1$ towards the rest of the simplex and eventually pushing out all other values through the face opposite to it. This is a well-defined path in $\rpT_{n+1}(M)$ since at all times the restriction to the faces is contained in the appropriate submanifold $M_{0S}$, using the fact that in $\u^\star_{\Gamma,\TG}$ the interval $\u|_{J_{n+1}}$ was never crossed.    
\end{proof}

\begin{cor}
    The maps $\ev_n(\u_{\TG,\Gamma^{g_{\ul{n}}}})$ and $\ev_n(\const_{\u|_{J_0}})$ are homotopic in $\rpT_n(M)$. In particular, the map $\realmap_n$ from \eqref{eq:realmap} lifts to a map \begin{align}\label{eq:lift}
        \realmap_n\colon\Z[\Tree_{\pi_1M}(n)] & \ra \pi_{n(d-3)+1}(\rpT_n(M),\Embp(C,M)),\\
        \Gamma^{g_{\ul{n}}} & \mapsto [\u_{\TG,\Gamma^{g_{\ul{n}}}}^\star].\nonumber
    \end{align}
\end{cor}

\begin{rem}\label{rem:previous-approach}
    In~\cite{K-thesis-paper} we took a slightly different approach. Namely, the last theorem can be restated as the existence of a path $\Psi^\TG\colon[0,1]\to \Omega^{n(d-3)}\rpT_n(M)$
    from $\ev_n(\u_{\TG,\Gamma^{g_{\ul{n}}}})\in\Omega^{n(d-3)}\rpT_n(M)$ to $\ev_n(\const_\u)$. This then defines a point
    \[
        [\Psi^\TG,\u_{\TG,\Gamma^{g_{\ul{n}}}}]\in \pi_{n(d-3)+1}(\rpT_n(M),\Embp(J_0,M_0)),
    \]
    and applying $\evrel_{n+1}$ gives a point
    \[
    [\Psi^\TG,\ev_{n+1}(\u_{\TG,\Gamma^{g_{\ul{n}}}})]\in\pi_{n(d-3)+1}(\rpT_n(M),\rpT_{n+1}(M)).
    \]
    Since $\p_{n+1}\colon \rpT_{n+1}(M)\to \rpT_n(M)$ is a fibration, we have
    \[
    \pi_{n(d-3)+1}(\rpT_n(M),\rpT_{n+1}(M))\cong\pi_{n(d-3)}\pF_{n+1}(M),
    \]
    and $[\Psi^\TG,\ev_{n+1}(\u_{\TG,\Gamma^{g_{\ul{n}}}})]$ exactly corresponds to $[\u_{\TG,\Gamma^{g_{\ul{n}}}}^\star]\in\pi_{n(d-3)}\pF_{n+1}(M)$. This follows by making the last isomorphism explicit, which is related to Lemma~\ref{lem:rpn-fiber}, see \cite[Lem.2.10]{K-thesis-paper}.
\end{rem}

%~~~~~~~~~~~~~~~~~~~~~~~~~~~~~~~~~~~~~~~~~~~~~~~~~~
\section{Proof of the main result}\label{sec:proof}

In Section~\ref{subsec:thesis} we summarise the results from \cite{K-thesis-paper} that will be used.
In Section~\ref{subsec:main-proof-arc} we prove the case $C=\D^1$ of the main Theorem~\ref{thm:main}, and in Section~\ref{subsec:main-proof-circle} we use that to deduce the case $C=\S^1$. In Section~\ref{subsec:proofs-cor} we prove Corollaries~\ref{cor:main} and ~\ref{cor:Dd}.

%~~~~~~~~~~
\subsection{A recollection}\label{subsec:thesis}

Recall that Theorem~\ref{thm:main} states that the map~\eqref{eq:lift} factors through the quotient by the relations $(\AS)$ and $(\IHX)$ from Section~\ref{subsec:trees}, and
% is an isomorphism from
that postcomposing the resulting map out of the group of decorated Lie trees $\Lie_{\pi_1M}(n)$ with the homomorphism
\[\begin{tikzcd}
     \pi_{n(d-3)+1}(\rpT_n(M),\Embp(J_0,M_0))\rar{\evrel_{n+1}} &\pi_{n(d-3)+1}(\rpT_n(M),\rpT_{n+1}(M))\cong\pi_{n(d-3)}\pF_{n+1}(M)
\end{tikzcd}
\]
and then an explicit isomorphism $W(\chi\circ\deriv)_*\colon \pi_{n(d-3)}\pF_{n+1}(M)\overset{\cong}{\ra} \Lie_{\pi_1M}(n)$
gives the identity. In this section we briefly discuss this last isomorphism from our previous work.

\begin{thm}[{\cite[Thm.B]{K-thesis-paper}}]
\label{thm:thesis}
    For any $d\geq3$ and $n\geq1$, the lowest nonvanishing homotopy group of the layer $\pF_{n+1}(M)$ in the Taylor tower for $\Embp(\D^1,M)$ admits an isomorphism
        \begin{equation}\label{eq:F-Lie}
        \begin{tikzcd}
            \pi_{n(d-3)}\pF_{n+1}(M)\rar{(\chi\circ\deriv)_*}[swap]{\sim} & \pi_{n(d-2)}\tofib(\Omega M_{\bull},\Omega\lambda)\rar{W}[swap]{\sim}
            & \Lie_{\pi_1M}(n),
        \end{tikzcd}
        \end{equation}
\end{thm} 

We first make a few remarks on the objects that appear in the statement. Firstly, $\Lie_{\pi_1M}(n)$ is the group of decorated Lie trees from Section~\ref{subsec:trees}. Secondly, the isomorphism $W$ will be discussed only later, see~\eqref{eq:final-gens}. Thirdly, $\Omega M_{\bull n+1}$ is the (contravariant) cube defined by the spaces $\Omega M_{Sn+1}$ for $S\subseteq\ul{n}$, where
\[  
    M_{Sn+1}\coloneqq M_{0Sn+1}\cup\ball_{0i_1},
\]
and maps $\Omega \lambda^k_S$ induced by the maps $\lambda^k_{S}\colon M_{Skn+1}\hra M_{Sn+1}$ that ``erase then drag'', see Figure~\ref{fig:M-erase}. 

Finally, a point in the total homotopy fibre $f\in\tofib(\Omega M_{\bull},\Omega\lambda)$ is by definition a collection of maps $f^S\colon\I^{\ul{n}\sm S}\to \Omega M_{Sn+1}$ such that $(\Omega \lambda^k_S)\circ f^{kS}=f^{S}|_{t_k=0}$, the restriction to the face $\I^{\ul{n}\sm Sk}\times \{0\}\subseteq\I^{\ul{n}\sm S}$ (and $f^{S}|_{t_i=1}=\const$, the basepoint). This collection is determined by $f\coloneqq f^{\ul{n}}\in\Omega M_{\ul{n+1}}$.
\begin{figure}[!htbp]
    \centering
    \begin{tabular}{@{}c@{}}
\begin{tikzpicture}[scale=0.8]
    \clip (-2.2,-2.1) rectangle (18.2,2.1);
    % boundary of M
    \fill[blue!50!black,opacity=0.15,draw=black,draw opacity=0.5]
        (16.2,-1.2) -- (18.1,-2) -- (18.1,2) --  (16.2,1.2) ;
    \fill[blue!50!black,opacity=0.15,draw=black,draw opacity=0.5]
        (-0.2,1.2) -- (-2.1,2) -- (-2.1,-2) -- (-0.2,-1.2) ;
    % cylinder left
    \fill[draw=blue,semithick,fill=blue!50!white,fill opacity=0.3]
        (-1,-1) -- (1,-1) to[out=145,in=-145, distance=0.6cm] 
        (1,1) -- (-1,1) to[out=-145,in=145, distance=0.6cm] (-1,-1);
    \draw[semithick,blue] 
        (-1,1) -- (1,1) 
        (1,-1) -- (-1,-1);
    % vertical arcs on the first sphere
    \foreach \x in {-1,1}{
        % \draw[blue,semithick]
        %     (\x,1) to[out=-145,in=145, distance=0.6cm] (\x,-1);
        \draw[blue, dotted]
            (\x,1) to[out=-5,in=5, distance=0.7cm] (\x,-1);
    }
    % right cylinder
    \fill[draw=blue,semithick,fill=blue!50!white,fill opacity=0.3]
        (15,1) -- (17,1) to[out=-5,in=5, distance=0.7cm] 
        (17,-1) -- (15,-1) to[out=5,in=-5, distance=0.7cm] (15,1);
    % old
    \draw[semithick]
        (-1,0) -- (2,0)
        node[pos=0.5, below=0.8cm,blue]{$\mathbb{B}_{-1}$}
        (5,0) -- (7,0) node[pos=0.5, below=0.8cm,blue]{$\mathbb{B}_{12}$}
        (8,0) -- (10,0)  -- (11,0)   -- (13,0)
        (14,0) -- (17,0) node[pos=0.5, below=0.8cm,blue]{$\mathbb{B}_{4}$} ;
    \draw[thick, red]
        (2.5,0) node[below]{$J_{0}$}
        (4.5,0) node[below]{$J_{1}$}
        (7.5,0) node[below]{$J_{2}$}
        (13.5,0) node[below]{$J_{4}$}
        % (16.5,0) node[below]{$J_{5}$}
        ;
    % \draw[mygreen, thick]
    %     (0,0) -- (2,0) node[fill=mygreen!20,draw, pos=0.5]{}
    %     (1, 0) node[below=2pt]{$J_0$};
    \foreach \y/\ytext in {1/0,6/1,9/2,12/3,15/4}{
        \fill[fill=black] (\y,0) circle (1.5pt) 
        % node[below]{\small$u_{\ytext}$} 
        ;
    }
    %spheres
    \foreach \y in {1,6,9,12,15}{
        \draw[blue] (\y-1,0) arc (180:360:1cm and 0.5cm);
        \draw[blue,dotted] (\y-1,0) arc (180:0:1cm and 0.5cm);
        \shade[ball color=blue!20!,opacity=0.3] (\y,0) circle (1cm);
        \draw[blue,semithick] (\y,0) circle (1cm);
    }
    % blue cylinder
    \fill[draw=blue,fill=blue!50!white,fill opacity=0.3]
        (9,1) -- (12,1) to[out=-5,in=5, distance=0.7cm] (12,-1) -- (9,-1) to[out=-5,in=5, distance=0.7cm] (9,1);
    \draw[color=blue] (10.5, -1) node[below]{$\mathbb{B}_{24}$};
    % vertical arcs
    \foreach \x in {6,9,12,15,17}{
        \draw[blue]
            (\x,1) to[out=-5,in=5, distance=0.7cm] (\x,-1);
        \draw[blue,dotted]
            (\x,1) to[out=-145,in=145, distance=0.6cm] (\x,-1);
    }
    % \foreach \y in {0,2}{
    %     \fill[mygreen,draw=black,semithick] (\y,0) circle (1.5pt);
    % }
    \fill[red,draw=black]
        (-1,0) circle (1.5pt) (-0.78,0) node[below]{\small$\mathsf{u}({-}1)$}
        (17,0) circle (1.5pt) node[below]{\small$\mathsf{u}(1)$};
\end{tikzpicture}

\\

\begin{tikzpicture}[scale=0.8]
    \clip (-2.2,-2.1) rectangle (18.2,2.1);
    % boundary of M
    \fill[blue!50!black,opacity=0.15,draw=black,draw opacity=0.5]
        (16.2,-1.2) -- (18.1,-2) -- (18.1,2) --  (16.2,1.2) ;
    \fill[blue!50!black,opacity=0.15,draw=black,draw opacity=0.5]
        (-0.2,1.2) -- (-2.1,2) -- (-2.1,-2) -- (-0.2,-1.2) ;
    % cylinder left
    \fill[draw=blue,semithick,fill=blue!50!white,fill opacity=0.3]
        (-1,-1) -- (1,-1) to[out=145,in=-145, distance=0.6cm] 
        (1,1) -- (-1,1) to[out=-145,in=145, distance=0.6cm] (-1,-1);
    \draw[semithick,blue] 
        (-1,1) -- (1,1) 
        (1,-1) -- (-1,-1);
    % vertical arcs on the first sphere
    \foreach \x in {-1,1}{
        % \draw[blue,semithick]
        %     (\x,1) to[out=-145,in=145, distance=0.6cm] (\x,-1);
        \draw[blue, dotted]
            (\x,1) to[out=-5,in=5, distance=0.7cm] (\x,-1);
    }
    % right cylinder
    \fill[draw=blue,semithick,fill=blue!50!white,fill opacity=0.3]
        (15,1) -- (17,1) to[out=-5,in=5, distance=0.7cm] 
        (17,-1) -- (15,-1) to[out=5,in=-5, distance=0.7cm] (15,1);
    % old
    \draw[semithick]
        (-1,0) -- (2,0)
        node[pos=0.5, below=0.8cm,blue]{$\mathbb{B}_{-1}$}
        (5,0) -- (7,0) node[pos=0.5, below=0.8cm,blue]{$\mathbb{B}_{12}$}
        (14,0) -- (17,0) node[pos=0.5, below=0.8cm,blue]{$\mathbb{B}_{4}$} ;
    \draw[thick, red]
        (2.5,0) node[below]{$J_{0}$}
        (4.5,0) node[below]{$J_{1}$}
        (7.5,0) node[below]{$J_{2}$}
        (13.5,0) node[below]{$J_{4}$}
        % (16.5,0) node[below]{$J_{5}$}
        ;
    % \draw[mygreen, thick]
    %     (0,0) -- (2,0) node[fill=mygreen!20,draw, pos=0.5]{}
    %     (1, 0) node[below=2pt]{$J_0$};
    \foreach \y/\ytext in {1/0,6/1,15/4}{
        \fill[fill=black] (\y,0) circle (1.5pt) 
        % node[below]{\small$u_{\ytext}$} 
        ;
    }
    %spheres
    \foreach \y in {1,6,15}{
        \draw[blue] (\y-1,0) arc (180:360:1cm and 0.5cm);
        \draw[blue,dotted] (\y-1,0) arc (180:0:1cm and 0.5cm);
        \shade[ball color=blue!20!,opacity=0.3] (\y,0) circle (1cm);
        \draw[blue,semithick] (\y,0) circle (1cm);
    }
    % vertical arcs
    \foreach \x in {6,15,17}{
        \draw[blue]
            (\x,1) to[out=-5,in=5, distance=0.7cm] (\x,-1);
        \draw[blue,dotted]
            (\x,1) to[out=-145,in=145, distance=0.6cm] (\x,-1);
    }
    % \foreach \y in {0,2}{
    %     \fill[mygreen,draw=black,semithick] (\y,0) circle (1.5pt);
    % }
    \fill[red,draw=black]
        (-1,0) circle (1.5pt) (-0.78,0) node[below]{\small$\mathsf{u}({-}1)$}
        (17,0) circle (1.5pt) node[below]{\small$\mathsf{u}(1)$};
\end{tikzpicture}

\\

\begin{tikzpicture}[scale=0.8]
    \clip (-2.2,-2.1) rectangle (18.2,2.1);
    % boundary of M
    \fill[blue!50!black,opacity=0.15,draw=black,draw opacity=0.5]
        (16.2,-1.2) -- (18.1,-2) -- (18.1,2) --  (16.2,1.2) ;
    \fill[blue!50!black,opacity=0.15,draw=black,draw opacity=0.5]
        (-0.2,1.2) -- (-2.1,2) -- (-2.1,-2) -- (-0.2,-1.2) ;
    % cylinder left
    \fill[draw=blue,semithick,fill=blue!50!white,fill opacity=0.3]
        (-1,-1) -- (1,-1) to[out=145,in=-145, distance=0.6cm] 
        (1,1) -- (-1,1) to[out=-145,in=145, distance=0.6cm] (-1,-1);
    \draw[semithick,blue] 
        (-1,1) -- (1,1) 
        (1,-1) -- (-1,-1);
    % vertical arcs on the first sphere
    \foreach \x in {-1,1}{
        % \draw[blue,semithick]
        %     (\x,1) to[out=-145,in=145, distance=0.6cm] (\x,-1);
        \draw[blue, dotted]
            (\x,1) to[out=-5,in=5, distance=0.7cm] (\x,-1);
    }
    % right cylinder
    \fill[draw=blue,semithick,fill=blue!50!white,fill opacity=0.3]
        (15,1) -- (17,1) to[out=-5,in=5, distance=0.7cm] 
        (17,-1) -- (15,-1) to[out=5,in=-5, distance=0.7cm] (15,1);
    % old
    \draw[semithick]
        (-1,0) -- (2,0)
        node[pos=0.5, below=0.8cm,blue]{$\mathbb{B}_{-1}$}
        (5,0) -- (13,0)
        (14,0) -- (17,0) node[pos=0.5, below=0.8cm,blue]{$\mathbb{B}_{4}$}
        ;
    \draw[thick, red]
        (2.5,0) node[below]{$J_{0}$}
        (4.5,0) node[below]{$J_{1}$}
        (13.5,0) node[below]{$J_{4}$}
        % (16.5,0) node[below]{$J_{5}$}
        ;
    % \draw[mygreen, thick]
    %     (0,0) -- (2,0) node[fill=mygreen!20,draw, pos=0.5]{}
    %     (1, 0) node[below=2pt]{$J_0$};
    \foreach \y/\ytext in {1/0,6/1,9/2,12/3,15/4}{
        \fill[fill=black] (\y,0) circle (1.5pt) ;
    }
    %spheres
    \foreach \y in {1,6,9,12,15}{
        \draw[blue] (\y-1,0) arc (180:360:1cm and 0.5cm);
        \draw[blue,dotted] (\y-1,0) arc (180:0:1cm and 0.5cm);
        \shade[ball color=blue!20!,opacity=0.3] (\y,0) circle (1cm);
        \draw[blue,semithick] (\y,0) circle (1cm);
        }
    %blue cylinder
    \fill[draw=blue,fill=blue!50!white,fill opacity=0.3]
        (6,1) -- (12,1) to[out=-5,in=5, distance=0.7cm] (12,-1) -- (6,-1) to[out=-5,in=5, distance=0.7cm] (6,1);
    \draw[color=blue] (9, -1) node[below]{$\mathbb{B}_{14}$};
    % vertical arcs
    \foreach \x in {6,9,12,15,17}{
        \draw[blue]
            (\x,1) to[out=-5,in=5, distance=0.7cm] (\x,-1);
        \draw[blue,dotted]
            (\x,1) to[out=-145,in=145, distance=0.6cm] (\x,-1);
    }
    % \foreach \y in {0,2}{
    %     \fill[mygreen,draw=black,semithick] (\y,0) circle (1.5pt);
    % }
    \fill[red,draw=black]
        (-1,0) circle (1.5pt) (-0.78,0) node[below]{\small$\mathsf{u}({-}1)$}
        (17,0) circle (1.5pt) node[below]{\small$\mathsf{u}(1)$};
\end{tikzpicture}
\end{tabular}
    \caption[The map $\lambda^2_{1,4,5}$.]{ The map $\lambda^2_{1,4,5}$ takes the top manifold $M_{124}$ (obtained from $M$ by removing the depicted balls and half-balls) to the bottom manifold $M_{14}$. From the top to the middle we erase the ball $\ball_{24}$, and from the middle to the bottom we apply a diffeomorphism mapping $\partial\ball_{12}$ to $\partial\ball_{14}$. Here $d=3$.}
    \label{fig:M-erase}
\end{figure}
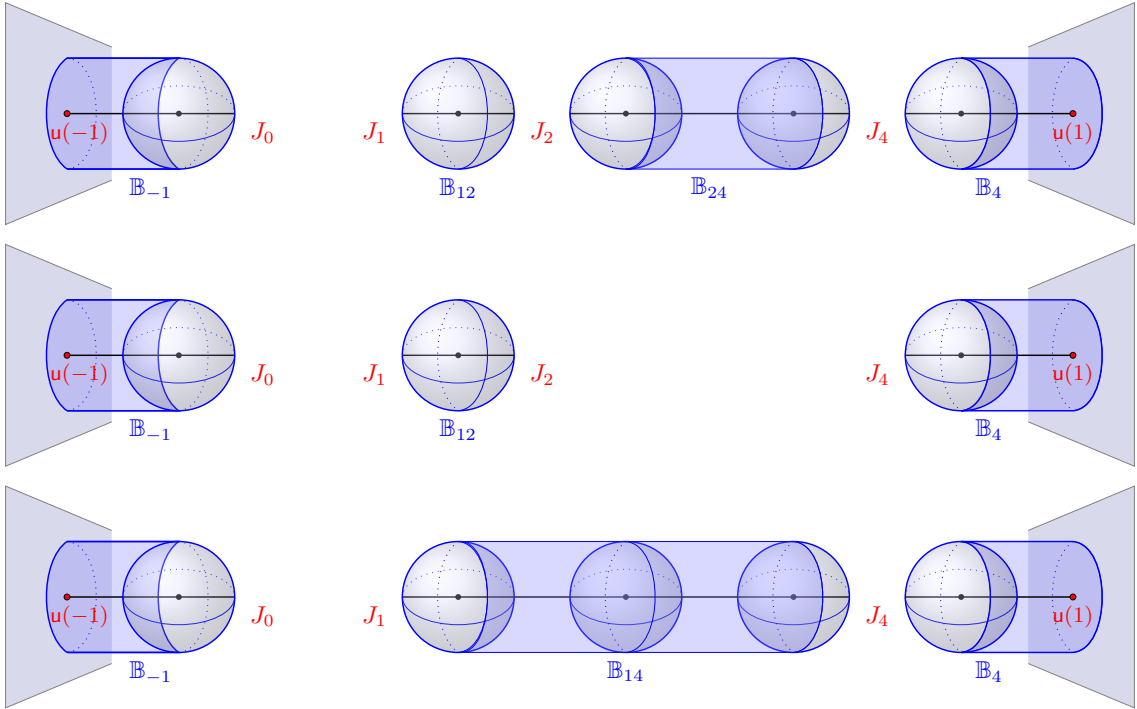

% \begin{rem}
%     A point $f\in\tofib\Omega M_{\bull}$ is a collection of maps $f^S\colon\I^S\to \Omega M_S$ such that $(\Omega i^k_S)\circ f^S=f^{kS}|_{\I^S}$, the restriction to the face $\I^S\subseteq\I^{\ul{n}}$. Whenever a map $F\colon\I^{\ul{n}}\to \Omega M_{\ul{n}}$ has the property that its restriction to each face $\I^S\subseteq\I^{\ul{n}}$ has image in the subspace $\Omega i^k_S(M_S)\subseteq \Omega M_{\ul{n}}$, we say that $F$ \emph{canonically extends} to the total homotopy fibre $\tofib \Omega M_{\bull}$.
% \end{rem}

Let $C\colon\I^{\ul{n}}\to\Omega^{n(d-3)}\Embp(J_0,M_{0\ul{n+1}})$ be a map that defines a point in $\Omega^{n(d-3)}\pF_{n+1}(M)$ (as in~\eqref{eq:tofib-pt} from Lemma~\ref{lem:rpn-fiber}), and let us now describe the point $\chi\deriv(A)\in\Omega^{n(d-2)}\tofib(\Omega M_{\bull},\Omega\lambda)$ by giving a map
\[
    \chi\deriv(C)^{\ul{n+1}}\colon\S^n\to\Omega^{n(d-3)} \Omega M_{\ul{n+1}}.
\]
Firstly, let $\deriv\colon\Embp(J_0,M_{0\ul{n+1}})\to \Omega M_{\ul{n+1}}$ be the map that concatenates the given embedding $J_0\hra M_{0\ul{n+1}}$ with $\u|_{J_0}^{-1}$ to get a loop in $M_{0\ul{n+1}}$, and then includes it into $M_{\ul{n+1}}$. Consider
\begin{equation}\label{eq:final-pt-cube-digression}
\begin{tikzcd}
     \I^{\ul{n}}\rar{C} &
     \Omega^{n(d-3)}\Embp(J_0,M_{0\ul{n+1}})\rar{\deriv} &
     \Omega^{n(d-3)}\Omega M_{\ul{n+1}}.
\end{tikzcd}
\end{equation}
By~\cite[Rem.7.1]{K-thesis-paper} the desired map $\chi\deriv(C)^{\ul{n+1}}$ is obtained from $\deriv C$ by a gluing operation:
\begin{equation}\label{eq:drag-refl}
    \chi\deriv(C)=\glueOp_{S\subseteq\ul{n}}\refl_S^{drag}(\deriv C).
\end{equation}
Here we identify $\S^n$ with the quotient of $(\D^1)^n$ by the boundary, and as in Definition~\ref{def:glueOp} write $(\D^1)^n=\glueOp_{S\subseteq\ul{n}}[-1,0]^S\times\I^{\ul{n}\sm S}$.
However, whereas in that definition we used the reflections $\refl_S$ and maps $j_{W/E}^S$, now as the target instead of $\Omega\bigvee\S^{d-1}$ we have $\Omega M_{\ul{n+1}}$, so we need to define
\begin{equation}\label{eq:refl-drag}
    \refl_S^{drag}(\deriv C)\colon [-1,0]^S\times\I^{\ul{n}\sm S}\to \Omega M_{\ul{n+1}}.
\end{equation}
For this we use certain ``dragging'' homotopies to ``reflect'' the cube $\deriv C$ across the $0$-faces, as in See Figure~\ref{fig:reflections}, and with respect to the balls $\S_i\coloneqq\partial\ball_{ii+1}$ for $i\in S$, as in Figure~\ref{fig:dragging-reflection}. 
\begin{figure}[!htbp]
    \centering
    \begin{tikzcd}[column sep=large]
\begin{tikzpicture}[scale=0.85,font=\small]
    \fill[red,thick,draw=red]
        (-1.25,-0.5) circle (0.8pt) node[below]{\tiny$(\deriv C)^\emptyset$};
    \fill[blue,thick,draw=blue]
        (-0.5,0.5) circle (0.8pt) %node[left]{\tiny$r^1(\deriv A)^\emptyset$}
        -- (0.75,1.75) circle (0.8pt) node[pos=0.5, left]{$(\deriv C)^1$} node[left]{\tiny$*$};
    \fill[orange,thick,draw=orange]
        (-0.25,-0.5) circle (0.8pt) %node[below]{\tiny$r^2(\deriv A)^\emptyset$}
        -- (1.75,-0.5) circle (0.8pt) node[pos=0.5, below]{$(\deriv C)^2$} node[below]{\tiny$*$};
    \draw[fill=green!10!white,text=green!40!black]
        (0.5,0.5) circle (0.8pt) %node[below]{\tiny$rr(\deriv A)^\emptyset$}
        -- (2.5,0.5) %node[pos=0.5, below]{\scriptsize$r^1(\deriv A)^2$} circle (0.8pt) 
        node[below]{\tiny$*$}
        -- (3.75,1.75) node[pos=0.5,right]{$\const_*$}
        -- (1.75,1.75) node[pos=0.5, above]{$\const_*$}
        -- (0.5,0.5) %node[pos=0.5,left]{\scriptsize$r^2(\deriv A)^1$} 
        node[pos=0.5, right=0.5cm, black]{$\deriv C$};
\end{tikzpicture}
    \arrow[|->]{r}{\chi} &\quad
\begin{tikzpicture}[font=\small]
    \clip (-3.2,-1.15) rectangle (4.1,2);
    % front rombs
    \draw[pattern=horizontal lines, pattern color= green,text=green!40!black]
        (-2.75,-0.75)
        -- (-0.75,-0.75) node[pos=0.5, below]{$\const_*$}
        -- (0.5,0.5)
        -- (-1.5,0.5)
        -- (-2.75,-0.75) node[pos=0.5,left]{$\const_*$} node[pos=0.5, right=0.1cm, black]{$\refl_{12}(\deriv C)$};
    \draw[pattern=vertical lines, pattern color= green,text=green!40!black]
        (-0.75,-0.75)
        -- (1.25,-0.75) node[pos=0.5, below]{$\const_*$}
        -- (2.5,0.5) node[pos=0.5,right]{$\const_*$}
        -- (0.5,0.5)
        -- (-0.75,-0.75) %node[pos=0.8]{\tiny$(r^2f^1)^{h^1}$} 
        node[pos=0.5, right=0.2cm, black]{$\refl_1(\deriv C)$};
    % back rombs
    \draw[fill=green!30!white, text=green!40!black]
        (-1.5,0.5)  %node[left]{\tiny$*$}
        -- (0.5,0.5) %node[pos=0.5]{\scriptsize$(r^1f^2)^{h^2}$}
        -- (1.75,1.75) %node[pos=0.5,right]{$\const_*$}
        -- (-0.25,1.75) node[pos=0.5, above=-3pt]{$\const_*$}
        -- (-1.5,0.5) node[pos=0.5,left]{$\const_*$}
        node[pos=0.5, right=0.1cm, black]{$\refl_2(\deriv C)$};
    \draw[fill=green!10!white,text=green!40!black]
        (0.5,0.5) circle (0.9pt) %node[above]{\scriptsize$rrf^\emptyset$}
        -- (2.5,0.5) %node[pos=0.5]{\scriptsize$r^1f^2$}
        -- (3.75,1.75) node[pos=0.5,right]{$\const_*$}
        -- (1.75,1.75) node[pos=0.5, above=-3pt]{$\const_*$}
        -- (0.5,0.5) %node[pos=0.5]{\footnotesize$r^2f^1$} 
        node[pos=0.5, right=0.5cm, black]{$\deriv C$};
\end{tikzpicture}
\end{tikzcd}
    \caption{Gluing together the reflections of the cube family $\deriv C$ gives the sphere family $\chi\deriv(C)^{\ul{n+1}}$.}
    \label{fig:reflections}
\end{figure}
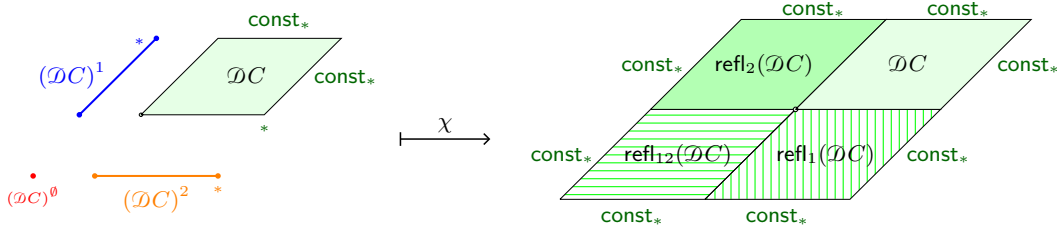

We explain how ``dragging'' looks near one of the balls $\ball_{ii+1}$; these are then combined analogously how various $j_E$'s are combined in the definition of $j_{W/E}^S$, see~\cite[Lem.4.9]{K-thesis-paper}.
In Figure~\ref{fig:dragging-reflection} we depicted $M_{12}$ and the map $\deriv C\colon\I\to\Omega^{d-3}\Omega M_{12}$ as a homotopy across the yellow disk (with $\I$ as the homotopy parameter). For simplicity we have $d=3$, but in general one should replace every loop by a $(d-3)$-parameter family parameterised by the meridian sphere $\S_1=\partial\ball_{12}$. The map $\refl_S^{drag}(\deriv C)$ is given by modifying the yellow disk, so that its lower part is replaced by the union of the arcs that swing around the other hemisphere of $\S_1$; in other words, gluing together the old homotopy and the new homotopy gives the sphere $\S_1$.
\begin{figure}[!htbp]
    \centering
        \begin{tikzpicture}[scale=0.8]
    \clip (-2.2,-2.1) rectangle (12.2,5);
% \begin{tikzpicture}[scale=0.7,every node/.style={scale=0.9},even odd rule,line width=0.2mm]
    % \clip (2.5,-1.33) rectangle (13.5,5.27);
    % BLACK HOLE
        \draw[thick, red]
        
        (4.5,0) node[below]{\small$J_{1}$}
        (7.6,0) node[below]{\small$J_{2}$}
        % (13.5,0) node[below]{$J_{4}$}
        % (16.5,0) node[below]{$J_{5}$}
        ;
    \fill[color=black!20!white] (3.4,3.5) circle (0.5cm);
    \draw[black,densely dotted,postaction=decorate,decoration = {markings, mark = at position 0.65 with {\arrowreversed{>}} }] (3.4,3.5) circle (0.65cm) node[below=0.48cm]{$g$};
    \fill[yellow!50,opacity=0.5]
                    (2,0) to[out=90, in=-80, distance=1.8cm]
                    (2.4,3.15) to[out=100, in=90, distance=2.3cm]
                    (4.5,3.15) to[out=-90, in=80, distance=1.5cm]
                    (6.1,2) to[in=40,out=-120, distance=0.07cm] 
                        (6.2,1.06)  to[out=-80, in=80, distance=0.3cm]
                        (6.1,-1.06)  to[in=-95,out=-100, distance=1.15cm]
                    (5.6,1.8) to[out=80,in=-90, distance=1.2cm] 
                    (4.3,3) to[in=100, out=90, distance=2cm]
                    (2.6,3) to[in=90, out=-80, distance=1cm]
                    (3,0) -- cycle ;
    \draw[black,semithick,opacity=0.5]
                    (2,0) to[out=90, in=-80, distance=1.8cm]
                    (2.4,3.15) to[out=100, in=90, distance=2.3cm]
                    (4.5,3.15) to[out=-90, in=80, distance=1.5cm]
                    (6.1,2) to[in=40,out=-120, distance=0.07cm] 
                        (6.2,1.06)  
                        (6.1,-1.06)  to[in=-95,out=-100, distance=1.15cm]
                    (5.6,1.8) to[out=80,in=-90, distance=1.2cm] 
                    (4.3,3) to[in=100, out=90, distance=2cm]
                    (2.6,3) to[in=90, out=-80, distance=1cm]
                    (3,0)  ;
    \draw[black,semithick,opacity=0.5,dotted]
        (6.2,1.06)  to[out=-80, in=80, distance=0.3cm]
                        (6.1,-1.06);

    % for crossings
    \fill[white] (8.2,-0.1) rectangle (9,0.1) ;
    % % sphere
    %     \shade[ball color=blue!5!white,opacity=0.6] (10,0) circle (1cm);
    %     \draw (10-1,0) arc (180:360:1cm and 0.5cm);
    %     \draw[dotted] (10-1,0) arc (180:0:1cm and 0.5cm);
    %     \draw (10,0) circle (1cm);

    % boundary of M
    \fill[blue!50!black,opacity=0.15,draw=black,draw opacity=0.5]
        (10.2,-1.2) -- (12.1,-2) -- (12.1,2) --  (10.2,1.2) ;
    \fill[blue!50!black,opacity=0.15,draw=black,draw opacity=0.5]
        (-0.2,1.2) -- (-2.1,2) -- (-2.1,-2) -- (-0.2,-1.2) ;
    % cylinder left
    \fill[draw=blue,semithick,fill=blue!50!white,fill opacity=0.3]
        (-1,-1) -- (1,-1) to[out=145,in=-145, distance=0.6cm] 
        (1,1) -- (-1,1) to[out=-145,in=145, distance=0.6cm] (-1,-1);
    \draw[semithick,blue] 
        (-1,1) -- (1,1) 
        (1,-1) -- (-1,-1);
    % vertical arcs on the first sphere
    \foreach \x in {-1,1}{
        % \draw[blue,semithick]
        %     (\x,1) to[out=-145,in=145, distance=0.6cm] (\x,-1);
        \draw[blue, dotted]
            (\x,1) to[out=-5,in=5, distance=0.7cm] (\x,-1);
    }
    % right cylinder
    \fill[draw=blue,semithick,fill=blue!50!white,fill opacity=0.3]
        (9,1) -- (11,1) to[out=-5,in=5, distance=0.7cm] 
        (11,-1) -- (9,-1) to[out=5,in=-5, distance=0.7cm] (9,1);
    % old
    \draw[semithick]
        (-1,0) -- (2,0)
        node[pos=0.5, below=1.1cm,blue]{\small$\mathbb{S}_{-1}$}
        (5,0) -- (7,0) node[pos=0.5, below=1.1cm,blue]{\small$\mathbb{S}_{1}$}
        (8,0) -- (11,0) node[pos=0.5, below=1.1cm,blue]{\small$\mathbb{S}_{2}$} ;

    % \draw[mygreen, thick]
    %     (0,0) -- (2,0) node[fill=mygreen!20,draw, pos=0.5]{}
    %     (1, 0) node[below=2pt]{$J_0$};
    \foreach \y/\ytext in {1/0,6/1,9/4}{
        \fill[fill=black] (\y,0) circle (1.5pt) 
        % node[below]{\small$u_{\ytext}$} 
        ;
    }
    %spheres
    \foreach \y in {1,6,9}{
        \draw[blue] (\y-1,0) arc (180:360:1cm and 0.5cm);
        \draw[blue,dotted] (\y-1,0) arc (180:0:1cm and 0.5cm);
        \shade[ball color=blue!20!,opacity=0.3] (\y,0) circle (1cm);
        \draw[blue,semithick] (\y,0) circle (1cm);
    }
    % vertical arcs
    \foreach \x in {6,9,11}{
        \draw[blue]
            (\x,1) to[out=-5,in=5, distance=0.7cm] (\x,-1);
        \draw[blue,dotted]
            (\x,1) to[out=-145,in=145, distance=0.6cm] (\x,-1);
    }
    % \foreach \y in {0,2}{
    %     \fill[mygreen,draw=black,semithick] (\y,0) circle (1.5pt);
    % }
    \fill[red,draw=black]
        (-1,0) circle (1.5pt) (-0.78,0) node[below]{\small$\mathsf{u}({-}1)$}
        (11,0) circle (1.5pt) node[below]{\small$\mathsf{u}(1)$};

    % yellow disk
    \fill[yellow!50!,opacity=0.8]
        (5.6,1.8) to[out=210,in=180, distance=1.3cm] (5.95,-1.32) to[in=-95,out=111, distance=1cm] (5.6,1.8) ;
    \draw[opacity=0.8,semithick]
        (5.6,1.8) to[out=210,in=180, distance=1.3cm] (5.95,-1.32)  ;

    % the yellow disk boundary
    \draw[black!70, thick]
                    (5.6,1.8) to[out=-95,in=-100, distance=1.15cm] (6.1,-1.06)
                    (6.2,1.06) to[out=40,in=-120, distance=0.07cm] (6.1,2) ;
    % the third from right
    \draw[,  thick]
            (5.6,1.8) to[out=-80,in=-80, distance=1.4cm] (6.8,-0.8)
            (6.5,0.95) to[out=100,in=-80, distance=0.1cm] (6.1,2) ;
    % the second from right
    \draw[,  thick]
           (5.6,1.8) to[out=-65,in=-50, distance=3.6cm] (6.1,2);
    % the rightmost
    \draw[,  thick]
            (5.6,1.8) to[out=-40,in=-40, distance=2.2cm] (6.1,2);
    
    \draw[mygreen, very thick]
        (2,0) -- (3,0)  node[pos=0.5,below]{\small$J_0$};
\end{tikzpicture}
    \caption[The dragging reflection.]{The dragging reflection.}
    \label{fig:dragging-reflection}
\end{figure}
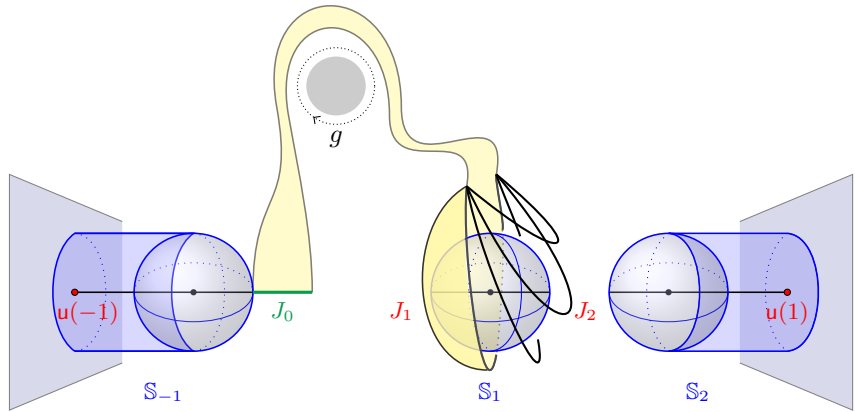

\begin{rem}
    Some of statements were proven in \cite{K-thesis-paper} only for $d=3$, since that dimension was the main setting of that paper. However, they concern the delooping map $\chi$, which was constructed there even in a greater generality, for any cubical diagram which has a left homotopy inverse cubical diagram. Thus, it is immediate that the analogous results hold for all $d\geq3$.
\end{rem}

%~~~~~~~~~~
\subsection{Proof of the main theorem for the arc}\label{subsec:main-proof-arc}

We will first prove that the composite map $W(\chi\circ\deriv)_* \ev_{n+1} \realmap_n$ is the canonical quotient map $\Z[\Tree_{\pi_1M}(n)]\sra\Lie_{\pi_1M}(n)$.
Since all maps are linear, it suffices to show that for any $\Gamma^{g_{\ul{n}}}\in\Tree_{\pi_1M}(n)$ and the family $\u^\star_{\TG,\Gamma^{g_{\ul{n}}}}$ from \eqref{eq:grasper-gives-path} we have
\[
    W(\chi\circ\deriv)_*([\u^\star_{\Gamma,\TG}])=W([\chi\deriv\u^\star_{\TG,\Gamma^{g_{\ul{n}}}}])=[\Gamma^{g_{\ul{n}}}].
\]
Firstly, by the paragraph after Theorem~\ref{thm:thesis} the class $[\chi\deriv\u^\star_{\Gamma,\TG}]\in\pi_{n(d-2)}\tofib\Omega M_{\bull}$ can be represented by a point
    \begin{equation}\label{eq:final-pt}
        (\chi\deriv\u^\star_{\TG,\Gamma^{g_{\ul{n}}}})^{\ul{n}}\in\Omega^{n(d-2)}
        \Omega M_{\ul{n}}.
    \end{equation}
By \eqref{eq:drag-refl} this is in turn represented by a map $\S^n\to\Omega^{n(d-3)}\Omega M_{\ul{n}}$ which is  obtained  by $\glueOp$-gluing the ``dragging reflections'' $\refl_S^{drag}$ of the map
\begin{equation}\label{eq:final-pt-cube}
\begin{tikzcd}
     \I^{\ul{n}}\rar{\u_{\TG,\Gamma^{g_{\ul{n}}}}^\star} &
     \Omega^{n(d-3)}\Embp(J_0,M_{0\ul{n+1}})\rar{\deriv} &
     \Omega^{n(d-3)}\Omega M_{\ul{n+1}}.
\end{tikzcd}
\end{equation}
On the other hand, the inverse of the isomorphism $W$ takes the generating decorated tree $\Gamma^{g_{\ul{n}}}$ to the class obtained by canonically extending to $\tofib\Omega M_{\bull}$ the map
\begin{equation}\label{eq:final-gens}
\begin{tikzcd}[column sep = large]
    \S^{n(d-2)}\rar{x_{\Gamma,d-1}} &
    \Omega\bigvee_n\S^{d-1}\rar[hook]{\Omega \S(g_{\ul{n}})} &
    \Omega M_{\ul{n+1}},
\end{tikzcd}
\end{equation}
where $x_{\Gamma,d-1}$ is the Samelson product from \eqref{eq-def:x-Gamma} and $\S(g_{\ul{n}})\colon\bigvee_n\S^{d-1}\hra M_{\ul{n}}$ is the inclusion of the spheres $\S_i\coloneqq\partial\ball_{ii+1}$ between punctures $J_i$ and $J_{i+1}$ (see Figure~\ref{fig:nbhd-of-U-S}), together with the whisker obtained by conjugating the path $\u|_{[L_0,R_i]}$ by $g_i\in\pi_1M$. 
Moreover, by Proposition~\ref{prop:suspensions} the map $x_{\Gamma,d-1}$ is obtained by $\glueOp$-gluing the reflections of the cube $x^\star_{\Gamma,d-1}$. Thus, \eqref{eq:final-gens} is obtained by $\glueOp$-gluing the reflections $\refl_S$ of
\begin{equation}\label{eq:final-gens-cube}
\begin{tikzcd}[column sep = large]
    \I^{\ul{n}}\rar{x^\star_{\Gamma,d-1}} &
    \Omega^{n(d-3)}\Omega\bigvee_n\ball^{d-1}\arrow{rr}[hook]{\Omega^{n(d-3)}\Omega \S(g_{\ul{n}})} &&
    \Omega^{n(d-3)}\Omega M_{\ul{n+1}}.
\end{tikzcd}
\end{equation}
Therefore, to prove that \eqref{eq:final-pt} and \eqref{eq:final-gens} are homotopic, it suffices to prove that the maps \eqref{eq:final-pt-cube} and \eqref{eq:final-gens-cube} are homotopic, as well as their respective reflections, and that homotopies $\glueOp$-glue together.

That these two maps are homotopic will follow from Theorem~\ref{thm:analogues}, which gives a homotopy between $\mu^\star_{\Gamma,d}$ and $x^\star_{\Gamma,d-1}$ when both are included into $\Omega\ball^d_n$. Indeed, on one hand, since $\deriv$ includes embeddings into all maps and closes them into loops using $\u|_{J_0}^{-1}$, and since $\u^\star_{\Gamma,\TG}=(\TG_n\circ\mu^\star_{\Gamma,d})|_{\S^1_-}$ by definition in Theorem~\ref{thm:grasper-gives-path}, we conclude that the desired map $\deriv(\u_{\TG,\Gamma^{g_{\ul{n}}}}^\star)$ from \eqref{eq:final-pt-cube} is homotopic to the composition of the upper maps in the diagram:
\[
\begin{tikzcd}[column sep = large,row sep=small]
    & \Emb((\S^1,\S^1_+),(\ball^d_n,a_0))\arrow[hook]{dr} && \\ 
    \I^{\ul{n}}\times\S^{n(d-3)} \arrow{dr}[swap]{x^\star_{\Gamma,d-1}}\arrow{ur}{\mu^\star_{\Gamma,d}} && \Omega\ball^d_n\rar{\Omega\TG_n}\rar[swap]{} &
    \Omega M_{\ul{n}}. \\
    & \Omega\bigvee_n\ball^{d-1}\arrow[hook]{ur}\arrow[bend right=10]{urr}[swap]{\Omega \S(g_{\ul{n}})} &&
\end{tikzcd}
\]
On the other hand, the map~\eqref{eq:final-gens-cube} is the composite of the maps at the bottom. But now, the triangle on the right commutes up to homotopy since the underlying decoration of $\TG_n$ is precisely $g_{\ul{n}}$, and the square on the left commutes up to homotopy by Theorem~\ref{thm:analogues}.

Furthermore, the reflections of this homotopy form homotopies between the respective reflections. Indeed, when reflecting \eqref{eq:final-pt-cube} we also drag the meridian ball $\ball_i$ (which is the eastern hemisphere of the sphere $\S_i$), so that it becomes the western hemisphere, as discussed in Figure~\ref{fig:dragging-reflection}.

Therefore, we have shown that the composite map $W(\chi\circ\deriv)_* \ev_{n+1} \realmap_n$ is the canonical quotient map $\Z[\Tree_{\pi_1M}(n)]\sra\Lie_{\pi_1M}(n)$. Note that this shows that both $\evrel_{n+1}$ and $\realmap_n$ are surjective.

Finally, to see that $\realmap_n$ vanishes on $(\AS)$ and $(\IHX)$ relations, we simply use that $\evrel_{n+1}$ is an isomorphism by the Convergence Theorem~\ref{thm:rpn}. Alternatively, see Remark~\ref{rem:graded-signs}.
\qed

%~~~~~~~~~~
\subsection{Proofs of the main theorem for the circle, and of the corollaries}\label{subsec:main-proof-circle}\label{subsec:proofs-cor}

For the embeddings of $\S^1$ we have the following fact which follows quickly from definitions; see for example \cite{KT-rpn}.
\begin{thm}
\label{thm:rpn-circle}
    There is a commutative diagram in which rows are fibration sequences
    \begin{equation}\label{eq:circle-to-arc}
    \begin{tikzcd}
        \Embp(\D^1,M\sm\D^d)\rar\dar & \Emb(\S^1,M)\rar\dar & \S M\arrow[equals]{d}
        \\
        T_n(\D^1,M\sm\D^d)\rar & T_n(\S^1,M)\rar & \S M.
    \end{tikzcd}
    \end{equation}
\end{thm}

    It follows that there is a weak homotopy equivalence 
    \[
        \pF_{n+1}(\D^1,M\sm\D^d)\simeq\pF_{n+1}(\S^1,M).
    \]
    Therefore, the lowest nonvanishing homotopy group is again
    \[
        \pi_{n(d-3)}\pF_{n+1}(\S^1,M)\cong\Lie_{\pi_1M}(n)
    \]
    and we thus have a diagram
    \[\begin{tikzcd}[column sep=4pt]
        & \Lie_{\pi_1M}(n)
        \arrow[shift right]{dl}[swap]{\realmap_n^{\D^1}}
        \arrow[shift left]{dr}{\realmap_n^{\S^1}} 
        & \\
        \pi_{n(d-3)+1}\big(T_n(\D^1,M\sm\D^d),\Emb(\D^1,M\sm\D^d)\big)
        \arrow{rr}{\incl}\arrow[shift right]{ur}[swap,near end]{\ev^{\mathsf{rel},\D^1}_{n+1}} 
        && 
        \pi_{n(d-3)+1}\big(T_n(\S^1,M),\Emb(\S^1,M)\big)
        \arrow[shift left]{ul}[near end]{\ev^{\mathsf{rel},\S^1}_{n+1}}
    \end{tikzcd}
    \]
    The $\realmap$-triangle commutes by construction and Remark~\ref{rem:arc-to-circle}, and the $\ev$-triangle commutes by \eqref{eq:circle-to-arc}.
    Therefore, using that $\ev^{\mathsf{rel},\D^1}_{n+1}\circ\realmap_n^{\D^1}=\Id$ by Section~\ref{subsec:main-proof-arc}, we have 
    \[
    \ev^{\mathsf{rel},\S^1}_{n+1}
    \circ\realmap_n^{\S^1}(\Gamma^{g_{\ul{n}}})=
    \ev^{\mathsf{rel},\S^1}_{n+1}
    \circ\incl\circ\realmap_n^{\D^1}(\Gamma^{g_{\ul{n}}})=
    \ev^{\mathsf{rel},\D^1}_{n+1}
    \circ\realmap_n^{\D^1}(\Gamma^{g_{\ul{n}}})
    =\Gamma^{g_{\ul{n}}}.
    % \tag*{\qed}
    \]
This finishes the proof of Theorem~\ref{thm:main}.\qed

%~~~~~~~~~~
% \subsection{Proofs of corollaries}

\begin{proof}[Proof of Corollary~\ref{cor:main}]
    We have a commutative diagram with exact rows
\[\begin{tikzcd}
    \ker\ev_n\arrow[tail]{r}\arrow[two heads,tail]{d}{\cong}[swap]{\evrel_{n+1}} & \pi_{n(d-3)}\Embp(C,M)\rar[two heads]{\ev_n}\arrow[two heads,tail]{d}{\cong}[swap]{\ev_{n+1}} & \pi_{n(d-3)}T_n(C,M)\arrow[equals]{d}\\
    \faktor{\Lie_{\pi_1M}(n)}{\im\delta_n}\arrow[tail]{r} & \pi_{n(d-3)}T_{n+1}(C,M)\rar{\p_{n+1}} & \pi_{n(d-3)}T_n(C,M).
\end{tikzcd}
\]
    The vertical arrow in the middle is an isomorphism by the Goodwillie--Klein--Weiss convergence theorem, recalled as Theorem~\ref{thm:rpn}, and the map 
    \[
        \delta_n\colon\pi_{n(d-3)+1}T_n(C,M)\to\pi_{n(d-3)+1}(T_{n+1}(C,M),T_n(C,M))\cong\Lie_{\pi_1M}(n)
    \]
    is the composition of the canonical map in the long exact sequence of the pair with the isomorphism $W(\chi\circ\deriv)_*$ from Theorem~\ref{thm:thesis}. Recall from Section~\ref{subsec:trees} that $\Lie_{\pi_1M}(n)\coloneqq\Z[\Tree_{\pi_1M}(n)]/\AS,\IHX$.
    
    On the other hand, by Theorem~\ref{thm:main} we have a commutative triangle
    \[\begin{tikzcd}
        & \ker\ev_n\arrow[two heads,tail]{d}{\evrel_{n+1}}\\
        \Z[\Tree_{\pi_1M}(n)]\arrow{ur}{\realmap_n}\arrow[two heads]{r}  & \faktor{\Z[\Tree_{\pi_1M}(n)]}{\AS,\IHX,\im\delta_n}
    \end{tikzcd}
    \]
    where the horizontal arrow is the quotient map. Therefore, for a linear combination $F$ of $\pi_1M$-decorated trees, the class $\realmap_n(F)\in\pi_{n(d-3)}\Embp(C,M)$ is trivial if and only if $F$ is in the linear span of $(\AS)$ and $(\IHX)$ relations and $\im \delta_n$.
\end{proof}

\begin{proof}[Proof of Corollary~\ref{cor:Dd}]
    This is immediate from the the collapse of the spectral sequence (so that $\im(\delta_n)=\im(d_1)=\stusq_{\mathsf{odd/even}}$) explained in Section~\ref{subsec:related-work}.
\end{proof}

% Referencee
\printbibliography[heading=bibintoc]

\vspace{10pt}
\hrule

\end{document}